\documentclass[12pt]{amsart}
\usepackage{amscd, amsfonts, amssymb,mathrsfs}
\usepackage{fullpage}
\usepackage{latexsym}
\usepackage{verbatim}
\usepackage{enumerate}
\usepackage[colorlinks=true,citecolor=red,linkcolor=blue]{hyperref}
\usepackage{tikz}
\usetikzlibrary{arrows,matrix,decorations,positioning, backgrounds,fit}

\usepackage[all,cmtip]{xy}

\usepackage{graphicx}

\newcommand\fg{\mathfrak g}

\newcommand\inverse{{^{-1}}}

\newcommand\opt{\mathrm {opt}}

\newcommand{\iso}{\cong}
\newcommand{\NN}{{\mathbb N}}
\newcommand{\ZZ}{{\mathbb Z}}
\newcommand{\QQ}{{\mathbb Q}}
\newcommand{\RR}{{\mathbb R}}
\newcommand{\KK}{{\mathbb K}}

\newcommand{\Gm}{\mathbb{G}_m}

\DeclareMathOperator{\Aut}{Aut}
\DeclareMathOperator{\Inn}{Inn}
\DeclareMathOperator{\Isom}{Isom}

\DeclareMathOperator{\GL}{GL}
\DeclareMathOperator{\Gal}{Gal}
\DeclareMathOperator{\SL}{SL}

\DeclareMathOperator{\Lie}{Lie}

\DeclareMathOperator{\proj}{proj}
\DeclareMathOperator{\IM}{Im}

\DeclareMathOperator{\sgn}{sgn}

\DeclareMathOperator{\red}{red}
\DeclareMathOperator{\supp}{supp}

\numberwithin{equation}{section}

\newtheorem{thm}[equation]{Theorem}

\newtheorem{lem}[equation]{Lemma}
\newtheorem{cor}[equation]{Corollary}
\newtheorem{prop}[equation]{Proposition}
\newtheorem{conj}[equation]{Conjecture}

\theoremstyle{definition}
\newtheorem{defn}[equation]{Definition}
\newtheorem{exmp}[equation]{Example}
\newtheorem{exmp/rem}[equation]{Example/Remark}
\theoremstyle{remark}
\newtheorem{rem}[equation]{Remark}
\newtheorem{rems}[equation]{Remarks}

\newcommand{\ovl}{\overline}


\date{May 19, 2023.}

\title[Edifices]{Edifices: Building-like spaces associated to linear algebraic groups}

\dedicatory{In memory of Jacques Tits}

\author[M.\  Bate]{Michael Bate}
\address
{Department of Mathematics,
University of York,
York YO10 5DD,
United Kingdom}
\email{michael.bate@york.ac.uk}

\author[B.\ Martin]{Benjamin Martin}
\address
{Department of Mathematics,
University of Aberdeen,
King's College,
Fraser Noble Building,
Aberdeen AB24 3UE,
United Kingdom}
\email{b.martin@abdn.ac.uk}

\author[G. R\"ohrle]{Gerhard R\"ohrle}
\address
{Fakult\"at f\"ur Mathematik,
Ruhr-Universit\"at Bochum,
D-44780 Bochum, Germany}
\email{gerhard.roehrle@rub.de}

\begin{document}

\begin{abstract}
 Given a semisimple linear algebraic $k$-group $G$, one has a spherical building $\Delta_G$, and one can interpret the geometric realisation $\Delta_G(\RR)$ of $\Delta_G$ in terms of cocharacters of $G$.  The aim of this paper is to extend this construction to the case when $G$ is an arbitrary connected linear algebraic group; we call the resulting object $\Delta_G(\RR)$ the \emph{spherical edifice} of $G$.  We also define an object $V_G(\RR)$ which is an analogue of the vector building for a semisimple group; we call $V_G(\RR)$ the \emph{vector edifice}.  The notions of a linear map and an isomorphism between edifices are introduced; we construct some linear maps arising from natural group-theoretic operations.  We also devise a family of metrics on $V_G(\RR)$ and show they are all bi-Lipschitz equivalent to each other; with this extra structure, $V_G(\RR)$ becomes a complete metric space.  Finally, we present some motivation in terms of geometric invariant theory and variations on the Tits Centre Conjecture.
\end{abstract}

\maketitle

\setcounter{tocdepth}{1}
\tableofcontents


\section{Introduction}
\label{sec:intro}

Jacques Tits showed in his seminal 1974 monograph \cite{tits1} how to associate a spherical building $\Delta_G$ to an isotropic semisimple linear algebraic group $G$ over a field $k$: the simplices of the building correspond to the $k$-parabolic subgroups of $G$, ordered by reverse inclusion.  This has led to a rich and fruitful interaction between the theory of algebraic groups and the theory of buildings; for example, Serre has formulated his theory of $G$-completely reducible subgroups of $G$ in building-theoretic language \cite{serre1.5}, \cite{serre2}.  Tits's celebrated Centre Conjecture for spherical buildings---whose proof is the culmination of a series of papers involving several authors \cite{muhlherrtits}, \cite{lrc}, \cite{rc}---has consequences for algebraic groups: for instance, it implies the Borel-Tits Theorem \cite[\S 3]{boreltits2}.  The authors in previous and current work have studied the interactions among the Centre Conjecture, $G$-complete reducibility and geometric invariant theory (GIT), e.g., see \cite{sphericalcochars}, \cite{BMR}, \cite{BMR:tits}, \cite{BMR:strong}, \cite[\S 5]{BMR:semisimplification}, \cite[\S 5.4]{GIT}, and \cite{BMR:typeA}.  The link with geometric invariant theory comes via a rational version of the Hilbert-Mumford Theorem, used to study orbits of an action of $G$ on an affine variety, see~\cite{cochars}, \cite[\S 3]{GIT}.

The geometric realisation $\Delta_G(\RR)$ of $\Delta_G$ has a striking interpretation in terms of cocharacters of $G$.  Given a cocharacter $\lambda$ of $G$, one can associate to it a parabolic subgroup $P_\lambda$ of $G$ in a standard way, see ~\cite[\S 2]{rich2}, \cite[\S 8.4]{spr2}.  Roughly speaking, points in $\Delta_G(\RR)$ correspond to formal $\RR$-linear combinations of cocharacters, modulo a natural equivalence relation.  If $x\in \Delta_G(\RR)$ is the point corresponding to $\lambda$ then $x$ belongs to the geometric realisation of the simplex corresponding to $P_\lambda$.  This construction is based on ideas of Tits and Chevalley and is described in 
\cite{curtislehrertits}, \cite{BMR:typeA} and below.  See \cite[Ch.\ 2, Sec.\ 2]{mumford} for further discussion (note that $\Delta_G(\QQ)$ is called the \emph{rational flag complex} in \emph{loc.\ cit.}).  For applications to geometric invariant theory, one must work with this geometric realisation rather than the simplicial building, since in the Hilbert-Mumford Theorem one is concerned with the behaviour of individual cocharacters; see Section~\ref{sec:TCCGIT}.

Our purpose in this paper is to generalise the construction of a building via cocharacters to the case when $G$ is an arbitrary connected linear algebraic $k$-group, not necessarily semisimple.  This was achieved for reductive groups in \cite{curtislehrertits};  in that case one can still form a spherical building from the $k$-parabolic subgroups of $G$, which is canonically isomorphic to the spherical building of the semisimple group $G/Z(G)^0$, but one sees a difference at the level of cocharacters as $G$ may admit non-trivial cocharacters with image in $Z(G)^0$.  We define spaces $\Delta_G(\RR)$ and $\Delta_G$ for arbitrary $G$ in Section~\ref{sec:edifices}, which coincide with the spaces described above when $G$ is semisimple or reductive.  We call $\Delta_G(\RR)$ the \emph{spherical edifice} associated to $G$ and $\Delta_G$ the \emph{combinatorial edifice} associated to $G$.  We also define a space $V_G(\RR)$, which we call the \emph{vector edifice} associated to $G$; this is our generalisation for arbitrary $G$ of the vector building for a semisimple group defined in \cite{rousseau}.  In fact, we concentrate on vector edifices rather than spherical edifices as the former have some technical advantages (see Remark~\ref{rem:spherical_deficient}).

In general $\Delta_G(\RR)$ is not the geometric realisation of $\Delta_G$ (or of any spherical building); nonetheless it enjoys some of the same properties.  For instance, $G(k)$ acts on $\Delta_G(\RR)$ and we may endow $\Delta_G(\RR)$ with a $G(k)$-invariant complete metric: see Section~\ref{sec:admmetric}.  We prove the following key property (Proposition~\ref{prop:gphom}).  Let $f\colon G\to H$ be a homomorphism of $k$-groups.  Then $f$ gives rise in a natural way to a map $\kappa_f\colon V_G(\RR)\to V_H(\RR)$, and $\kappa_f$ is continuous with respect to the metric topologies (Corollary~\ref{cor:lincts}).  For example, if $G$ is reductive and $P$ is a parabolic subgroup of $G$ then $\kappa_i\colon V_P(\RR)\to V_G(\RR)$ is a bijection, where $i\colon P\to G$ is inclusion (see Example~\ref{ex:bijectivenotiso}).  We also obtain an induced map from $\Delta_G(\RR)$ to $\Delta_H(\RR)$.  Even when $G$ and $H$ are reductive, $f$ need not map parabolic subgroups of $G$ to parabolic subgroups of $H$, so $f$ does not give rise to a map $\Delta_G\to \Delta_H$ of simplicial buildings (see Remark~\ref{rem:not_combnl}).  Our constructions therefore provide extra flexibility: maps between the geometric objects exist even when maps between the corresponding combinatorial objects do not.  We rely on this heavily in our forthcoming work \cite{BMR:typeA}.

While edifices are worth studying as objects in their own right, we have two further motivations for investigating them. First, assume $G$ is reductive.  Let $P$ be a parabolic subgroup of $G$ and let $L$ be a Levi subgroup of $P$.  For reasons discussed in Section~\ref{sec:TCCGIT}, we wish to construct a map with certain properties from $V_G(\RR)$ to $V_L(\RR)$.  It is well known that such a map exists at the level of the simplicial building: this is the so-called projection map from $\Delta_G$ to $\Delta_L$, see Section \ref{sec:projection}.  We need a map on the geometric realisations.  Our constructions give us what we want: we define the map to be $\kappa_\pi\circ \kappa_i^{-1}$, where $i$ is the inclusion of $P$ in $G$ from above and $\pi\colon P\to L$ is the canonical projection.

Second, edifices give us a framework for studying $G$-complete reducibility and geometric invariant theory for non-reductive $G$: they play the role that the geometric realisation of the spherical building plays when $G$ is semisimple.  We mention here in particular the case of pseudo-reductive $G$.  
As explained in \cite{prasad}, one can associate a spherical Tits system to $G$ \cite[Thm.\ 4.1.7]{CP}, and one can then construct a spherical building $\Delta_G$ using the pseudo-parabolic subgroups of $G$;
however, there does not seem to be an analogue of the ``geometric realisation by cocharacters'' previously known here.
Moreover, when working with pseudo-reductive $G$, one often needs to extend scalars to a purely inseparable field extension $k'/k$; if the base change $G_{k'}$ has non-trivial unipotent radical over $k'$ then $\Delta_G$ need no longer be a spherical building, but the formalism of edifices is flexible enough to handle this situation.  The authors and their collaborators will pursue these ideas in forthcoming work.

The paper is set out as follows.  In Section~\ref{sec:prelims} we review some background material on algebraic groups and cocharacters, and we define the key notion of an R-parabolic subgroup.  We define the vector edifice $V_G(\KK)$, the spherical edifice $\Delta_G(\KK)$ and the combinatorial edifice $\Delta_G$ in Section~\ref{sec:edifices}, where $\KK= \RR$ or $\QQ$, and introduce some related notions: apartments (certain subsets of the vector edifice), opposites, addition, and convexity.  One substantial difference from the reductive group case is that the ``common apartment property'' can fail: not every pair of points in $V_G(\KK)$ need belong to a common apartment.  In Section~\ref{sec:linmapsI} we introduce the notions of linear maps and isomorphisms between vector edifices.  We show that an algebraic group homomorphism $f\colon G\to H$ gives rise to a linear map $\kappa_f\colon V_G(\KK)\to V_H(\KK)$, and that $\kappa_f$ is injective if $\ker(f)$ is finite.  In Section~\ref{sec:fieldexts} we construct linear maps of vector edifices arising from base change of the ground field $k$.

We turn to metrics in Section~\ref{sec:admmetric}.  We define the notion of an \emph{admissible metric} on $V_G(\KK)$.  The idea is to do this first for reductive groups, where we can work inside individual apartments; for an arbitrary $G$, we can embed $G$ inside some reductive $G'$, which gives an embedding of $V_G(\KK)$ inside $V_{G'}(\KK)$, then pull back a metric from $V_{G'}(\KK)$ to $V_G(\KK)$.  Admissible metrics on $V_G(\KK)$ are not unique, but we prove that any two are bi-Lipschitz equivalent (Corollary~\ref{cor:quasi-isom}); this takes some work.  In the final section (Section~\ref{sec:TCCGIT}), we discuss applications of our formalism to the Tits Centre Conjecture and to geometric invariant theory.

\section{Preliminaries}
\label{sec:prelims}

Throughout $k$ denotes a field with separable closure $k_s$ and algebraic closure $\ovl{k}$.
We work with affine (linear) algebraic $k$-groups.
Formally, we view such a group $G$ as a functor from $k$-algebras to groups, 
represented by an affine $k$-algebra $k[G]$, the coordinate algebra of $G$. 
If $X$ is a $k$-variety and $k'/k$ is a field extension, then $X_{k'}$ denotes the $k'$-variety given by the base change 
of $X$ to $k'$, with coordinate algebra $k'[X_{k'}]:= k[X]\otimes_k k'$; if $f\colon X\to Y$ is a morphism of $k$-varieties then $f_{k'}\colon X_{k'}\to Y_{k'}$ is the morphism obtained by base extension.
Unless specified otherwise, we use the phrase ``let $G$ be a $k$-group'' as shorthand for ``let $G$ be a smooth affine algebraic $k$-group''.  We write $\fg$ for the Lie algebra of $G$.  For a $k$-group $G$ and a $k$-algebra $A$, we let $G(A)$ denote the group of $A$-points of $G$.
We adopt the framework of \cite{CGP}: so by a subgroup of $G$ we mean a $k$-defined subgroup, by a homomorphism we mean a $k$-homomorphism and by a cocharacter we mean a $k$-defined cocharacter.  We assume, however, that subgroups are smooth; in cases when smoothness is not evident, we use the terminology subgroup scheme.  By the kernel of a homomorphism we mean the scheme-theoretic kernel (note that this need not be smooth).  Likewise, by the centraliser $C_G(H)$ of a subgroup $H$ we mean the scheme-theoretic centraliser and by the intersection of subgroups we mean the scheme-theoretic intersection.  Note that $C_G(H)$ is smooth if $H$ is a torus.  By an embedding of varieties we mean a closed embedding.  For background on linear algebraic groups, see \cite{CGP}, \cite{Bo}, \cite{spr2}.
We record one crucial result here \cite[Thm.\ C.2.3]{CGP}, which is used frequently in the sequel without further comment:
the maximal split $k$-tori in a $k$-group $G$ are all conjugate under the action of $G(k)$.  By the rank of $G$ we mean the $k$-rank: that is, the dimension of a maximal split torus of $G$.  If $T$ is a maximal split torus of $G$ then $W_k: = (N_G(T)/C_G(T))(k)$ is the \emph{relative Weyl group} of $G$.  If $g\in G(k)$ then $\Inn_g\colon G\to G$ is conjugation by $g$.

For an algebraic $k$-scheme $X$ we let $X_{\red}$ denote the unique reduced subscheme of $X$ with the same underlying topological space as $X$ \cite[A.30]{milne}. 
When $k$ is perfect $(X_{\red})_{\bar{k}} = (X_{\bar{k}})_{\red}$ is reduced, and so when $X$ is a group scheme over a perfect field $k$, $X_{\red}$ is smooth
\cite[Prop.~1.26, Cor.~1.39]{milne}.

Our convention in this paper is that reductive groups are connected.  If $G$ is a $k$-group then we write $R_{u,k}(G)$ for the $k$-unipotent radical of $G$ (the largest connected normal unipotent $k$-subgroup).  If $R_u(G_{\ovl{k}})$ descends to a subgroup of $G$ then we write $R_u(G)$ for $R_{u,k}(G)$.  We denote by $Y_G$ the set of cocharacters (one-parameter subgroups) of $G$ --- that is, an element $\lambda\in Y_G$ is a homomorphism $\lambda$ from the multiplicative group $\Gm$ to $G$.  If $H$ is a subgroup of $G$ then we may regard $Y_H$ as a subset of $Y_G$ via the inclusion $H\subseteq G$.  We say that $\lambda\in Y_G$ \emph{evaluates in $H$} if $\lambda$ arises from an element of $Y_H$ in this way.  We say that two cocharacters \emph{commute} if their images commute.
Given $\lambda\in Y_H$ and $g\in G(k)$, we define $g\cdot \lambda\in Y_{gHg^{-1}}$ by $(g\cdot \lambda)(a)= g\lambda(a)g^{-1}$ for any $k$-algebra $A$ and any $a\in {\mathbb G}_m(A) = A^\times$; this gives an action of $G(k)$ on $Y_G$.

Now suppose $T$ is a split torus, and let $X_T$ denote the group of characters of $T$: that is, $\chi\in X_T$ is a homomorphism $T\to \Gm$.
There is the usual pairing between $Y_T$ and $X_T$: given $\lambda\in Y_T$ and $\chi\in X_T$, the map $\chi\circ\lambda\colon \Gm\to \Gm$ is an endomorphism of $\Gm$, and hence 
is given by raising elements to a power; we denote this power by $\langle \lambda,\chi \rangle$.
If now $T$ is a subtorus of some $k$-group $G$, then for $g\in G(k)$ and $\chi\in X_T$, we get a character $g\cdot\chi \in X_{gTg^{-1}}$ via the formula $(g\cdot\chi)(gtg^{-1}) = \chi(t)$ for any $k$-algebra $A$ and $t\in T(A)$.
It is easy to check that $\langle g\cdot\lambda,g\cdot\chi \rangle =  \langle \lambda,\chi \rangle$ for all $g\in G(k)$.

Throughout, the symbol $\KK$ denotes $\RR$ or $\QQ$.  We write $\KK^+$ for $\KK\cap (0,\infty)$.  If $T$ is a torus then $Y_T$ with the usual addition of cocharacters is a free $\ZZ$-module of finite rank. 
We define $Y_T(\KK):= Y_T\otimes_\ZZ \KK$, a finite-dimensional $\KK$-vector space, and identify $Y_T$ with the subset of $Y_T(\KK)$ consisting of elements of the form $\lambda\otimes 1$ for $\lambda\in Y_T$ in the usual way.  We give $Y_T(\KK)$ the usual topology arising from a positive-definite bilinear form (this topology does not depend on the form). 
We may similarly define $X_T(\KK):= X_T\otimes_\ZZ \KK$,
and then the pairing $Y_T \times X_T \to \ZZ$ extends to give a $\KK$-bilinear map $Y_T(\KK)\times X_T(\KK)\to \KK$.

Below $G$ denotes a connected $k$-group.  At various points we make the extra assumption that $G$ is reductive.
For $f\colon G\to H$ a homomorphism of $k$-groups, we say that $f$ is surjective provided $f(\ovl{k})\colon G(\ovl{k})\to H(\ovl{k})$ is surjective.  An \emph{isogeny} is a surjective homomorphism of $k$-groups with finite kernel.

\subsection{Limits and R-parabolic subgroups}
For full details on the following construction, the reader is referred to \cite[Sec.\ 2.1]{CGP}.
Let $V$ be an affine $k$-scheme equipped with an action of $\Gm$.
Given $v\in V(A)$ for some $k$-algebra $A$, we say that $\lim_{a\to 0} a\cdot v$ \emph{exists} if the
orbit map $(\Gm)_A\to V_A, a\mapsto a\cdot v$ over $A$ extends to an $A$-morphism $F : \mathbb{A}^1_A\to V_A$, and in this case we write 
$\lim_{a\to0} a\cdot v$ for the limit $F(0)$ (which is necessarily uniquely defined). 
In \cite[Lem.~2.1.4]{CGP}, it is shown that there is a closed subscheme $V'$ of $V$ such that for any $k$-algebra $A$, $V'(A)$
consists of the $v\in V(A)$ such that the limit $\lim_{a\to 0} a\cdot v$ exists.  If $f\colon V\to W$ is an equivariant map of $\Gm$-varieties, $v\in V(A)$ and $\lim_{a\to 0} a\cdot v$ exists then $\lim_{a\to 0} a\cdot f(v)$ exists, and the converse holds is $f$ is an embedding.

The construction in the previous paragraph is of central importance in this paper when $V=G$ is a $k$-group, 
and the action of $\Gm$ on $G$ is the conjugation action via a cocharacter $\lambda\colon \Gm\to G$.
In this case, for each $\lambda\in Y_G$, one obtains a closed subgroup $P_\lambda(G)$ of $G$: for each $k$-algebra $A$,
$$
P_\lambda(G)(A):=\left\{g\in G(A) \ \left|\  \lim_{a\to 0} \lambda(a)g\lambda(a)^{-1} \textrm{ exists}\right\}\right.
$$
(that $P_\lambda(G)$ is a smooth $k$-subgroup scheme is in \cite[Lem.~2.1.5]{CGP}).
We also write $\lim_{a\to 0} \lambda(a)\cdot g$ as a shorthand for the limit.

We call the subgroups of $G$ of the form $P_\lambda(G)$ for $\lambda\in Y_G$ the \emph{R-parabolic subgroups} of $G$.\footnote{The ``R'' here is in honour of Roger Richardson, who made extensive use of the link between cocharacters and parabolic subgroups in reductive groups, in the context of geometric invariant theory (see \cite{rich2}).}
We usually suppress the $G$ in $P_\lambda(G)$ unless there is some ambiguity --- this occurs, for instance, when $H$ is a subgroup of $G$ and 
we wish to view $\lambda\in Y_H$ as a cocharacter of both $G$ and $H$, with two corresponding subgroups $P_\lambda(H)\subseteq P_\lambda(G)$.  
In case $G$ is pseudo-reductive the subgroups $P_\lambda(G)$ are the so-called pseudo-parabolic subgroups of $G$, but this is not the case in general: a pseudo-parabolic subgroup of a general connected $k$-group $G$ must also contain $R_{u,k}(G)$ but $P_\lambda$ need not; see \cite[Sec.~2.2]{CGP}.

We define $L_\lambda(G)$ (or just $L_\lambda$) to be $C_G({\rm Im}(\lambda)) = P_\lambda\cap P_{-\lambda}$ 
and we define a subgroup $U_\lambda = U_\lambda(G)$ functorially by
$$
U_\lambda(G)(A) := \left\{g\in G(A)\ \left|\ \lim_{a\to0} \lambda(a)g\lambda(a)^{-1} = 1\right\}\right. ,
$$
for each $k$-algebra $A$.
It follows from \cite[Lem.\ 2.1.5, Prop.\ 2.1.8]{CGP} that $P_\lambda$, $L_\lambda$ and $U_\lambda$ are smooth connected subgroup schemes of $G$, $U_\lambda$ is a normal unipotent subgroup of $P_\lambda$, and $P_\lambda = L_\lambda\ltimes U_\lambda$.
If $P$ is an R-parabolic subgroup of $G$ then we call any subgroup $L_\lambda$ for $\lambda\in Y_G$ such that $P= P_\lambda$ an \emph{R-Levi subgroup} of $P$.  Given $\lambda\in Y_G$, we have a surjective homomorphism $c_\lambda\colon P_\lambda\to L_\lambda$ given by $c_\lambda(g)= \lim_{a\to 0} \lambda(a)g\lambda(a)^{-1}$ for each $k$-algebra $A$ and each $g\in P_\lambda(A)$. 

We collect some basic properties of R-parabolic subgroups.
First off, it follows from the definition of limit in this case that for any subgroup $H$ of $G$ and a cocharacter $\lambda\in Y_H\subseteq Y_G$, we have
\begin{equation*}
\label{eqn:subgppar}
 P_\lambda(H)= P_\lambda(G)\cap H, \ L_\lambda(H)= L_\lambda(G)\cap H, \ U_\lambda(H)= U_\lambda(G)\cap H;
\end{equation*}
see \cite[Lem.\ 2.1.5]{CGP} and the paragraph following the proof of \cite[Lem.~2.1.4]{CGP}.
  
Note that every R-parabolic subgroup of $G$ contains a maximal split torus of $G$ --- for $P = P_\lambda$ for some $\lambda$, and if $T$ is a maximal split torus of $G$ such that $\lambda\in Y_T$, then $T\subseteq L_\lambda \subseteq P_\lambda$.
It is clear that if $g\in G(k)$ and $\lambda\in Y_G$, then $P_{g\cdot\lambda} = gP_\lambda g^{-1}$, $L_{g\cdot\lambda} = gL_\lambda g^{-1}$ and $U_{g\cdot\lambda} = gU_\lambda g^{-1}$; in particular, if $g\in P_\lambda(k)$, then $P_{g\cdot\lambda} = P_\lambda$.  If $k'/k$ is a field extension then $P_{\lambda_{k'}}= (P_\lambda)_{k'}$, $L_{\lambda_{k'}}= (L_\lambda)_{k'}$ and $U_{\lambda_{k'}}= (U_\lambda)_{k'}$.

We say the R-parabolic subgroups $P$ and $Q$ are \emph{opposite} if there exists $\lambda\in Y_G$ such that $P = P_\lambda$ and 
$Q = P_{-\lambda}$: in this case, $P_\lambda\cap P_{-\lambda}= L_\lambda$.

The next result follows from \cite[Prop.\ 2.1.8, Prop.\ 2.1.12]{CGP}.

\begin{lem}
\label{lem:big_cell}
 Let $G$ be a $k$-group and let $\lambda\in Y_G$.  Then the multiplication map $P_\lambda\times U_{-\lambda}\to G$ is an open immersion.  If $G$ is connected and solvable then this map is an isomorphism.
\end{lem}

\begin{lem}
\label{lem:common_tor}
 Let $\lambda\in Y_G$ and let $\mu\in Y_{P_\lambda}$.  Then there exist $u\in U_\lambda(k)$ and a maximal split torus $T$ of $P_\lambda$ such that $\lambda$ and $u\cdot \mu$ belong to $Y_T$.
\end{lem}

\begin{proof}
 Choose maximal split tori $T_0$ and $T_1$ of $P_\lambda$ such that $\lambda\in Y_{T_0}$ and $\mu\in Y_{T_1}$.  By conjugacy of maximal split tori, there exists $g\in P_\lambda(k)$ such that $gT_1g^{-1}= T_0$.  Write $g= lu$ with $l\in L_\lambda(k)$ and $u\in U_\lambda(k)$.  Then $\lambda, u\cdot \mu\in Y_T$, where $T:= l^{-1}T_0l$.
\end{proof}

\begin{lem}\label{lem:rigidity}
Let $T$ be a maximal split torus of $G$, and suppose $\lambda\in Y_T$.
We have $N_{P_\lambda(k)}(T) = N_{L_\lambda(k)}(T)$ and $N_{P_\lambda(k)}(L_\lambda) = L_\lambda(k)$.
\end{lem}

\begin{proof}
Suppose we have $g\in P_\lambda(k)$ such that $gTg^{-1} = T$.
Now write $g = lu$ with $l\in L_\lambda(k)$ and $u\in U_\lambda(k)$. 
Note that 
$\lim_{a\to 0} \lambda(a)\cdot g$ normalises $T$, because  $\lambda(a)$ and $g$ do, thus so does $l =\lim_{a\to 0} \lambda(a)\cdot g$, and hence $u$ normalises $T$ also.
Now the morphism $\Gm\to U_\lambda, a\mapsto \lambda(a)\cdot u$ extends to a morphism from all of $\mathbb{A}^1$ (by definition of the limit),
and this morphism has connected image contained in the normaliser of $T$, containing the points $1 = \lim_{a\to 0} \lambda(a)\cdot u$ and $u = \lambda(1)\cdot u$.
Rigidity of tori implies that the identity component of $N_G(T)$ is the identity component of $C_G(T)$,
\cite[8.10, Cor.~2]{Bo}, and since $\lambda$ evaluates in $T$, $C_G(T)\subseteq L_\lambda$.
But we also have $u\in U_\lambda$, so this implies $u=1$.
Thus $g\in L_\lambda(k)$, which proves the first statement.

Now suppose $h\in P_\lambda(k)$ normalises $L_\lambda$.
Then $hTh^{-1}$ is another maximal split torus of $L_\lambda$, and hence is conjugate to $T$ by some element of $L_\lambda(k)$.
But this says that we may multiply $h$ by an element of $L_\lambda(k)$ to obtain an element of $N_{P_\lambda(k)}(T)\subseteq L_\lambda(k)$,
so we're done.
\end{proof}

\begin{lem}\label{lem:basicpropertiesRpars}
Let $P = P_\lambda$ be an R-parabolic subgroup of $G$, and let $T$ be a maximal split torus of $G$ such that $\lambda\in Y_T$.
\begin{itemize}
\item[(i)] Given any other maximal split torus $T'$ of $G$ contained in $P$, there is a $\mu\in Y_{T'}$ with $P_\mu = P$ and $U_\mu = U_\lambda$.  In fact, we can take $\mu= u\cdot \lambda$ for some $u\in U_\lambda(k)$.
\item[(ii)] The R-Levi subgroup $L_\lambda$ is the unique $P_\lambda(k)$-conjugate of $L_\lambda$ containing $T$.
\item[(iii)] Let $\mu\in Y_G$ such that $P_\mu= P_\lambda$ and $L_\mu= L_\lambda$.  Then $U_\mu= U_\lambda$.
\end{itemize}
\end{lem}

\begin{proof}
(i). Since $T$ and $T'$ are maximal split tori of the smooth connected $k$-group $P$, they are $P(k)$-conjugate.
Let $g\in P(k)$ be such that $T' = gTg^{-1}$.  Write $g= ul$ with $u\in U_\lambda(k)$ and $l\in L_\lambda(k)$.
Then $u\cdot \lambda = g\cdot \lambda\in Y_{T'}$, $P_{u\cdot \lambda}= uP_\lambda u^{-1} = P_\lambda$, and $U_{u\cdot \lambda}= uU_\lambda u^{-1} = U_\lambda$, as required.

(ii). Suppose we have $g\in P(k)$ such that $T\subseteq gL_\lambda g^{-1}$.
Then $T$ and $g^{-1}Tg$ are maximal split tori of $L_\lambda$, so are conjugate by an element of $L_\lambda(k)$.
Hence, by adjusting $g$ with an element of $L_\lambda(k)$ if necessary, we may assume that $g$ normalises $T$.
Now Lemma~\ref{lem:rigidity} implies that $g\in L_\lambda(k)$, as required.

(iii). Since ${\IM}(\mu)\subseteq Z(L_\lambda)^0\subseteq T\subseteq L_\lambda$, $\mu$ belongs to $Y_T$.  It follows easily that $c_\mu\circ c_\lambda= c_\lambda\circ c_\mu$.  So let $v\in U_\mu(\ovl{k})$.  We can write $v= lu$ for some $l\in L_\lambda(\ovl{k})$ and some $u\in U_\lambda(\ovl{k})$.  Then
$$ 1= c_\lambda(c_\mu(v))= c_\lambda(c_\mu(lu))= c_\lambda(c_\mu(l))c_\lambda(c_\mu(u))= lc_\mu(c_\lambda(u))= l, $$
so $v= u$.  This shows that $U_\mu\subseteq U_\lambda$, and the reverse inclusion follows by symmetry.
\end{proof}

If $G$ is reductive then the R-parabolic subgroups of $G$ are precisely the parabolic subgroups, and the R-Levi subgroups of a parabolic subgroup are precisely its Levi subgroups; furthermore, we have $U_\lambda = R_u(P_\lambda)$. These facts mean that the result above has a rather simpler formulation and proof for reductive $G$: the content of parts (i) and (ii) of Lemma~\ref{lem:basicpropertiesRpars} for reductive $G$ are the well-known facts that each maximal torus of a parabolic subgroup is contained in a unique Levi subgroup, and
the unipotent radical of the parabolic subgroup acts simply transitively on its Levi subgroups.

If $G$ is not reductive then things are more complicated.  We have $U_\lambda\subseteq R_{u,k}(P_\lambda)$ since $U_\lambda$ is a connected normal unipotent subgroup of $P_\lambda$, but the inclusion can be proper.  We can have cocharacters $\lambda,\mu\in Y_G$ such that $P_\lambda\subseteq P_\mu$ but $U_\lambda\not\supseteq U_\mu$.  In fact, there can exist $\lambda,\mu\in Y_G$ such that $P_\lambda= P_\mu$ but $U_\lambda\not\supseteq U_\mu$ and $L_\mu$ is not conjugate to $L_\lambda$.   For instance, take $k= \ovl{k}$ and let $G= P_\mu(G')$ be a Borel subgroup of a non-trivial connected semisimple $k$-group $G'$.  Choose a maximal torus $T$ of $G$ such that $\mu\in Y_T$.  Let $\lambda= 0\in Y_G$.  Then $P_\lambda(G)= P_\mu(G)= G= L_\lambda(G)\neq L_\mu(G)= T$ but $R_u(P_\mu)= U_\mu(G)\nsubseteq U_\lambda(G)= 1$.

We do, however, have the following result.

\begin{lem}
\label{lem:Levi_factor}
 Let $P$ be an R-parabolic subgroup of $G$.  Suppose $P=P_\lambda= P_\mu$ and $U_\lambda\subseteq U_\mu$ for some  $\lambda, \mu\in Y_G$.  Let $H$ be a subgroup of $P$ such that $P\iso H\ltimes U_\mu$.  Then: 
 \begin{itemize}
 \item[(i)]  $H$ is $U_\mu(k)$-conjugate to $L_\mu$.
 \item[(ii)] $H$ is $U_\mu(k)$-conjugate to a subgroup of $L_\lambda$. In particular, $L_\mu$ is conjugate to a subgroup of $L_\lambda$.  
 \end{itemize}
\end{lem}

\begin{proof}
	Since $H$ is a complement to $U_\mu$ in $P_\mu$, $H$ is isomorphic to $L_\mu$ under the (restriction of the) projection $P_\mu \to L_\mu$, and hence $H$ contains a maximal split torus $T$ of $P_\mu$. 
	Using Lemma \ref{lem:basicpropertiesRpars}(i), by conjugating $H$ with some element in $U_\mu(k)$, we may assume that $\mu$ evaluates in $T$ as well. 
	Now $H = P_\mu(H) = L_\mu(H) \ltimes U_\mu(H)$, but $U_\mu(H) = U_\mu\cap H$ is trivial, so $H = L_\mu(H)$.
	This implies that $H\subseteq L_\mu$, thus $H=L_\mu$ since $H$ is isomorphic to $L_\mu$ and both are smooth and connected, proving (i).
	
	The proof of (ii) is similar: by Lemma \ref{lem:basicpropertiesRpars}(i) and the fact that $U_\lambda\subseteq U_\mu$, we may also conjugate $H$ by an element of $U_\mu(k)$ so that the maximal split torus $T$ of $H$ contains the image of $\lambda$. 
	Then $H = P_\lambda(H) = L_\lambda(H)\ltimes U_\lambda(H)$. Again, since $U_\lambda\subseteq U_\mu$ and $U_\mu\cap H = 1$, we have $H = L_\lambda(H)$, so $H$ is contained in $L_\lambda$.	
\end{proof}

\begin{cor}
\label{cor:Levi_conj}
 Let $\lambda,\mu \in Y_G$ such that $P_\lambda = P_\mu$ and $U_\lambda = U_\mu$.
 Then there is  a unique $u\in U_\lambda(k)$ such that $L_\mu = uL_\lambda u^{-1}$.
\end{cor}

\begin{proof}
 The existence of $u$ follows from Lemma~\ref{lem:Levi_factor}(i) with $H= L_\lambda$.  If $u_1, u_2\in U_\lambda(k)$ and $u_1L_\lambda u_1^{-1}= u_2L_\lambda u_2^{-1}$ then $u_1^{-1}u_2$ normalises $L_\lambda$, so $u_1= u_2$ by Lemma~\ref{lem:rigidity}.  This proves uniqueness.
\end{proof}

\begin{rem}
\label{rem:weightspaces}
 We need some notation.  Let $V$ be a rational $G$-module and let $T$ be a maximal split torus of $G$.  Let $\Phi\subseteq X_T$ be the set of weights of $T$ on $V$.  Since $T$ is split, $V$ splits into a direct sum of weight spaces $V_\chi$ for $T$.  Given $v\in V(\ovl{k})$, we can write $v$ uniquely as $v= \sum_{\chi\in \Phi} v_\chi$, where each $v_\chi\in V_\chi(\ovl{k})$.  We define $\supp(v)= \{\chi\in \Phi \mid v_\chi\neq 0\}$. If $g\in G(k)$ then the weights of $gTg^{-1}$ on $V$ are the characters of the form $g\cdot \chi$, where $\chi$ runs over the set of weights of $T$ on $V$, and we have $V_{g\cdot \chi}= g\cdot V_\chi$ for each $\chi\in \Phi$.  Let  $\lambda\in Y_{T}$.  Set $\Phi_{\lambda, \epsilon}= \{\chi\in \Phi\mid\sgn(\langle \lambda, \chi\rangle)= \epsilon\}$, where $\epsilon\in \{+, -, 0\}$.  We define $V_{\lambda, \epsilon}$ to be the sum of the weight spaces $V_\chi$ for $\chi\in \Phi_{\lambda, \epsilon}$, and $V_{\lambda, \geq 0}:= V_{\lambda, 0}\oplus V_{\lambda, +}$ (as in \cite[Rem.~2.8]{GIT}).
 
 Now we consider an important special case of this set-up.  By \cite[Lem.~1.1(a)]{kempf}, there is a $G$-equivariant embedding $\rho$ of $G$ (as a $G$-variety under the conjugation action) into a rational $G$-module $V$.  Let $T$, etc., be as above.  For $v\in V(A)$, $\lim_{a\to 0} \lambda(a)\cdot v$ exists if and only if $v\in V_{\lambda, \geq 0}(A)$ \cite[Ex.\ 2.1.1]{CGP}.  We see that $P_\lambda= \rho^{-1}(V_{\lambda, \geq 0})$, and it follows that $P_{n\lambda}= P_\lambda$ for every $n\in \NN$.  By similar arguments, $L_\lambda= \rho^{-1}(V_{\lambda, 0})$ and $U_\lambda= \rho^{-1}(V_{\lambda, +})$.
\end{rem}

\begin{rem}
\label{rem:finitepar}
There are only finitely many R-parabolic subgroups $P$ containing a fixed maximal torus $S$ of $G$.  To see this, choose a $G$-equivariant embedding $\rho$ of $G$ into a rational $G$-module $V$, as in Remark \ref{rem:weightspaces}.  Let $T$ be the unique maximal split subtorus of $S$.  
By Lemma \ref{lem:basicpropertiesRpars}(i), we can assume that $P = P_\lambda$ for some $\lambda\in Y_T$.  We have $P_\lambda= \rho^{-1}(V_{\lambda, \geq 0})$.  Since $\Phi$ is finite, there are only finitely many possibilities for $\Phi_{\lambda, \geq 0}$ and $V_{\lambda, \geq 0}$ and hence only finitely many possibilities for $P$.
\end{rem}

\begin{lem}
\label{lem:sum}
 Let $T$ be a maximal split torus of $G$ and let $\lambda, \mu\in Y_T$.  Suppose $P_\lambda= P_\mu$.  Then $P_{\lambda+ \mu}= P_\lambda= P_\mu$.
\end{lem}

\begin{proof}
 Set $P= P_\lambda= P_\mu$.  We prove first that $P_{\lambda+ \mu}\supseteq P$.  Choose a $G$-equivariant embedding $\rho$ of $G$ into a rational $G$-module $V$, and fix a maximal split torus $T$ of $G$.  For any $\nu\in Y_T$ and any $g\in G(\ovl{k})$, we have $g\in P_\nu(\ovl{k})$ if and only if
 $$ \min \{\langle \nu, \chi\rangle\mid \chi\in \supp(\rho(g))\}\geq 0. $$
 So let $g\in P(\ovl{k})$.  Then
  $$ \min \{\langle \lambda, \chi\rangle\mid \chi\in \supp(\rho(g))\}\geq 0 $$
 and
   $$ \min \{\langle \mu, \chi\rangle\mid \chi\in \supp(\rho(g))\}\geq 0, $$
 so
  $$ \min \{\langle \lambda+ \mu, \chi\rangle\mid \chi\in \supp(\rho(g))\}\geq 0. $$
  Hence $g\in P_{\lambda+ \mu}(\ovl{k})$.  As $P$ and $P_{\lambda+ \mu}$ are smooth, this shows that $P_{\lambda+ \mu}\supseteq P$.
  
  To complete the proof, it is enough to show that $\Lie(P_{\lambda+ \mu})= \Lie(P)$.  Let $\Phi$ be the set of roots of $G$ with respect to $T$.  By \cite[Prop.\ 2.1.8]{CGP}, if $\nu\in Y_T$ then $\Lie(P_\nu)= \fg_{\nu, \geq 0}$.  Since $\Lie(P_\lambda)= \Lie(P_\mu)$, we deduce that $\Phi_{\lambda, \geq 0}= \Phi_{\mu, \geq 0}$, so we have $\Phi_{\lambda+ \mu, \geq 0}= \Phi_{\lambda, \geq 0}= \Phi_{\mu, \geq 0}$ and $\fg_{\lambda+ \mu, \geq 0}=\fg_{\lambda, \geq 0}= \fg_{\mu, \geq 0}$.  The result follows.
 \end{proof}

\begin{lem}
\label{lem:all_tor}
 Let $P$ be an R-parabolic subgroup of $G$ and let $T$ be a maximal torus of $P$.  Then there exists $\lambda\in Y_T$ such that $P_\lambda= P$.
\end{lem}

\begin{proof}
 Choose a finite Galois extension $k'/k$ such that $T$ is $k'$-split and set $\Gamma= \Gal(k'/k)$.  By Lemma~\ref{lem:basicpropertiesRpars}(i) there exists $\lambda'\in Y_{T_{k'}}$ such that $P_{\lambda'}= P_{k'}$.  Now $\gamma\cdot \lambda'$ belongs to $Y_{T_{k'}}$ for each $\gamma\in \Gamma$ and $\sum_{\gamma\in \Gamma} \gamma\cdot \lambda'$ is $\Gamma$-stable, so it descends to an element $\lambda$ of $Y_T$.  It is easily seen that $P_{\gamma\cdot \lambda'}= \gamma(P_{\lambda'})= \gamma(P)= P$ for all $\gamma\in \Gamma$.  It follows from Lemma~\ref{lem:sum} that $P_\lambda= P$, so we are done.
\end{proof}

\begin{lem}
\label{lem:smoothint}
 Let $P$ and $Q$ be R-parabolic subgroups of $G$.  Suppose $P\cap Q$ contains a maximal split torus $T$ of $G$.  Then $P\cap Q$ is smooth.
\end{lem}

\begin{proof}
We can write $P = P_\lambda$ and $Q = P_\mu$ for some $\lambda,\mu \in Y_T$, by Lemma~\ref{lem:basicpropertiesRpars}(i).
Since $\mu$ evaluates in $P$, it follows that $P\cap Q = P_\lambda\cap P_\mu = P_\mu(P_\lambda)$.
Since $P_\lambda$ is smooth and connected, this means that the intersection is an R-parabolic subgroup of $P_\lambda$, and hence is also smooth and connected.
(A direct proof using the Lie algebras is also possible, like the one in \cite[Prop.~2.1.8]{CGP}.)
\end{proof}

\begin{lem}
\label{lem:common_tor_crit}
 Let $P$ and $Q$ be R-parabolic subgroups of $G$.  The following are equivalent.
 \begin{itemize}
  \item[(i)] $P\cap Q$ contains a maximal split torus of $G$.
  \item[(ii)] $P\cap Q$ contains a maximal torus of $G$.
  \item[(iii)] $P_{\ovl{k}}\cap Q_{\ovl{k}}$ contains a maximal torus of $G_{\ovl{k}}$.
 \end{itemize}
\end{lem}

\begin{proof}
 Clearly (ii) implies (iii).  If $P_{\ovl{k}}\cap Q_{\ovl{k}}$ contains a maximal torus of $G_{\ovl{k}}$ then $P_{\ovl{k}}\cap Q_{\ovl{k}}$ is smooth by Lemma~\ref{lem:smoothint}, so $P\cap Q$ is smooth; then $P\cap Q$ contains a maximal torus and this is a maximal torus of $G$.  This shows that (iii) implies (ii).  Now suppose $S$ is a maximal split torus of $G$ contained in $P\cap Q$.  We can choose $\lambda, \mu\in Y_S$ such that $P= P_\lambda$ and $Q= P_\mu$.  Let $T$ be a maximal torus of $G$ that contains $\IM(\lambda)$ and $\IM(\mu)$; then $T\subseteq L_\lambda\cap L_\mu\subseteq P_\lambda\cap P_\mu= P\cap Q$.  Conversely, suppose $T$ is a maximal torus of $G$ contained in $P\cap Q$.  By Lemma~\ref{lem:all_tor}, there exist $\lambda, \mu\in Y_T$ such that $P_\lambda= P$ and $P_\mu= Q$.  Then $\IM(\lambda)\IM(\mu)$ is a split torus, so it is contained in a maximal split torus $S$ of $G$.  We have $S\subseteq L_\lambda\cap L_\mu\subseteq P_\lambda\cap P_\mu= P\cap Q$.  This shows that (i) and (ii) are equivalent.
\end{proof}

\begin{lem}\label{lem:samecocharssameparabolic}
Suppose $P$ and $Q$ are two R-parabolic subgroups of a $k$-group $G$. Identify $Y_P$ and $Y_Q$ as subsets of $Y_G$ in the usual way. 
Then $P\subseteq Q$ if and only if $Y_P\subseteq Y_Q$.  
\end{lem}

\begin{proof}
That $Y_P\subseteq Y_Q$ when $P\subseteq Q$ is obvious.
For the other implication,
first observe that for any $\lambda \in Y_G$ such that $Q = P_\lambda$, and any $1\neq u \in U_{-\lambda}(k)$, $u\cdot\lambda$ does not evaluate in $Q$.
To see this, note that since the image of $\lambda$ normalises $U_{-\lambda}$, 
for every $a \in k_s$ we can write
$$
u\lambda(a)u^{-1} = \lambda(a)(\lambda(a)^{-1}u\lambda(a))u^{-1} = \lambda(a)u_a
$$ 
for some $u_a \in U_{-\lambda}(k_s)$.
Since $1\neq u \in U_{-\lambda}$, $u$ does not centralise $\rm Im(\lambda)$, and hence there exists $a\in k_s$ such that
$u_a \neq 1$, which means that $\lambda(a)u_a \not\in P_\lambda$, and we're done.

Now if we assume that $Y_P\subseteq Y_Q$, then we can find a maximal split torus $T$ in the intersection $P\cap Q$:
for any split torus of $P$ is generated by the images of its cocharacters, and these also lie in $Y_Q$.
Choose cocharacters $\lambda,\mu\in Y_T$ such that $P=P_\lambda$ and $Q=P_\mu$.  
Suppose $P$ is not contained in $Q$.  Now $U_{-\mu}(P)$ is a non-trivial unipotent subgroup of $P$ by Lemma~\ref{lem:big_cell}.  Since the image of $\mu$ acts non-trivially on this subgroup, it cannot be wound, by \cite[Prop.~B4.4]{CGP}.  
In particular, $U_{-\mu}(P)$ has a non-identity 
$k$-point $u$.  
Now $\mu$ is a cocharacter of $Q$ evaluating in $P$, so it is a cocharacter of $P$, so $u\cdot\mu \in Y_P$.
The first paragraph of the proof shows that $u\cdot\mu \not\in Y_Q$, which gives a contradiction.
\end{proof}

\subsection{Spherical buildings}
As general references on buildings we refer to \cite{abro} and \cite{tits1}.  We regard a spherical building --- as opposed to its geometric realisation --- as a combinatorial object.  It is a simplicial complex with a distinguished set of subcomplexes called apartments.  If $G$ is reductive then we can form the spherical building $\Delta_G$ as follows: the simplices of $\Delta_G$ are the parabolic subgroups ordered by reverse inclusion, and an apartment of $\Delta_G$ is the subcomplex consisting of all parabolic subgroups containing a fixed maximal split torus $T$ of $G$.  Two simplices are opposite if the corresponding parabolic subgroups are opposite.  There is a notion of convexity for subcomplexes  \cite[Prop.\ 3.1]{serre2}.  An automorphism of a spherical building $\Delta$ is an automorphism of simplicial complexes that maps apartments bijectively to apartments.  We write $\Aut (\Delta)$ for the group of automorphisms of $\Delta$.  For instance, any algebraic group automorphism of $G$ --- e.g., an inner automorphism --- gives rise to an automorphism of $\Delta_G$.  There are also automorphisms of $\Delta_G$ arising from automorphisms of the ground field $k$.  Both these kinds of automorphism have counterparts for vector edifices: see Section~\ref{subsec:linmapsII} and Section~\ref{sec:field_auts}.

\section{Edifices}
\label{sec:edifices}

In order to define vector edifices, we need to understand how to glue together the $\KK$-vector spaces $Y_T(\KK)$ for the maximal split tori of a $k$-group $G$ in a suitable way, where $\KK= \RR$ or $\QQ$.
This first requires us to attach an R-parabolic subgroup to an arbitrary element of $Y_T(\KK)$ (above, we have defined $P_\lambda$ just for $\lambda\in Y_T$).
This work is achieved in the next few subsections.

First we clarify some notation.  When we write, say, $\lambda\in Y_T(\KK)$ then it is understood that we are working with a fixed choice of $T$.  If $\mu$ is an actual cocharacter --- that is, if $\mu\in Y_G$ --- then there can exist many different maximal split tori $T$ such that $\mu\in Y_T$, so there is potential for misunderstanding.  This does not cause problems, however, for we show below that our constructions when applied to such a $\mu$ don't depend on the choice of $T$: see the paragraph following Definition~\ref{defn:V_G(K)}.

\subsection{Linear maps induced by homomorphisms of tori}\label{sec:linmapstori}
Suppose $f\colon T\to T'$ is a homomorphism of split $k$-tori. 
Then we get an induced map $Y_T\to Y_{T'}$ via $\lambda\mapsto f\circ\lambda$,
and this extends to a $\KK$-linear map $f_*\colon Y_{T}(\KK)\to Y_{T'}(\KK)$.
This map is injective if $f$ has finite kernel.
In particular, if $T\subseteq G$ for some $k$-group $G$ and $g\in G(k)$ then the conjugation map $\Inn_g$ induces a linear isomorphism from $Y_{T}(\KK)$ to $Y_{gTg^{-1}}(\KK)$; we denote the image of $\lambda\in Y_{T}(\KK)$ by $g\cdot \lambda$.
If in fact $\lambda\in Y_T$, then $g\cdot\lambda\in Y_{gTg^{-1}}$ as just defined agrees with the usual definition of $g\cdot \lambda$ given in Section~\ref{sec:prelims}.
We also have functoriality: if $f'\colon T'\to T''$ is another homomorphism of split $k$-tori then $(f'\circ f)_*= (f')_*\circ f_*$.  
In particular, if $f\colon G\to H$ is a homomorphism of $k$-groups then, given split $k$-tori $T\subseteq G$ and $S\subseteq H$ such that $f(T)\subseteq S$, we have a map $h_*\colon Y_{T}(\KK)\to Y_{S}(\KK)$, where $h\colon T\to S$ is the restriction of $f$.  
Often we abuse notation slightly and write $f_*\colon Y_{T}(\KK)\to Y_{S}(\KK)$ for this map.

We return to the subject of linear maps induced by group homomorphisms in Section~\ref{subsec:linmapsII} below.

\subsection{More on R-parabolic subgroups} 
Now we can define $P_\lambda$, $L_\lambda$ and $U_\lambda$ for $\lambda\in Y_{T}(\KK)$, where $T\subseteq G$ is a split torus.
We may as well assume that $T$ is a maximal split torus of $G$ and $\lambda\in Y_{T}(\KK)$.
Then we may choose a $G$-equivariant embedding $\rho$ of $G$ into a rational $G$-module $V$. Using the notation from Remark~\ref{rem:weightspaces},  
we define a subspace $W$ of $Y_T(\KK)$ by
$$
W= \{\mu\in Y_{T}(\KK)\mid\langle \mu, \chi\rangle= 0\textrm{ for all } \chi\in \Phi_{\lambda, 0} \}.
$$  
Clearly,
$\lambda \in W$, by the definition of $\Phi_{\lambda, 0}$.
Since $\Phi_{\lambda, 0}\subseteq X_T$, $W$ is defined by equations with coefficients in $\ZZ$ and hence is defined over $\QQ$.
Therefore, we get that $\dim_\QQ (W \cap Y_T(\QQ)) = \dim_\KK W$, which
implies that $W\cap Y_{T}(\QQ)$ is dense in $W$. 
Thus there exists a sequence $(\lambda_n)$ in
$W \cap Y_T(\QQ)$ converging to $\lambda$. For all 
$n \in \NN$ and for $\lambda' := \lambda_n$, it follows from the definition
of $W$ that 
\begin{equation}
\label{eq:chi}
 \sgn(\langle \lambda', \chi\rangle)= \sgn(\langle \lambda, \chi\rangle) 
\end{equation}
holds for all $\chi\in \Phi_{\lambda, 0}$, and for $n$ large enough \eqref{eq:chi} even holds for all $\chi\in \Phi_{\lambda, \pm}$,
because the linear functional $\langle  - , \chi\rangle$ is continuous. 
An element $\lambda'\in Y_{T}(\QQ)$ 
satisfying \eqref{eq:chi} is called 
a \emph{rational cocharacter approximation to $\lambda$} (\emph{with respect to $V$}) (\emph{in $Y_T(\KK)$}).
Multiplying a rational cocharacter approximation by a suitable positive integer, we may obtain an element $\lambda'\in Y_{T}$ satisfying \eqref{eq:chi},
which we call a \emph{cocharacter approximation to $\lambda$} (\emph{with respect to $V$}).  If $\lambda'\in Y_T(\QQ)$ then we set $P_{\lambda'}= P_{n\lambda'}$, $L_{\lambda'}= L_{n\lambda'}$ and $U_{\lambda'}= U_{n\lambda'}$ for any $n\in \NN$ such that $n\lambda'\in Y_G$: clearly this is well-defined.
We call a sequence $(\lambda_n)$ of rational cocharacter approximations to $\lambda$ such that $\lambda_n\to \lambda$ a \emph{rational approximating sequence to $\lambda$} (\emph{with respect to $V$}) (\emph{in $Y_T(\KK)$}).
   
\begin{defn}\label{def:Plambda}
Given $\lambda\in Y_T(\KK)$ and a rational cocharacter approximation $\lambda'$ to $\lambda$, we define $P_\lambda= P_{\lambda'}$, $L_\lambda= L_{\lambda'}$, and $U_\lambda = U_{\lambda'}$.
\end{defn}

\begin{lem}
\label{lem:par_levi_prop}
With the notation as above, we have the following.
 \begin{itemize}
  \item[(i)] The subgroups $P_\lambda$, $L_\lambda$ and $U_\lambda$ do not depend on the choice of $\lambda'$ or the choice of $\rho$.
  \item[(ii)] For any $g\in G(k)$, 
  if $\lambda'\in Y_T$ is a cocharacter approximation to $\lambda$ then $g\cdot \lambda'\in Y_{gTg^{-1}}$
   is a cocharacter approximation to $g\cdot \lambda$. Consequently, for any $g\in G(k)$, we have
  $P_{g\cdot \lambda}= gP_\lambda g^{-1}$,  $L_{g\cdot \lambda}= gL_\lambda g^{-1}$ and $U_{g\cdot \lambda}= gU_\lambda g^{-1}$.
 \end{itemize}
\end{lem}

\begin{proof}
 (i). We have $P_{\lambda'}= \rho^{-1}(V_{\lambda', \geq 0})$, $L_{\lambda'}= \rho^{-1}(V_{\lambda', 0})$ and $U_{\lambda'}= \rho^{-1}(V_{\lambda', >0})$ (Remark~\ref{rem:weightspaces}), so for a given $\rho$ the fact that these subgroups do not depend on the choice of $\lambda'$ is clear.  Now let $\rho_i\colon G\to V_i$ be $G$-equivariant embeddings with $\Phi_i$ the sets of weights of $T$ on $V_i$ for $i = 1,2$.  
 Let $\rho\colon G\to V_1\oplus V_2$ be the diagonal embedding, and let $\Phi = \Phi_1 \cup \Phi_2$ be the associated set of weights.  
 Choose $\lambda'\in Y_{T}$ such that $\lambda'$ is a cocharacter  approximation to $\lambda$ with respect to $V_1\oplus V_2$.  Then $\sgn(\langle \lambda', \chi\rangle)= \sgn(\langle \lambda, \chi\rangle)$ for all $\chi\in \Phi_1$ and for all $\chi\in \Phi_2$, so $\lambda'$ is a cocharacter approximation to $\lambda$ with respect to $V_1$ and with respect to $V_2$.  This proves part (i).
 
 (ii). The first statement follows from the fact that the pairing $\langle -,- \rangle$ is $G(k)$-invariant, and the fact that the set of weights of $gTg\inverse$ on $V$ is
 $g \cdot \Phi$, where $\Phi$ is the set of weights of $T$.  The second statement now follows.  
\end{proof}

\begin{rems}
\label{rem:L_robust}
 (i). Let $S$ be any split torus of $G$, let $\lambda\in Y_{S}(\KK)$, and let $T$ be any maximal split torus of $G$ containing $S$.
 We claim that there exists a cocharacter approximation $\lambda'$ to $\lambda$ contained in $Y_S\subseteq Y_T$.
 To see this, note that in the construction above, the $S$-weights on $V$ are given by the restrictions of the $T$-weights on $V$ to $S$, because $S\subseteq T$. 
 We may therefore perform the construction of $\lambda'$ above entirely within $Y_S(\KK)$ just using these restrictions, and we obtain a cocharacter approximation $\lambda'\in Y_S$ satisfying \eqref{eq:chi}.
 Analogous conclusions hold for rational cocharacter approximations and rational approximating sequences in $Y_S(\KK)$.
 
(ii). Now suppose $T_1$ and $T_2$ are two maximal split tori of $G$ that contain $S$ and let $i_1, i_2$ be the inclusions of $S$ in $T_1, T_2$ respectively.
Let $\lambda'$ be the cocharacter approximation contained
in $Y_S \subseteq Y_{T_1} \cap Y_{T_2}$ from part (i).
It is immediate from the construction that $L_{(i_j)_*(\lambda)}= L_{i_j\circ \lambda'}$ and $P_{(i_j)_*(\lambda)}= P_{i_j\circ \lambda'}$ for $j= 1,2$, so $L_{(i_1)_*(\lambda)}= L_{(i_2)_*(\lambda)}= L_\lambda$ and $P_{(i_1)_*(\lambda)}= P_{(i_2)_*(\lambda)}= P_\lambda$.  In particular, both $T_1$ and $T_2$ are contained in $L_\lambda$, and there is no dependence of our construction on the choice of a maximal split torus containing $S$.

(iii). Let $T$ be a maximal split torus of $G$.  It is clear from the construction that if $\lambda\in Y_T(\KK)$ then $T\subseteq L_\lambda\subseteq P_\lambda$.  It is also clear that $P_\lambda= \rho^{-1}(V_{\lambda, \geq 0})$, $L_\lambda= \rho^{-1}(V_{\lambda, 0})$ and $U_{\lambda}= \rho^{-1}(V_{\lambda, >0})$ for any $\lambda\in Y_T(\KK)$.
\end{rems}

The next result is an extension to all of $Y_T(\KK)$ of the obvious fact that for a cocharacter $\lambda\in Y_T$ we have $L_\lambda = G$ 
if and only if $\lambda$ evaluates in the centre of $G$.

\begin{lem}
\label{lem:lives_in_centre}
 Let $T$ be a maximal split torus of $G$ and let $\lambda\in Y_{T}(\KK)$.  Let $Z_\lambda$ be the unique maximal split torus of $Z(L_\lambda)^0$ and let $Z$ be the unique maximal split torus of $Z(G)^0$.  Then:
 \begin{itemize}
  \item [(i)] $\lambda\in Y_{Z_\lambda}(\KK)$.
  \item [(ii)] $L_\lambda= G$ if and only if $\lambda\in Y_Z(\KK)$.
 \end{itemize}
 \end{lem}

\begin{proof}
 (i). Let $(\lambda_n)$ be a rational approximating sequence to $\lambda$ in $Y_T(\KK)$.  Choose $a_n\in \NN$ such that $a_n\lambda_n\in Y_{T}$ for each $n\in \NN$.  Then for all $n\in \NN$, $L_{\lambda_n}= L_{a_n\lambda_n}= C_G(a_n\lambda_n)$, so $a_n\lambda_n\in Y_{Z_{\lambda_n}}$, and so $\lambda_n\in Y_{Z_{\lambda_n}}(\KK)$.  But $L_{\lambda_n}= L_\lambda$ for $n$ sufficiently large, and $Y_{Z_\lambda}(\KK)$ is closed in $Y_{T}(\KK)$, so $\lambda\in Y_{Z_\lambda}(\KK)$, as required.
 
 (ii). The forward implication follows from part (i).  The reverse implication follows from Remark~\ref{rem:L_robust}(i) applied to the split torus $Z$ --- we may find a cocharacter approximation $\lambda'\in Y_Z$ and then we have $L_\lambda = L_{\lambda'} = G$.
\end{proof}

\subsection{The vector edifice}
\label{subsec:vector_edifice}
Now we can define $Y_{G}(\KK)$ and $V_{G}(\KK)$.   We work with the disjoint union $\bigsqcup_T Y_{T}(\KK)$, where $T$ ranges over the maximal split tori of $G$.  It is convenient to write elements of $\bigsqcup_T Y_{T}(\KK)$ explicitly as pairs $(T,\lambda)$, where $T$ is a maximal split torus of $G$ and $\lambda\in Y_{T}(\KK)$.
The group $G(k)$ acts on $\bigsqcup_T Y_{T}(\KK)$ by $g\cdot (T,\lambda)= (gTg^{-1},g\cdot \lambda)$.

We define two relations on $\bigsqcup_T Y_{T}(\KK)$ by
$$ (T_1,\lambda_1)\sim (T_2,\lambda_2)\ \mbox{if there exists $l\in L_{\lambda_1}(k)$ such that $T_2= lT_1l^{-1}$ and $\lambda_2= l\cdot \lambda_1$} $$
and
$$ (T_1,\lambda_1)\approx (T_2,\lambda_2)\ \mbox{if there exists $g\in P_{\lambda_1}(k)$ such that $T_2= gT_1g^{-1}$ and $\lambda_2= g\cdot \lambda_1$}. $$

\begin{lem}
\label{lem:PLindpndce}
With the notation as above, we have the following.
 \begin{itemize}
  \item[(i)] If $(T_1,\lambda_1) \sim (T_2,\lambda_2)$ then $L_{\lambda_1} = L_{\lambda_2}$.
  \item[(ii)] If $(T_1,\lambda_1)\approx (T_2,\lambda_2)$ then $P_{\lambda_1} = P_{\lambda_2}$.
  \item[(iii)] The relations $\sim$ and $\approx$ are equivalence relations.
 \end{itemize}
\end{lem}

\begin{proof}
 Parts (i) and (ii) are a consequence of Lemma~\ref{lem:par_levi_prop}(ii), and (iii) follows from (i) and (ii).
\end{proof}

\begin{defn}\label{defn:V_G(K)}
We define
\begin{equation*}
Y_{G}(\KK) := \bigsqcup_T Y_{T}(\KK)/\sim
\end{equation*}
and
\begin{equation*}
V_{G}(\KK) := \bigsqcup_T Y_{T}(\KK)/\approx,
\end{equation*}
and we call $V_G(\KK)$ the \emph{vector edifice of $G$ (over $\KK$)}.
When we refer to ``a vector edifice $V$'' below, it is with the understanding that $V = V_G(\KK)$ for some $k$-group $G$ and for $\KK= \QQ$ or~$\RR$.
\end{defn}

We can regard $V_{G}(\KK)$ as a quotient of $Y_{G}(\KK)$, since $\approx$ is coarser than $\sim$,
and we denote the corresponding projection by 
$$
\phi_G\colon Y_G(\KK)\to V_G(\KK).
$$ 
We write $\varpi_G\colon \bigsqcup_T Y_T(\KK)\to Y_G(\KK)$ for the canonical projection and we define 
$$\omega_G := \phi_G\circ \varpi_G \colon \bigsqcup_T Y_T(\KK)\to V_G(\KK).$$

Henceforth, we often abuse notation and denote elements of $Y_G(\KK)$ simply by $\lambda$.  Note that we may regard $Y_G$ as a subset of $Y_G(\KK)$: for given $\lambda\in Y_G$,
the equivalence relation $\sim$ does the job of identifying all pairs $(T,\lambda)$ where $T$ is a maximal split torus of $G$  
containing the image of $\lambda$, because all such tori are conjugate by elements of $L_\lambda(k)$.
More generally, if $S$ is a subtorus of two maximal split tori $T_1$ and $T_2$, then $\sim$ formally matches up the ``common subspace'' $Y_S(\KK)$ inside $Y_{T_1}(\KK)$ and $Y_{T_2}(\KK)$.  

\begin{rem}
\label{rem:common_tor}
	We observe that Lemma~\ref{lem:common_tor}
	holds not only for elements of $Y_G$ but also for elements of $Y_G(\KK)$:
	If $\lambda\in Y_G(\KK)$ and $\mu\in Y_{P_\lambda}(\KK)$ then there exist a maximal split torus $T$ of $G$ and $u\in U_\lambda(k)$ such that $\lambda, u\cdot \mu$ belong to $Y_T(\KK)$. To see this, repeat the proof of  Lemma~\ref{lem:common_tor} with 
	$Y_{T_0}$, $Y_{T_1}$, and $Y_{T}$ replaced by 
	$Y_{T_0}(\KK)$, $Y_{T_1}(\KK)$, and $Y_{T}(\KK)$. 
\end{rem}

\begin{rem}
\label{rem:obvious_gen}
  Lemmas~\ref{lem:rigidity}, \ref{lem:basicpropertiesRpars} and \ref{lem:Levi_factor} and Corollary~\ref{cor:Levi_conj} extend to $Y_G(\KK)$: just apply each original result to suitable cocharacter approximations.
\end{rem}

\begin{rems}
\label{rem:eqvce_props}
(i). 
One checks easily that the action of $G(k)$ on $\bigsqcup_T Y_{T}(\KK)$ descends to give actions of $G(k)$ on $Y_G(\KK)$ and $V_G(\KK)$.
We sketch the argument for $V_G(\KK)$.
Suppose we have a pair $(T,\lambda)$ consisting of a maximal split torus $T$ and $\lambda\in Y_T(\KK)$, and elements $p\in P_\lambda(k)$ and  $g\in G(k)$. 
We wish to show that 
$$
(gTg^{-1},g\cdot \lambda)\approx (g(pTp^{-1})g^{-1},g\cdot (p\cdot\lambda)) = (gpT(gp)^{-1},(gp)\cdot\lambda).
$$
Since both $gTg^{-1}$ and $gpT(gp)^{-1}$ are split tori of $gP_\lambda g^{-1} = P_{g \cdot \lambda}$, there is a $p' \in P_{g\cdot\lambda}(k)$ such that $p'$ conjugates $gTg^{-1}$ to $gpT(gp)^{-1}$. 
Writing $p' = gqg^{-1}$ for some $q\in P_\lambda(k)$, we see that $q^{-1}p$ belongs to $N_{P_\lambda(k)}(T)$, and this group coincides with $N_{L_\lambda(k)}(T)$, by the extension of Lemma~\ref{lem:rigidity} to $Y_G(\KK)$ (Remark \ref{rem:obvious_gen}).
Consequently, $q^{-1}p$ fixes $\lambda$, so 
$$
p'\cdot(g\cdot\lambda) = (gq)\cdot\lambda = (gp)\cdot\lambda, 
$$
and we are done.

(ii).
In a similar vein, if $(T, \lambda)\approx (T,\mu)$ then $\lambda= \mu$.  
For suppose $g\in P_\lambda(k)$ is such that $gTg^{-1}= T$ and $g\cdot \lambda= \mu$. 
Then, by appealing to the extension of Lemma~\ref{lem:rigidity} 
to $Y_G(\KK)$ as in (i), we see that  $g\in L_\lambda(k)$, and hence $\lambda = \mu$.
\end{rems}

\begin{defn}
 Given $x\in V_G(\KK)$, we set $P_x:=P_\lambda$ 
 for any $\lambda\in Y_G(\KK)$ such that $\phi_G(\lambda) = x$.  Since the relation $\approx$ only identifies pairs $(T_1,\lambda_1)$ and $(T_2,\lambda_2)$ for which $P_{\lambda_1} = P_{\lambda_2}$ (Lemma~\ref{lem:PLindpndce}(ii)), this is well-defined.
\end{defn}

\begin{rem}
 For now we regard $V_G(\KK)$ just as a set.  In Section~\ref{sec:admmetric} we topologise it by endowing it with a metric.
\end{rem}

\begin{exmp}
	\label{ex:red_virtual_cochar}
	The definition of $P_\lambda$, $L_\lambda$ and $U_\lambda$ for  $\lambda\in Y_G(\KK)$ agree with those of \cite[Sec.\ 2.2]{BMR:strong} when $G$ is reductive (where the definitions are given in terms of pairings with roots).  To see this, we can assume by Remark~\ref{rem:L_robust}(iii) that $k$ is separably closed and choose a maximal torus $T$ of $G$.  Choose an equivariant embedding $i$ of $G$ in a rational $G$-module $V$.  Let $\lambda\in Y_T(\KK)$ and let $\lambda'\in Y_T$ be a cocharacter approximation to $\lambda$.  The derivative of $i$ gives a $G$-equivariant embedding of $\fg$ in $T_0(V)$, and 
	$T_0(V)$  is isomorphic to $V$ as a $G$-module.  Hence all the roots of $G$ with respect to $T$ appear in the set of weights of $T$ on $V$.  
	It follows that $\sgn\langle \lambda', \alpha\rangle= \sgn\langle \lambda, \alpha\rangle$ for every root $\alpha$.  Now $P_\lambda$ is equal to $P_{\lambda'}$, which is generated by $T$ together with the root groups $U_\alpha$ for the roots $\alpha$ such that $\langle \lambda', \alpha\rangle\geq 0$.  This shows that $P_\lambda$ coincides with the subgroup defined in \cite[Sec.\ 2.2]{BMR:strong}, since $\langle \lambda, \alpha\rangle\geq 0$ if and only if $\langle \lambda', \alpha\rangle\geq 0$ for each root $\alpha$.  The argument for $L_\lambda$ and $U_\lambda$ is similar: here $L_{\lambda'}$ is generated by $T$
	together with the root groups $U_\alpha$ for the roots $\alpha$ such that $\langle \lambda', \alpha\rangle = 0$, and 
	$U_{\lambda'}$ is generated by $T$
	together with the root groups $U_\alpha$ for the roots $\alpha$ such that $\langle \lambda', \alpha\rangle > 0$.
\end{exmp}

\begin{lem}
	\label{lem:homapprox}
	Let $f\colon G\to H$ be a homomorphism of connected $k$-groups.  Let $T$ be a maximal split torus of $G$ and let $S$ be a maximal split torus of $H$ such that $f(T)\subseteq S$.  Let $\lambda\in Y_T(\KK)$.  Then there exist a $G$-equivariant embedding $\nu\colon G\to V$, an $H$-equivariant embedding $\psi\colon H\to W$, and $\lambda'\in Y_T$ such that $\lambda'$ is a cocharacter approximation to $\lambda$ with respect to $V$ and $f_*(\lambda')$ is a cocharacter approximation to $f_*(\lambda)$ with respect to $W$.
\end{lem}

\begin{proof}
	Pick a $G$-equivariant embedding $\phi\colon G\to V_1$ for some rational $G$-module $V_1$, and an $H$-equivariant embedding $\psi\colon H\to W$ for a rational $H$-module $W$.  Define $\nu\colon G\to V:= V_1\oplus W$ by $\nu(g)= (\phi(g), \psi(f(g)))$ for $g\in G(A)$ for any $k$-algebra $A$.  Then $\nu$ is a $G$-equivariant embedding of $G$ in $V$, and it follows from the construction that if $\lambda'$ is a cocharacter approximation to $\lambda$ with respect to $V$ then $f_*(\lambda')$ is a cocharacter approximation to $f_*(\lambda)$ with respect to $W$. Here we used the fact that $\langle \lambda, \chi \circ f\rangle = \langle f \circ \lambda, \chi \rangle$, where $\chi$ is any weight of $S$ on $W$ and $\lambda \in Y_T$, and that this formula extends to all $\lambda \in Y_T(\KK)$ (replacing $f \circ \lambda$ with $f_*(\lambda)$).
\end{proof}

\begin{rem}
	The proof of Lemma \ref{lem:homapprox} actually gives a stronger conclusion: in the notation of the lemma 
	there exist equivariant embeddings such that for
	any $\lambda \in Y_T(\KK)$ and
	 \emph{any}
	 cocharacter approximation $\lambda'$ to $\lambda$ with respect to $V$, $f_*(\lambda')$ is also a cocharacter	approximation to
	 $f_*(\lambda)$ with respect to $W$.   It follows that if $(\lambda_n)$ is a rational approximating sequence in $Y_T(\KK)$ to $\lambda$ then $(f(\lambda_n))$ is a rational approximating sequence in $Y_S(\KK)$ to $f_*(\lambda)$.
\end{rem}

\begin{cor}
	\label{cor:par_intersect}
	Suppose $G$ is a subgroup of a connected $k$-group $H$.  Then for any $\lambda\in Y_G(\KK)$ we have $P_\lambda= P_\lambda(H)\cap G$, $L_\lambda= L_\lambda(H)\cap G$ and $U_\lambda= U_\lambda(H)\cap G$.
\end{cor}

\begin{proof}
	Let $i\colon G\to H$ be the inclusion.  By Lemma~\ref{lem:homapprox}, there exists $\lambda'\in Y_G$ such that we have the following equalities: 
	$P_\lambda= P_{\lambda'}$; $L_\lambda= L_{\lambda'}$; $U_\lambda= U_{\lambda'}$; $P_{i^*(\lambda)}(H)= P_{i^*(\lambda')}(H)$; $L_{i^*(\lambda)}(H)= L_{i^*(\lambda')}(H)$; and $U_{i^*(\lambda)}(H)= U_{i^*(\lambda')}(H)$.
	But $P_{i^*(\lambda')}(H)=P_{\lambda'}(H)$, because $i^*(\lambda')$ is just $\lambda'$ viewed as an element of $Y_H$ via the inclusion $Y_G\subseteq Y_H$,
	and similarly $L_{i^*(\lambda')}(H) = L_{\lambda'}(H)$ and $U_{i^*(\lambda')}(H) = U_{\lambda'}(H)$.
	The result follows since $P_{\lambda'}= P_{\lambda'}(H)\cap G$, $L_{\lambda'}= L_{\lambda'}(H)\cap G$ and $U_{\lambda'}= U_{\lambda'}(H)\cap G$.
\end{proof}

\begin{lem}
	\label{lem:par_levi_f_prop}
	Let $f\colon G\to H$ be a homomorphism of $k$-groups, let $T$ be a maximal split torus of $G$ and let $\lambda,\mu\in Y_T(\KK)$.  Let $S$ be a maximal split torus of $H$ such that $f(T)\subseteq S$.
	\begin{itemize}
		\item[(i)] We have $f(P_\lambda)\subseteq P_{f_*(\lambda)}$ and $f(L_\lambda)\subseteq L_{f_*(\lambda)}$, with equality if $f$ is surjective.
		\item[(ii)] If $f$ has finite kernel, then $P_{f_*(\lambda)} = P_{f_*(\mu)}$ implies $P_\lambda = P_\mu$, and similarly $L_{f_*(\lambda)} = L_{f_*(\mu)}$  implies $L_\lambda = L_\mu$.  Hence, if $f$ is surjective and has finite kernel, then for $\lambda,\mu\in Y_T(\KK)$ we have $P_\lambda= P_\mu$ if and only if $P_{f_*(\lambda)} = P_{f_*(\mu)}$, and $L_\lambda = L_\mu$ if and only if $L_{f_*(\lambda)} = L_{f_*(\mu)}$.
	\end{itemize}
\end{lem}

\begin{proof}
	(i). For any $\lambda'\in Y_T$, it is immediate from the definitions that $f(P_{\lambda'})\subseteq P_{f_*(\lambda')}$ and $f(L_{\lambda'})\subseteq L_{f_*(\lambda')}$, and we have equality by \cite[Cor.~2.1.9]{CGP} if $f$ is surjective (note that the hypotheses of \emph{loc.\ cit.}\ hold because a surjective map of smooth affine group schemes is faithfully flat by \cite[Prop.~1.70]{milne}).  By Lemma~\ref{lem:homapprox} there exists $\lambda'\in Y_T$ such that $\lambda'$ is a cocharacter approximation to $\lambda$ and
	$f_*(\lambda')$ is a cocharacter approximation to $f_*(\lambda)$.  So
	$f(P_\lambda) = f(P_{\lambda'}) \subseteq P_{f_*(\lambda')}= P_{f_*(\lambda)}$, with equality if $f$ is surjective, and similarly for $L_\lambda$.
	
	(ii). As in part (i), it is enough by Lemma~\ref{lem:homapprox} to prove the result when $\lambda\in Y_G$. 
	Let $P = P_{f_*(\lambda)}$.
	Since $f$ has finite kernel, the preimage $f^{-1}(P)$ is a finite extension of the smooth connected group $P_\lambda$.
	Hence, we may conclude that $((f_{\bar{k}})^{-1}(P_{\bar{k}}))_{\red}$ and $(P_\lambda)_{\bar{k}}$ have the same identity component, which is $(P_\lambda)_{\bar{k}}$ itself.
	Setting $Q=P_{f_*(\mu)}$, we deduce that
	\begin{align*}
	P = Q &\implies f^{-1}(P) = f^{-1}(Q) \implies ((f_{\bar{k}})^{-1}(P_{\bar{k}}))_{\red} = ((f_{\bar{k}})^{-1}(Q_{\bar{k}}))_{\red} \\
	&\implies (P_\lambda)_{\bar{k}} = (P_\mu)_{\bar{k}} \implies P_\lambda = P_\mu.
	\end{align*}
	A similar argument works for $L_\lambda$ and $L_\mu$.  The final assertion now follows from part (i).
\end{proof}

\subsection{Apartments}
 It follows from Remark~\ref{rem:eqvce_props} that the restriction of $\phi_G$ gives rise to a bijection from $Y_{T}(\KK)$ onto its image in 
$V_{G}(\KK)$; 
we denote this image by $V_{T}(\KK)$ and we call this an \emph{apartment} of $V_{G}(\KK)$.  The set $V_T(\KK)$ inherits the structure of a $\KK$-vector space from $Y_T(\KK)$. 
We denote the common zero of all the apartments by $0$. 

\begin{rem}\label{rem:inconsistent}
Our new notation has introduced a potential inconsistency --- given a split torus $T$ in $G$, we have two objects labelled $V_T(\KK)$: the apartment of $V_G(\KK)$ corresponding to $T$, and the vector edifice of $T$ as a $k$-group in its own right.
It is not hard to see that these two objects are essentially the same, but we leave a proper discussion of this until
we have introduced linear maps below; see Remark~\ref{rem:notinconsistent}. 
\end{rem}

\begin{lem}\label{lem:toritoapts}
The map $T\mapsto V_T(\KK)$ gives a bijection between the set of maximal split tori in $G$ and the set of apartments in $V_G(\KK)$.
\end{lem}

\begin{proof}
The given map is surjective by definition, so we just need to show that it is injective.  Suppose $T_1$ and $T_2$ are maximal split tori in $G$,
and suppose that $V_{T_1}(\KK) = V_{T_2}(\KK)$.
Then for all $\lambda\in Y_{T_1}$ there exists $\mu\in Y_{T_2}$ such that $(T_1,\lambda)\approx(T_2,\mu)$.
This means that $T_2\subseteq P_\lambda$ for every $\lambda\in Y_{T_1}$,
and hence $T_2\subseteq P_\lambda\cap P_{-\lambda} = L_\lambda$ for every $\lambda\in Y_{T_1}$.
Hence $T_2$ commutes with every cocharacter of $T_1$, which implies that $T_2$ commutes with $T_1$,
and we conclude that $T_2 = T_1$, by maximality of $T_1$ and $T_2$.
\end{proof}

\subsection{Common apartments, addition, opposites}\label{sec:common_apt}
As detailed in the introduction, the structures we have just defined are generalisations to arbitrary $k$-groups
of constructions which already appear in the literature for reductive groups, where they give rise to (vector and spherical) buildings \cite{curtislehrertits}, \cite[Ch.\ 2, Sec.\ 2]{mumford}, \cite[Sec.\ 2.4]{BMR:strong}. 
However, we cannot import all of the structure/properties we see in the reductive case to the case of arbitrary $G$.  
One fundamental problem is that two points of $V_{G}(\KK)$ need not lie in a common apartment, as Example \ref{ex:no_common_apt} below shows.\footnote{For another example of building-like structures where the same phenomenon occurs, see the theory of hovels \cite{rousseau:hovels}, \cite{GR}.}
Before giving that example, we give a characterisation of when two points \emph{do} lie in a common apartment.

\begin{lem}
\label{lem:common_apt}
Given $x,y\in V_G(\KK)$, the apartments containing both $x$ and $y$ are precisely those of the form $V_T(\KK)$ where $T$ is a maximal split torus of $G$ contained in $P_x\cap P_y$.
In particular, $x$ and $y$ are contained in a common apartment if and only if $P_x\cap P_y$ contains a maximal split torus of $G$.
\end{lem}

\begin{proof}
If $x$ and $y$ are contained in $V_T(\KK)$ for some maximal split torus $T$, then we may write $x = \phi_G(\lambda)$, $y=\phi_G(\mu)$ for
$\lambda,\mu\in Y_T(\KK)$.
But then $T\subseteq P_\lambda\cap P_\mu = P_x\cap P_y$ by Remark~\ref{rem:L_robust}(iii).
Conversely, let $T$ be a maximal split torus of $G$ contained in $P_x\cap P_y$.  Choose $\lambda, \mu\in Y_G(\KK)$ such that $x= \phi_G(\lambda)$ and $y= \phi_G(\mu)$.  By Lemma~\ref{lem:basicpropertiesRpars}(i) and Remark~\ref{rem:obvious_gen}, there exist $u\in U_\lambda(k)$ and $v\in U_\mu(k)$ such that $u\cdot \lambda, v\cdot \mu\in Y_T(\KK)$.  Then $x= \phi_G(u\cdot \lambda)\in V_T(\KK)$ and $y= \phi_G(v\cdot \mu)\in V_T(\KK)$.
\end{proof}

\begin{rem}\label{rem:conjugateapartments}
 In particular, taking $x= y$ in Lemma~\ref{lem:common_apt} we see that for any $x\in V_G(\KK)$, $x$ belongs to $V_T(\KK)$ if and only if $T$ is contained in $P_x$.
 This also shows that if $T'$ is another maximal split torus of $G$ such that $x\in V_{T'}(\KK)$, 
 then there is $g\in P_x(k)$ such that $T' = gTg^{-1}$. 
\end{rem}

\begin{exmp}
\label{ex:no_common_apt}
 Let $B$ be a Borel subgroup of a reductive group $G$ with $G$ not a torus, and write $B = P_\lambda$ for some $\lambda\in Y_G$.
Then $\phi_B(-\lambda)$ and $\phi_B(-u\cdot \lambda)$ do not lie in a common apartment of $V_B(\KK)$ for any $1\neq u\in R_u(B)(k)$, because
$P_{-\lambda}(B)=T$ and $P_{-u\cdot\lambda}(B)= uTu^{-1}$ are distinct maximal tori of $B$, 
and $V_T(\KK)$ (resp., $V_{uTu^{-1}}(\KK)$) is the unique apartment containing $\phi_B(-\lambda)$ (resp., $\phi_B(-u\cdot\lambda)$), by Lemma \ref{lem:common_apt}.
\end{exmp}

The preceding discussion motivates the following definition.

\begin{defn}
\label{defn:common_apt}
 We say that $V_G(\KK)$ has the \emph{common apartment property} if every $x,y\in V_G(\KK)$ lie in a common apartment of $V_G(\KK)$.
\end{defn}

It follows from Lemma \ref{lem:common_apt} that the common apartment property holds for $G$ if and only if the intersection of any two R-parabolic subgroups contains a maximal split torus of $G$ --- this is the case for reductive groups, and more generally for pseudo-reductive groups, where the R-parabolic subgroups are the same thing as the \emph{pseudo-parabolic subgroups}, see \cite[Def.~2.2.1, Ex.~2.2.2]{CGP}. 
In fact, \cite[Prop.~3.5.12(1)]{CGP} shows that the intersection of two pseudo-parabolic subgroups in any $k$-group $G$
contains a maximal split torus, but working with pseudo-parabolic subgroups rather than R-parabolic subgroups
does not give the structure we need for applications below. 
Returning to the example of the Borel subgroup $B$ in Example \ref{ex:no_common_apt},
the only pseudo-parabolic subgroup of $B$ is $ B$ itself; 
in contrast, we show below (Example~\ref{ex:bijectivenotiso}) that there is a naturally arising \emph{bijection} between the vector edifice $V_B(\KK)$ and the vector edifice $V_G(\KK)$.

An immediate consequence of this lack of common apartment is that there is no obvious addition law on $V_{G}(\KK)$ in general, and the notions of convexity and opposition become more complicated. 
One can, however, define addition when the two points one considers live in a common apartment.  

\begin{lem}
\label{lem:add}
\begin{itemize}
	\item[(i)]  Let $T$ be a maximal split torus of $G$. Suppose $x,y\in V_G(\KK)$ are contained in the common apartment $V_T(\KK)$.
	Since $V_T(\KK)$ is a vector space, there is an element $x+y\in V_T(\KK)$. 
	The element $x+y$ is independent of the choice of apartment containing $x$ and $y$.
	\item[(ii)]  Let $x\in V_G(\KK)$ and let $a\in \KK^+$.  The element $ax$ is independent of the choice of apartment containing $x$.
\end{itemize}
\end{lem}

\begin{proof}
 (i). The common apartments containing $x$ and $y$ correspond to the maximal split tori of $G$ contained in the intersection $P_x\cap P_y$ (Lemma~\ref{lem:common_apt}).  By hypothesis, there is at least one such $T$.  Let $T'$ be another.  By Lemma~\ref{lem:smoothint}, $P_x\cap P_y$ is smooth, so there exists $g\in P_x(k)\cap P_y(k)$ such that $gTg^{-1}= T'$.  Choose $\lambda, \mu\in Y_T(\KK)$ such that $x= \phi_G(\lambda)$ and $y= \phi_G(\mu)$.  Then $x= \phi_G(g\cdot \lambda)$ and $y= \phi_G(g\cdot \mu)$.  
 Fix a $G$-equivariant embedding of $G$ in a rational $G$-module $V$, and let $\lambda'$ and $\mu'$ be cocharacter approximations (with respect to $V$) in $Y_T$ to $\lambda$ and $\mu$, respectively.
 Then  $\lambda'+\mu' \in Y_T$ is a cocharacter approximation to $\lambda+\mu$, and so $P_{\lambda+ \mu} = P_{\lambda'+ \mu'}$.
 Since $P_x = P_\lambda = P_{\lambda'}$ and 
 $P_y = P_\mu = P_{\mu'}$, we see that $\lim_{a\to 0} (\lambda'+ \mu')(a)\cdot g$ exists, so $g\in P_{\lambda'+ \mu'}(k) = P_{\lambda+ \mu}(k)$.  Hence $\phi_G(g\cdot \lambda+ g\cdot \mu)= \phi_G(g\cdot (\lambda+ \mu))= \phi_G(\lambda+ \mu)$.  The result follows.
 
 (ii). Let $T, T'$ be maximal split tori of $G$ such that $x\in V_T(\KK)\cap V_{T'}(\KK)$.  There exists $g\in P_x$ such that $gTg^{-1}= T'$.  Pick $\lambda\in Y_T(\KK)$ such that $x= \phi_G(\lambda)$.  Then $\phi_G(a(g\cdot \lambda))= \phi_G(g\cdot (a\lambda))= \phi_G(a\lambda)$ since $P_{a\lambda}= P_\lambda$.  The result follows.
\end{proof}

\begin{defn}\label{def:opp}
We say that $x$ and $y$ are \emph{opposite} in $V_G(\KK)$ if there is an apartment containing $x$ and $y$ and such that $x+y=0$ in that apartment.
It follows from Lemma~\ref{lem:add} that if $x$ and $y$ are opposite, then $x+y=0$ in every apartment containing $x$ and $y$.
\end{defn}

\begin{rem}
 Note that part (ii) of Lemma~\ref{lem:add} is false if $a= -1$ (or, more generally, if $a< 0$): for $x$ has many different opposites, one for each apartment containing $x$, and each is of the form $-x$ inside the corresponding apartment.  See Remark~\ref{rem:many_opposites}.
\end{rem}

Here is a criterion for points to be opposite in terms of cocharacters.

\begin{lem}\label{lem:uniqueopp}
Suppose $x$ and $y$ are opposite points of $V_G(\KK)$.
Then $P_x$ and $P_y$ are opposite R-parabolic subgroups of $G$,
and there is a unique $\lambda \in Y_G(\KK)$ such that $x = \phi_G(\lambda)$ and $y = \phi_G(-\lambda)$.
\end{lem}

\begin{proof}
First choose an apartment $V_T(\KK)$ such that $x,y\in V_T(\KK)$.
Then since $\phi_G$ restricts to a linear isomorphism $Y_T(\KK)\to V_T(\KK)$,
we have a unique $\lambda \in Y_T(\KK)$ such that $x= \phi_G(\lambda)$ and $y = \phi_G(-\lambda)$.
But then $P_x = P_\lambda$ and $P_y = P_{-\lambda}$, hence $P_x$ and $P_y$ are opposite.
Any maximal split torus $T'$ of $G$ such that $x, y\in V_{T'}(\KK)$ must lie in $P_x\cap P_y = L_\lambda$ by Lemma~\ref{lem:common_apt}.  
Since $L_\lambda$ is smooth, any two such tori are  conjugate by an element of $L_\lambda(k)$.
But elements of $L_\lambda(k)$ fix $\lambda$ and $-\lambda$, which shows that $\lambda$ is indeed unique as an element of $Y_G(\KK)$.
\end{proof}

\begin{rem}
\label{rem:many_opposites}
Note that for a given $\lambda\in Y_G(\KK)$, there are in general many opposite points to $\phi_G(\lambda)$ in $V_G(\KK)$. 
To see this, note that any of the cocharacters $\mu = u\cdot(-\lambda)$ 
has this property, where $u\in U_\lambda(k)$ --- the point here is that for such a $u$, $\phi_G(u\cdot\lambda) = \phi_G(\lambda)$,
but $\phi_G(u\cdot(-\lambda))\neq \phi_G(-\lambda)$ in general. 
\end{rem}

\begin{defn}\label{def:convex}
	Let $x,y \in V_{G}(\KK)$.
If there is an apartment in $V_{G}(\KK)$ containing $x$ and $y$, then we let $$[x,y] := \{ax+(1-a)y\mid a\in [0,1]\cap \KK\}$$ 
denote the \emph{geodesic} between $x$ and $y$ in $V_{G}(\KK)$ ---
again, Lemma~\ref{lem:add} implies that this geodesic (when it exists at all) is independent of the apartment we choose to draw it in.
We say a subset $S\subseteq V_G(\KK)$ is \emph{convex} if for all $x,y\in S$ such that $x$ and $y$ lie in a common apartment, $[x,y]\subseteq S$.
\end{defn}

The failure of the common apartment property makes it harder to metrise $V_{G}(\KK)$.  
We return to this in Section~\ref{sec:admmetric}.

\begin{defn}
\label{defn:type}
 We say that $x,y\in V_G(\KK)$ \emph{have the same type} if $y= g\cdot x$ for some $g\in G(k)$.
\end{defn}

\begin{rem}
\label{rem:W_type}
 Suppose $x,y\in V_G(\KK)$ have the same type and $x$ and $y$ belong to a common apartment $V_T(\KK)$.  Choose $\lambda, \mu\in Y_T(\KK)$ such that $x= \phi_G(\lambda)$ and $y= \phi_G(\mu)$.  Since $x$ and $y$ have the same type, there exists $g\in G(k)$ such that $\mu= g\cdot \lambda$.  Then $\mu\in Y_{T'}(\KK)$, where $T':= gTg^{-1}$.  So $T,T'$ are maximal split tori of $L_\mu$, so they are $L_\mu(k)$-conjugate.  Now $L_\mu(k)$ fixes $\mu$, so after multiplying $g$ on the left by an element of $L_\mu(k)$, we can assume that $g$ normalises $T$.  This shows that $x$ and $y$ are $W_k$-conjugate,  where $W_k$ is the relative Weyl group.  We conclude that for a fixed maximal split torus $T$ of $G$ and fixed $x\in V_T(\KK)$, the set $\{z\in V_T(\KK)\mid \mbox{$x$ and $z$ have the same type}\}$ is finite.
 
 For example, suppose $G= \SL_2$.  If $x,y\in V_G(\KK)$ are distinct elements having the same type then $x$ and $y$ are opposite, because each apartment $V_T(\KK)$ is 1-dimensional and the non-trivial element of the relative Weyl group $W_k$ acts by $z\mapsto -z$.
\end{rem}

\subsection{The spherical edifice}
Recall from Lemma \ref{lem:add}(ii) that we have a well-defined operation of scalar multiplication by elements of $\KK^+$ on the vector edifice $V_{G}(\KK)$,
allowing the following definition.

\begin{defn}\label{defn:spherical}
The \emph{spherical edifice} $\Delta_{G}(\KK)$ is defined to be the set of $\KK^+$-orbits on $V_{G}(\KK)\setminus \{0\}$.
We denote the natural map from $V_{G}(\KK)\setminus \{0\}$ to $\Delta_{G}(\KK)$ by $$\zeta_G : V_{G}(\KK)\setminus \{0\} \to \Delta_{G}(\KK).$$
\end{defn}

\begin{rem}
The obvious inclusion of $V_G(\QQ)$ in $V_G(\RR)$ gives an inclusion of $\Delta_G(\QQ)$ in $\Delta_G(\RR)$.
\end{rem}

\begin{rem}
 The notions of apartment, opposite points and convexity make sense for $\Delta_G(\KK)$ as well.  We define an \emph{apartment} of $\Delta_G(\KK)$ to be a subset of the form $\zeta_G(V_T(\KK)\backslash\{0\})$ for some maximal split torus $T$ of $G$.  If $y_1, y_2\in \Delta_G(\KK)$ then we say that $y_1$ and $y_2$ are \emph{opposite} if there exists a maximal split torus $T$ of $G$ and a point $x\in V_T(\KK)$ such that $y_1= \zeta_G(x)$ and $y_2= \zeta_G(-x)$.  If $y_1= \zeta_G(x_1)$ and $y_2= \zeta_G(x_2)$ are not opposite then we define the \emph{geodesic} $[y_1, y_2]$ in $\Delta_G(\KK)$ to be $\zeta_G([x_1,x_2])$; this does not depend on the choice of $x_1$ and $x_2$.  We say a subset $\Sigma$ of $\Delta_G(\KK)$ is \emph{convex} if for all $y_1, y_2\in \Sigma$ such that $[y_1, y_2]$ exists, $[y_1, y_2]\subseteq \Sigma$.
\end{rem}

\subsection{The combinatorial edifice}
\begin{defn}\label{defn:combinatorial}
We define the \emph{combinatorial edifice} $\Delta_{G}$ to be the poset formed by the R-parabolic subgroups of $G$
under reverse inclusion.
Given $\lambda\in Y_G(\KK)$, $P = P_\lambda$ and $x=\phi_G(\lambda)\in V_G(\KK)$, we denote the corresponding element of $\Delta_G$ by 
$\sigma_P$, $\sigma_\lambda$ or $\sigma_x$,
and we write $\leq$ for the partial order on $\Delta_G$: i.e., $\sigma_Q\leq \sigma_P$ if and only if $P\subseteq Q$.
We let $\varnothing$ denote the element in this poset corresponding to $G$ itself.
\end{defn}

The elements of the combinatorial edifice can be realised geometrically inside $V_G(\KK)$: 
given an R-parabolic subgroup $P$ of $G$, consider the subset
$$
\widetilde{\sigma}_P = \{x\in V_G(\KK) \mid P_x= P\}.
$$
(We also define $\widetilde{\sigma}_\lambda:= \widetilde{\sigma}_{P_\lambda}$ and $\widetilde{\sigma}_x:= \widetilde{\sigma}_{P_x}$.)  By choosing a maximal split torus $T$ of $P$, we may realise 
$\widetilde{\sigma}_P$ as a subset of $V_T(\KK)$ (Lemma~\ref{lem:all_tor}), so it makes sense to take linear combinations of elements in $\widetilde{\sigma}_P$.

\begin{lem}
\label{lem:sx_cvx_cone}
 Let $x,y\in \widetilde{\sigma}_P$.  Then $ax+ by\in \widetilde{\sigma}_P$ for all $a,b\in \KK^+$.
\end{lem}

\begin{proof}
 Fix a maximal split torus $T$ of $P$.   We can choose $\lambda, \mu\in Y_T(\KK)$ such that $x= \phi_G(\lambda)$ and $y= \phi_G(\mu)$; then $P_\lambda= P_\mu= P$.  Recall that $P_\nu= \rho^{-1}(V_{\nu, \geq 0})$ for any $\nu\in Y_T(\KK)$ (Remark~\ref{rem:L_robust}(iii)).  It follows that $P_{a\lambda}= P_\lambda$ and $P_{b\mu}= P_\mu$ for any $a,b> 0$.  So it is enough to show that $P_{\lambda+ \mu}= P$.
 
 Fix a $G$-equivariant embedding $\rho'$ of $G$ in a rational $G$-module $V'$.  Let $V$ be the rational $G$-module $V'\oplus \fg$ and let $\rho$ be $\rho'$ followed by the obvious inclusion of $V'$ in $V$.  Since $P_\nu= \rho^{-1}(V_{\nu, \geq 0})$ for any $\nu\in Y_T(\KK)$, the argument of the first part of the proof of Lemma~\ref{lem:sum} shows that $P_{\lambda+ \mu}\supseteq P$.
 
 Let $\nu\in Y_T(\KK)$ and let $\nu'\in Y_T$ be a cocharacter approximation to $\nu$.  By our choice of $\rho$ we have $\fg_{\nu, \geq 0}= \fg_{\nu', \geq 0}$, and we deduce that
 $$ \Lie(P_\nu)= \Lie(P_{\nu'})= \fg_{\nu', \geq 0}= \fg_{\nu, \geq 0}, $$
 where the middle equality is from \cite[Prop.\ 2.1.8]{CGP}.  The argument of the second part of the proof of Lemma~\ref{lem:sum} shows that $\Lie(P_{\lambda+ \mu})= \Lie(P)$, and we deduce that $P_{\lambda+ \mu}= P$, as required.
\end{proof}

It is well known that if $G$ is semisimple then $\Delta_G$ is a spherical building; in particular, it is a simplicial complex.  The same is true for reductive $G$ because the parabolic subgroups of $G$ are in bijective correspondence with the parabolic subgroups of $[G,G]$.  The following example shows that the combinatorial edifice need not be a simplicial complex for arbitrary $G$, and the partial order on $\Delta_G$ need not be realised geometrically by relationships between closures of the corresponding subsets in $V_G(\KK)$.

\begin{exmp}\label{exmp:notsimplicial}
Let $k$ be algebraically closed, and set $G = \GL_2\ltimes V$, where $V$ is the natural two-dimensional module for $\GL_2$.
Let $e_1$ and $e_2$ be the usual standard basis vectors in $V$, and let $V_1$ and $V_2$ be the corresponding one-dimensional subspaces of $V$.
Let $B^+$ denote the upper triangular Borel subgroup in $\GL_2$ and $B^-$ the lower triangular Borel subgroup.
Let $T$ be the diagonal maximal torus in $\GL_2$.
Then an element of $Y_T(\KK)$ can be identified with a pair $(a,b)\in \KK^2$, and the 
corresponding R-parabolic subgroup has the form $P\ltimes W$ for some parabolic subgroup $P$ of $\GL_2$ and some subspace $W$ of $V$ as in the following table:\\
 
\begin{center}
\begin{tabular}{c|c||c|c}
Conditions on $(a,b)$ & $P\ltimes W$ & Conditions on $(a,b)$ & $P\ltimes W$ \\ \hline
$a=b\geq0$ & $G$ & $0>a=b$ & $\GL_2$\\
$a>b\geq 0$ & $B^+\ltimes V$ & $0>b>a$ & $B^-$\\
$a\geq 0>b$ & $B^+\ltimes V_1$ & $b\geq0>a$ & $B^-\ltimes V_2$\\
$0>a>b$ & $B^+$ & $b>a\geq 0$ & $B^-\ltimes V$
\end{tabular}
\end{center}
\bigskip

In an (abstract) simplicial complex, we can recognise any element of the complex by the vertices (minimal non-empty elements) it contains. 
In this case, the minimal elements in the poset $\Delta_G$ correspond to the maximal proper R-parabolic subgroups: 
thus, here we get (the conjugates in $G$ of) $\GL_2$, $B^+\ltimes V$ and $B^-\ltimes V$.
Consider $\sigma = \sigma_{B^+\ltimes V_1}\in \Delta_G$.
Any $\varnothing\neq\tau\in \Delta_G$ with $\tau\lneq \sigma$ corresponds to an R-parabolic subgroup $P$ of $G$ with $P\supsetneq B^+\ltimes V_1$.
Such a $P$ must contain $T$, so $P = B^+\ltimes V$ is the only option.
Thus $\tau = \sigma_{B^+\ltimes V}$ and $\sigma$ contain the same minimal elements, and $\Delta_G$ is not a simplicial complex.

We leave it as an exercise for the reader to sketch the regions of the plane $\RR^2$ that correspond to each R-parabolic subgroup and show that (for example) the region $a\geq 0>b$ is not in the closure of the region $a>b\geq 0$, even though we have a containment $B^+\ltimes V_1\subseteq B^+\ltimes V$.
\end{exmp}

\section{Linear maps of vector edifices}
\label{sec:linmapsI}

In this section we define the notion of a linear map between vector edifices, and show how homomorphisms between $k$-groups give rise to such maps.
For some of the later work (especially in Section \ref{sec:fieldexts}) it is important to consider maps between edifices for algebraic groups defined over possibly different fields,
so that is how our definitions are framed in Section \ref{sec:linmapdef}.

\subsection{Definition and first properties}\label{sec:linmapdef}
\begin{defn}\label{defn:linearmap}
Suppose $V_1 = V_{G}(\KK)$ and $V_2 = V_{H}(\KK)$ are two vector edifices,
where $G$ is a $k$-group and $H$ is a $k'$-group for two (possibly different) fields $k$ and $k'$.
A function $\kappa\colon V_1\to V_2$ is called a \emph{linear map of vector edifices}
if for every apartment $A_1$ of $V_1$, there exists an apartment $A_2$ of $V_2$ such that $\kappa(A_1)\subseteq A_2$ and $\kappa|_{A_1}\colon A_1\to A_2$ is a $\KK$-linear map.  
\end{defn}

It is immediate that the composition of linear maps is a linear map.  

\begin{rem}
The reader is warned that the inverse of a bijective linear map of vector edifices is \emph{not} necessarily itself a linear map.
The reason for this is that even if $V_1$ and $V_2$ are in bijection with each other, it might not be the case that their systems of apartments are in bijection with each other.
See Example~\ref{ex:bijectivenotiso} below.
\end{rem}

\begin{defn}\label{defn:iso}
Suppose $V_1 = V_{G}(\KK)$ and $V_2 = V_{H}(\KK)$ are two vector edifices as above.
A linear map $\kappa\colon V_1\to V_2$ is an \emph{isomorphism of vector edifices}
if it is bijective and the inverse map is also a linear map of vector edifices.  If $V_1= V_2$ then we call $\kappa$ an \emph{automorphism of vector edifices}.  We write $\Aut (V)$ for the group of automorphisms of a vector edifice $V$.
\end{defn}

In practice, we often wish to construct linear maps of vector edifices from corresponding maps on cocharacters.
The following result shows that this process can be reversed.

\begin{lem}
\label{lem:lin_map_lift}
Let $G$ be a $k$-group and $H$ a $k'$-group for two fields $k$ and $k'$.
Let $\kappa\colon V_{G}(\KK)\to V_{H}(\KK)$ be a linear map. 
Then $\kappa$ lifts uniquely to a map $$\widehat{\kappa}\colon Y_{G}(\KK)\to Y_{H}(\KK)$$ with the following property: for any maximal split torus $T$ of $G$ and for any maximal split torus $S$ of $H$ such that $\kappa(V_{T}(\KK))\subseteq V_{S}(\KK)$, we have $\widehat{\kappa}(Y_T(\KK))\subseteq Y_S(\KK)$, and the following diagram commutes:
 $$
 \xymatrixcolsep{5pc}
 \xymatrix{
 	Y_{T}(\KK)\ar[d]_{\phi_{G}}\ar[r]^{\widehat{\kappa}}& Y_{S}(\KK)\ar[d]^{\phi_{H}}\\
 	V_{T}(\KK)\ar[r]^{\kappa}& V_{S}(\KK)\\}
 $$
\end{lem}

\begin{proof}
First off, it follows from Remark~\ref{rem:eqvce_props}(ii) that $\widehat{\kappa}$ is unique if it exists.  

Let $\lambda\in Y_{G}(\KK)$, and set $x= \phi_G(\lambda)$ and $y= \phi_G(-\lambda)$.  
Since $\kappa$ is linear on apartments, $\kappa(x)+ \kappa(y)= \kappa(x+y)= \kappa(0)= 0$, so $\kappa(x)$ and $\kappa(y)$ are opposite points in (the corresponding apartment of) $V_H(\KK)$.
By Lemma~\ref{lem:uniqueopp} we can find a unique $\mu\in Y_{H}(\KK)$ such that $\phi_H(\mu)= \kappa(x)$, $\phi_H(-\mu) = \kappa(y)$.  
We set $\widehat{\kappa}(\lambda) := \mu$.

Now let $T$ be a maximal split torus of $G$ such that $\lambda\in Y_{T}(\KK)$ and let $S$ be any maximal split torus of $H$ such that $\kappa(V_{T}(\KK))\subseteq V_{S}(\KK)$.  Set $x= \phi_G(\lambda)$ and $y= \phi_G(-\lambda)$.  Since $\kappa(x), \kappa(y)\in V_{S}(\KK)$ by construction and $\kappa(x)+ \kappa(y)= 0$, there exists a unique $\nu\in Y_{S}(\KK)$ such that $\phi_H(\nu)= \kappa(x)$ and $\phi_H(-\nu)= \kappa(y)$.  
We see that $\widehat{\kappa}(\lambda) = \mu = \nu \in Y_{S}(\KK)$, so that $\widehat{\kappa}(Y_T(\KK))\subseteq Y_S(\KK)$, and the commutativity of the given diagram follows.
\end{proof}

\begin{rem}\label{rem:linearfunctorial}
 It is immediate from the construction that $\widehat{\kappa}$ restricts to a linear map from each $Y_T(\KK)$ to the corresponding $Y_S(\KK)$.  
 We have obvious functoriality: $\widehat{\rm id}= {\rm id}$ and $\widehat{\tau\circ \kappa}= \widehat{\tau}\circ \widehat{\kappa}$.
\end{rem}

\subsection{Linear maps induced by group homomorphisms}\label{subsec:linmapsII}
Let $f\colon G\to H$ be a homomorphism of connected $k$-groups.  We show that $f$ gives rise to a linear map 
$\kappa_f \colon V_{G}(\KK)\to V_{H}(\KK)$.

\begin{lem}
\label{lem:max_inpndce}
 Let $T$ be a maximal split torus of $G$ and let $S_1$ and $S_2$ be maximal split tori of $H$ such that $f(T)\subseteq S_1\cap S_2$.  
 Let $f_1\colon T\to S_1$ and $f_2\colon T\to S_2$ be the maps induced by $f$.  
 Let  $\lambda\in Y_{T}(\KK)$ and set $\lambda_1= (f_1)_*(\lambda)$, $\lambda_2= (f_2)_*(\lambda)$.  
 Then $(S_1,\lambda_1)\sim (S_2,\lambda_2)$.
\end{lem}

\begin{proof}
By Remark~\ref{rem:L_robust}(ii), $S_1$ and $S_2$ are maximal split tori of $L_{\lambda_1}$, 
and hence $lS_1l^{-1}= S_2$ for some $l\in L_{\lambda_1}(k)$.  The conjugation map $\Inn_l$ gives rise to a linear map $\omega$ from $Y_{S_1}(\KK)$ to $Y_{S_2}(\KK)$, and we have $\omega\circ (f_1)_*= (f_2)_*$ by functoriality (Section~\ref{sec:linmapstori}).  The result follows.
\end{proof}

For each maximal split torus $T$ of $G$, choose a maximal split torus $S_T$ of $H$ such that $f(T)\subseteq S_T$.  
The maps $f_*\colon Y_{T}(\KK)\to Y_{S_T}(\KK)$ for each $T$ give rise to a map from $\bigsqcup_T Y_{T}(\KK)$ to $\bigsqcup_S Y_{S}(\KK)$.  Consider the compositions $\bigsqcup_T Y_{T}(\KK)\to \bigsqcup_S Y_{S}(\KK)\stackrel{\varpi_H}{\to} Y_{H}(\KK)$ and $\bigsqcup_T V_{T}(\KK)\to \bigsqcup_S V_{S}(\KK)\stackrel{\omega_H}{\to} V_{H}(\KK)$.  By Lemma~\ref{lem:max_inpndce}, these maps do not depend on the choices of the maximal split tori $S_T$ for each~$T$.

\begin{prop}
\label{prop:gphom}
 The maps above descend to give well-defined maps $Y_{G}(\KK)\to Y_{H}(\KK)$ and 
 $$\kappa_f\colon V_{G}(\KK)\to V_{H}(\KK).$$
 Moreover, $\kappa_f$ is a linear map of vector edifices.
\end{prop}

\begin{proof}
 We give the proof for $\kappa_f$.  
 Suppose $(T_1,\lambda_1)\approx (T_2,\lambda_2)$ for $(T_1,\lambda_1), (T_2,\lambda_2)\in \bigsqcup_T Y_{T}(\KK)$.  
 Then $T_2= gT_1g^{-1}$ and $g\cdot \lambda_1= \lambda_2$ for some $g\in P_{\lambda_1}(k)$.  Let $f_i\colon T_i\to S_{T_i}$ be the map induced by $f$.  It is enough to prove that $(S_{T_1}, (f_1)_*(\lambda_1))\approx (S_{T_2}, (f_2)_*(\lambda_2))$.  
 By Lemma~\ref{lem:max_inpndce}, there is no harm in taking $S_{T_2}$ to be $f(g)S_{T_1}f(g)^{-1}$.  
 Observe
 that this means that $f_2\circ \Inn_g =\Inn_{f(g)} \circ f_1$. 
 It follows from functoriality
 that 
 $$
 f(g)\cdot (f_1)_*(\lambda_1)= (f_2)_*(g\cdot \lambda_1) = (f_2)_*(\lambda_2).
 $$  
 Since $f(g)$ belongs to $P_{(f_1)_*(\lambda_1)(H)}(k)$ by Lemma~\ref{lem:par_levi_f_prop}(i), we obtain the desired result.
 It is also clear that $\kappa_f$ is linear --- for each maximal split torus $T$ of $G$, the corresponding map $f_*\colon Y_{T}(\KK)\to Y_{S_T}(\KK)$
 is $\KK$-linear.

 The proof for the map $Y_{G}(\KK)\to Y_{H}(\KK)$ is very similar: it just uses a conjugating element $l\in L_{\lambda_1}(k)$ instead.
\end{proof}

\begin{rems}
\label{rems:funct}
 (i). If $\lambda\in Y_G$ then $\widehat{\kappa_f}(\lambda)= f\circ \lambda$.  It is clear that $\kappa_{{\rm id}_G}= {\rm id}_{V_G(\KK)}$ and if $f\colon G\to H$ and $l\colon H\to K$ are homomorphisms of connected $k$-groups then $\kappa_{l\circ f}= \kappa_l\circ \kappa_f$.  In particular, if $f\in \Aut (G)$ then $\kappa_f$ belongs to $\Aut (V_G(\KK))$ and has inverse $\kappa_{f^{-1}}$.
  
 (ii). Let $f\colon G\to H$ be a surjective homomorphism of connected $k$-groups and let $x\in V_G(\KK)$.  It follows from Lemma~\ref{lem:par_levi_f_prop}(i) that $P_{\kappa_f(x)}= f(P_x)$.  
\end{rems}

\begin{exmp}
\label{ex:parconj}
 Let $g\in G(k)$.  Untangling the definitions, we see that $\kappa_{\Inn_g}(x)$ is just $g\cdot x$ in the sense of Section~\ref{subsec:vector_edifice}.  By Remark~\ref{rems:funct}(i), $\kappa_{\Inn_g}$ is invertible with inverse $\kappa_{\Inn_{g^{-1}}}$.  Remark~\ref{rems:funct}(ii) implies that $P_{g\cdot x}= P_{\kappa_{\Inn_g}(x)}= \Inn_g(P_x)= gP_xg^{-1}$.
\end{exmp}
 
\begin{rem}
\label{rem:type-preserving}
 Let $f\colon G\to H$ be a homomorphism.  Then $\kappa_f$ is equivariant in the following sense: if $x\in V_G(\KK)$ and $g\in G(k)$ then $\kappa_f(g\cdot x)= f(g)\cdot \kappa_f(x)$.  To see this, note that if $\lambda\in Y_G$ and $g\in G(k)$ then $f\circ (g\cdot \lambda)= f(g)\cdot (f\circ \lambda)$; the result follows from linearity.  This shows that $\kappa_f$ is type-preserving (see Definition~\ref{defn:type}).
\end{rem}

\subsection{Homomorphisms which induce injective linear maps}\label{sec:homs}
We can now show that inclusions of $k$-groups give inclusions of vector edifices. In fact, we can do better.

\begin{prop}
\label{prop:inclusion}
 Let $f\colon G\to H$ be a homomorphism of connected $k$-groups with finite kernel. 
 Then the maps $Y_{G}(\KK)\to Y_{H}(\KK)$ and $\kappa_f\colon V_{G}(\KK)\to V_{H}(\KK)$ are injective.
\end{prop}

\begin{proof}
 Again, we give the proof for $\kappa_f$; the other proof is very similar.  
 Let $(T_1,\lambda_1), (T_2,\lambda_2)\in \bigsqcup_T Y_{T}(\KK)$ be such that 
 $(S_{T_1}, f_*(\lambda_1))\approx (S_{T_2}, f_*(\lambda_2))$.  Let $f_i\colon T_i\to S_{T_i}$ be the map induced by $f$ for each $i$.
 Then $P_{(f_1)_*(\lambda_1)}(H) = P_{(f_2)_*(\lambda_2)}(H)$, so $P_{(f_1)_*(\lambda_1)}(f(G)) = P_{(f_2)_*(\lambda_2)}(f(G))$ by Corollary~\ref{cor:par_intersect}, and hence by Lemma~\ref{lem:par_levi_f_prop}(i) we have $f(P_{\lambda_1}) = f(P_{\lambda_2})$.  It follows from Lemma~\ref{lem:par_levi_f_prop}(ii) (applied to the surjective homomorphism $G\to f(G)$) that 
$P_{\lambda_1} = P_{\lambda_2}$.  Thus $T_1$ and $T_2$ are maximal split tori of $P_{\lambda_1}$, which means we can find 
 $g\in P_{\lambda_1}(k)$ such that $gT_2g^{-1} = T_1$. 
 Since $f(g) \in P_{(f_1)_*(\lambda_1)}(H)(k) = P_{(f_2)_*(\lambda_2)}(H)(k)$, we have $P_{(f_2)_*(g\cdot \lambda_2)}(H)= P_{f(g)\cdot (f_2)_*(\lambda_2)}(H)= f(g)P_{(f_2)_*(\lambda_2)}(H) f(g)^{-1}= P_{(f_2)_*(\lambda_2)}(H)$ by Remark~\ref{rem:type-preserving} and Lemma~\ref{lem:par_levi_prop}(ii).  Hence, replacing $\lambda_2$ with $g\cdot\lambda_2$, we may assume that $T_1 = T_2$. 
 Now, since $f$ has finite kernel, the map $(f_1)_*$ is injective on $Y_{T_1}(\KK)$, as observed in Section~\ref{sec:linmapstori}. 
 Since we have reduced to the case that $\lambda_1,\lambda_2\in Y_{T_1}(\KK)$ with 
 $(f_1)_*(\lambda_1) = (f_1)_*(\lambda_2)$, we conclude that $\lambda_1 = \lambda_2$. 
 \end{proof}

\begin{rem}\label{rem:notinconsistent}
We can now resolve the potential ambiguity noted in Remark~\ref{rem:inconsistent} above.
 If $T$ is a maximal split torus of $G$ then the notation $V_{T}(\KK)$ could mean the corresponding apartment of $V_{G}(\KK)$ or it could mean the vector edifice associated to $T$ in its own right.  But we see that if $i\colon T\to G$ is the inclusion then $\kappa_i$ gives a linear isomorphism from the latter object to the former, so the abuse of notation is harmless.
 Indeed, this observation works for any subgroup $H$ of $G$ --- we may identify the vector edifice $V_H(\KK)$ as a subobject of the vector edifice $V_G(\KK)$ via the injective linear map $\kappa_i\colon V_H(\KK)\to V_G(\KK)$ arising from the inclusion $i\colon H\to G$.
 We often do this without further comment in the rest of the paper.
\end{rem}

\begin{rem}
\label{rem:xinVTK}
 If $\lambda\in Y_G(\KK)$ then we may regard $\lambda$ as an element of $Y_{L_\lambda}(\KK)$ and $x:= \phi_G(\lambda)$ as an element of $V_{L_\lambda}(\KK)$ and $V_{P_\lambda}(\KK)$ in the obvious way.  
Since all maximal split tori in $L_\lambda$ are $L_\lambda(k)$-conjugate, it follows that if $T$ is a maximal split torus of $L_\lambda$ then $x\in V_T(\KK)$.
\end{rem}

\begin{exmp}\label{ex:bijectivenotiso}
Suppose that $G$ is such that $V_G(\KK)$ has the common apartment property, and $P$ is an R-parabolic subgroup of $G$.
Then the inclusion $i\colon P\to G$ actually induces a \emph{bijection} $\kappa_i\colon V_P(\KK)\to V_G(\KK)$.
To see this, note that for every maximal split torus $T$ of $G$ and every $\lambda\in Y_T(\KK)$, there is a maximal split torus $S_{T,\lambda}$ contained in $P_\lambda\cap P$, because of the common apartment property, 
and hence $(T,\lambda)\approx (S_{T,\lambda},\mu)$, where $\mu = g\cdot\lambda$ for $g\in P_\lambda(k)$ conjugating $T$ to $S_{T,\lambda}$.
Therefore, if we let $x = \phi_P(\mu)\in V_P(\KK)$ and $y = \phi_G(\lambda)\in V_G(\KK)$, we have $\kappa_i(x) = y$.
Thus $\kappa_i$ is surjective, and we already know it is injective by Proposition~\ref{prop:inclusion}.  Note that, even though $\kappa_i$ is a bijective linear map of vector edifices, it is \emph{not} an isomorphism in general: e.g., in the setting of Example~\ref{ex:no_common_apt} we have points $x_1, x_2\in V_P(\KK)$ such that $x_1$ and $x_2$ do not lie in a common apartment of $V_P(\KK)$, but $\kappa_i(x_1)$ and $\kappa_i(x_2)$ do lie in a common apartment of $V_G(\KK)$, because $G$ is reductive in that case (and hence $V_G(\KK)$ has the common apartment property).
\end{exmp}

\subsection{Homomorphisms which induce surjective linear maps}
Next we  
investigate the conditions under which a homomorphism $f\colon G\to H$ gives rise to a surjective linear map.  It turns out that the obvious condition (that $f$ is surjective) is not sufficient; this situation is made more complicated by, for example, the existence over imperfect fields of inseparable isogenies.  This motivates the following definition. See also Remark~\ref{rem:isogenydiscussion}(i) below.

\begin{defn}
Suppose $f\colon G\to H$ is a homomorphism of connected $k$-groups.
We say $f$ is \emph{apte} if every maximal split torus $S$ of $H$ has the form $S = f(T)$ for a maximal split torus $T$ of $G$.
\end{defn}

With the key definition in hand, the following result is easy.

\begin{lem}
\label{lem:surj}
Let $f\colon G\to H$ be an apte homomorphism of connected $k$-groups.  Then $\kappa_f$ is surjective.
\end{lem}

\begin{proof}
 Let $S$ be a maximal split torus of $H$.  The assumption that $f$ is apte means we can find a maximal split torus $T$ of $G$ with $S=f(T)$.  Let $h\colon T\to S$ be restriction of $f$.  Since the map of character groups $h^*\colon X_S\to X_T$ is injective, the dual map $h_*\colon Y_T\to Y_S$ has finite cokernel.   This implies that $f_*\colon Y_{T}(\KK)\to Y_{S}(\KK)$ is surjective, and the result follows.
\end{proof}

\begin{rems}\label{rem:isogenydiscussion}
(i). The lemma does not hold without the assumption that $f$ is apte. For a concrete example, let $G = H = \SL_2$ over an imperfect field $k$ in characteristic $2$, and let $f\colon G\to G$ be the standard Frobenius map which squares the entries of $2\times 2$ matrices.
That $k$ is imperfect means that there is some $a\in \ovl{k}\setminus k$ such that $a^2\in k$.
Now let 
$$
u = \left(\begin{array}{cc} 1&a\\0&1\end{array}\right) \in G(\ovl{k}),
$$
and let $T$ be the standard (split) diagonal torus in $G$.
The torus $uT_{\ovl{k}}u^{-1}$ is not a $k$-subgroup, hence in particular is not $k$-split.
However, its image $f_{\ovl{k}}(uT_{\ovl{k}}u^{-1}) = f(u)T_{\ovl{k}}f(u)^{-1}$ descends to a split $k$-torus $S= f(u)Tf(u)^{-1}$, since $f(u)\in G(k)$.
This shows that $f$ is not apte, and it is also clear that not every element of  $V_S(\KK)$ belongs to $\kappa_f(V_G(\KK))$.

(ii). Following \cite[22.11]{Bo}, we call a surjective homomorphism $f\colon G\to H$ a \emph{central isogeny} if it has finite kernel which
is contained in the \emph{scheme-theoretic} centre of $G$ --- that is, for every $k$-algebra $A$ and every $A$-algebra $A'$, the kernel of $f(A)\colon G(A)\to H(A)$ centralises $G(A')$.\footnote{For $G$ reductive, Borel \cite[22.3]{Bo} gives another definition of central isogeny, but this is shown to be equivalent to the scheme-theoretic one just given in \cite[22.15 Prop.]{Bo}.}  If (a) $f$ is surjective and $k$ is perfect or (b) $f$ is a central isogeny, then $f$ is apte: this follows from \cite[22.6 Thm.(ii)]{Bo}.  Neither of these hypotheses holds for the example in (i).

(iii). Let $f$ be a surjective homomorphism.  Suppose $f$ is smooth (this is equivalent to requiring $\ker f$ to be smooth, by \cite[Prop.~1.63]{milne}).  Then $f$ is apte.  To see this, let $S$ be a maximal split torus of $H$ and let $S_1$ be a maximal torus of $H$ containing $S$. 
 Then the subgroup scheme $f^{-1}(S_1)$ of $G$ is smooth.  Let $T_1$ be a maximal torus of $f^{-1}(S_1)$ and let $T$ be the unique maximal split torus of $T_1$.  
 By \cite[A.2.8]{CGP}, $f(T_1)= S_1$, so (the restriction of) $f$ is a surjective map of $k$-tori.
 It follows from \cite[8.15 Prop.(3)]{Bo} and surjectivity that $f(T) = S$.

(iv). Example~\ref{ex:bijectivenotiso} shows that $\kappa_f$ can be surjective even when $f$ is not apte.
\end{rems}

We can now prove the key result of this subsection:

\begin{prop}\label{prop:isogiso}
Suppose $f\colon G\to H$ is an apte homomorphism with finite kernel.  Then $\kappa_f$ is an isomorphism of vector edifices.
\end{prop}

\begin{proof}
It follows from Proposition~\ref{prop:inclusion} and Lemma~\ref{lem:surj} that $\kappa_f$ is a bijective linear map.
Since $f$ is apte, for each maximal split torus $S$ of $H$ there is a maximal split torus $T$ of $G$ such that $f(T) = S$, and since $f$ has finite kernel, this $T$ is unique.
Therefore, $\kappa_f$ pairs up the apartments of $V_G(\KK)$ and $V_H(\KK)$, and hence $\kappa_f^{-1}$ is also a linear map of vector edifices.  
\end{proof}

We finish with some results related to the common apartment property.

\begin{lem}
\label{lem:unipt_quot}
  Let $N$ be a connected normal unipotent subgroup of $G$ and let $\pi\colon G\to G/N$ be the canonical projection.  Let $x_1,x_2\in V_G(\KK)$.  Then $\kappa_\pi(x_1)= \kappa_\pi(x_2)$ if and only if $x_2= n\cdot x_1$ for some $n\in N(k)$.
\end{lem}

\begin{proof}
 If $n\in N(k)$ and $y= n\cdot x$ then $\kappa_\pi(y)= \kappa_\pi(n\cdot x)= \pi(n)\cdot \kappa_\pi(x)= \kappa_\pi(x)$ where the second equality is from Remark~\ref{rem:type-preserving}.  Conversely, suppose $\kappa_\pi(x_1)= \kappa_\pi(x_2)$; call this common value $y$.  For each $i$, let $G_i= P_{x_i}N$, let $Q_i= \pi(G_i)$ and let $\pi_i\colon G_i\to Q_i$ be the restriction of $\pi$.  Then $\pi_i$ is surjective, so $Q_i= \pi(P_{x_i})= P_y(G/N)$ by Remark~\ref{rems:funct}(ii).
 
 Choose a maximal split torus $T$ of $G/N$ such that $y\in V_T(\KK)$.  Then $T$ is a maximal split torus of $P_y(G/N)$ by Remark~\ref{rem:conjugateapartments}.  For each $i$, since $\pi_i$ is smooth, $\pi_i$ is apte (Remark~\ref{rem:isogenydiscussion}(iii)), so there is a maximal split torus $T_i$ of $G_i$ such that $\pi(T_i)= T$.  So $T_1$ and $T_2$ are maximal split tori of $H:= T_1N\cong T_1\ltimes N$.  Hence there exists $n\in N(k)$ such that $T_2= nT_1n^{-1}$.  Now $x_i$ belongs to $V_{T_i}(\KK)$ for each $i$, so $n\cdot x_1$ and $x_2$ belong to $V_{T_2}(\KK)$.  We have $\kappa_\pi(n\cdot x_1)= \kappa_\pi(x_2)= y$ by Remark~\ref{rem:type-preserving}.  But $\pi$ induces an isomorphism from $T_2$ onto $T$ since $N$ is unipotent, so $\kappa_\pi$ induces an isomorphism from $V_{T_2}(\KK)$ to $V_T(\KK)$.  We conclude that $x_2= n\cdot x_1$, as required.
\end{proof}

\begin{lem}
\label{lem:almost_common_apt}
Let $x,y\in V_G(\KK)$.  Then there exists $u\in R_{u,k}(G)(k)$ such that $u\cdot x$ and $y$ lie in a common apartment.
\end{lem}

\begin{proof}
 Let $H= G/R_{u,k}(G)$ and let $\pi\colon G\to H$ be the canonical projection.  Since $H$ is pseudo-reductive, there is a maximal split torus $S$ of $H$ such that $\kappa_\pi(x), \kappa_\pi(y)\in V_S(\KK)$ (see the discussion following Definition \ref{defn:common_apt}).  By
 Remark~\ref{rem:isogenydiscussion}(iii), as $\pi$ is smooth there is a maximal split torus $T$ of $G$ such that $\pi(T)= S$.  The map $\kappa_\pi$ gives an isomorphism from $V_T(\KK)$ to $V_S(\KK)$, so there exists $x', y'\in V_T(\KK)$ such that $\kappa_\pi(x')= \kappa_\pi(x)$ and $\kappa_\pi(y')= \kappa_\pi(y)$.  By Lemma~\ref{lem:unipt_quot}, there exist $v,w\in R_{u,k}(G)(k)$ such that $v\cdot x= x'$ and $w\cdot y= y'$.  Set $u:= w^{-1}v$.  Then $u\cdot x= w^{-1}\cdot x'$ and $y= w^{-1}\cdot y'$ lie in the common apartment $V_{w^{-1}Tw}(\KK)$, as required. 
\end{proof}

\begin{rem}
 Even though the common apartment property may fail to hold, we can show the following: if $x,y\in V_G(\KK)$ then there exists $z\in V_G(\KK)$ such that $x$ and $z$ lie in a common apartment, and $y$ and $z$ lie in a common apartment.  For instance, we can take $z= 0$, although this is slightly unsatisfactory as 0 does not yield a point in $\Delta_G(\KK)$.  More generally, we can take $z$ to be any element of $V_{Z(G)^0}(\KK)$.  Here is a construction which does not involve $Z(G)^0$.
  
 Let $U= R_{u,k}(G)$.  By Lemma~\ref{lem:almost_common_apt}, there exist a maximal split torus $T$ of $G$ and $u\in U(k)$ such that $u\cdot x, y\in V_T(\KK)$.  We can pick $\lambda\in Y_{u^{-1}Tu}(\KK)$ such that $x= \phi_G(\lambda)$ and $u\cdot x= \phi_G(u\cdot \lambda)$.  Note that $x$ and $\phi_G(-\lambda)$ lie in the common apartment $V_{u^{-1}Tu}(\KK)$, and $u\cdot \lambda$ and $-u\cdot \lambda$ both belong to $Y_T(\KK)$.  By \cite[13.4.4 Cor.]{spr2}, there exist $v\in P_\lambda(U)(k)$ and $w\in U_{-\lambda}(U)(k)$ such that $u= vw$.  Let $z= v\cdot \phi_G(-\lambda)$.  Then $x$ and $z$ lie in the common apartment $V_{vu^{-1}Tuv^{-1}}(\KK)$ since $v\cdot x= x$.  Now $w\cdot \phi_G(-\lambda)= \phi_G(-\lambda)$, so $z= vw\cdot \phi_G(-\lambda)= u\cdot \phi_G(-\lambda)= \phi_G(-u\cdot \lambda)$ belongs to $V_T(\KK)$.  Hence $z$ and $y$ lie in a common apartment, as required.
\end{rem}

\subsection{Isomorphisms of vector edifices induce isomorphisms of combinatorial edifices}
Our next aim is to prove that isomorphisms between vector edifices behave nicely with respect to the structure coming from the poset of R-parabolic subgroups under reverse inclusion.
Now that we have shown that for a subgroup $M$ of $G$ we can identify the sets $Y_M(\KK)$ and $V_M(\KK)$ inside $Y_G(\KK)$ and $V_G(\KK)$, we can extend Lemma~\ref{lem:samecocharssameparabolic}, which is the equivalence (i) $\iff$ (ii) in the following.

\begin{lem}\label{lem:recognise}
Let $P, Q$ be R-parabolic subgroups of $G$.
Then the following are equivalent:
\begin{itemize}
\item[(i)] $P\subseteq Q$;
\item[(ii)] $Y_P\subseteq Y_Q$;
\item[(iii)] $Y_P(\KK)\subseteq Y_Q(\KK)$.
\end{itemize}
\end{lem}

\begin{proof}
(i)$\iff$(ii) follows from Lemma~\ref{lem:samecocharssameparabolic} and that (iii)$\implies$(ii)
is easy to see, basically by construction of these objects.  
Suppose (ii) holds. Let $T$ be a maximal split torus of $P$. Because of the equivalence of (i) and (ii), we have
$T\subseteq Q$, so $Y_T(\KK)\subseteq Y_Q(\KK)$.  This shows that (ii)$\implies$(iii).
\end{proof}

\begin{lem}\label{lem:lin_map_lift2}
Let $\kappa\colon V_{G}(\KK)\to V_{H}(\KK)$ be a linear map, and let $\widehat{\kappa}\colon  Y_G(\KK)\to Y_H(\KK)$ be the corresponding lift from Lemma \ref{lem:lin_map_lift}.
Then for any $\lambda\in Y_{G}(\KK)$, we have 
$$
\widehat{\kappa}(Y_{L_\lambda}(\KK))\subseteq Y_{L_{\widehat{\kappa}(\lambda)}}(\KK)\quad  \textrm{ and } \quad  \widehat{\kappa}(Y_{P_\lambda}(\KK))\subseteq Y_{P_{\widehat{\kappa}(\lambda)}}(\KK).
$$
Hence also $\kappa(V_{L_\lambda}(\KK)) \subseteq V_{L_{\widehat{\kappa}(\lambda)}}(\KK)$ and $\kappa(V_{P_\lambda}(\KK))\subseteq V_{P_{\widehat{\kappa}(\lambda)}}(\KK)$. 
Thus, if we set $x=\phi_G(\lambda)$, we have $\kappa(V_{P_x}(\KK))\subseteq V_{P_{\kappa(x)}}(\KK)$.
\end{lem}

\begin{proof}
Choose any $z\in V_{L_\lambda}(\KK)$: say, $z= \phi_{L_\lambda}(\sigma)$ for some $\sigma\in Y_{L_\lambda}(\KK)$. 
Choose any maximal split torus $T$ of $L_\lambda$ such that $\sigma\in Y_{T}(\KK)$.  Then $\lambda\in Y_{T}(\KK)$, by Remark \ref{rem:xinVTK}, and $T$ is a maximal split torus of $G$.  
Let $S$ be any maximal split torus of $H$ such that $\kappa(V_{T}(\KK))\subseteq V_{S}(\KK)$.  Then $\widehat{\kappa}(\lambda)$ and $\widehat{\kappa}(\sigma)$ belong to $Y_{S}(\KK)$.
Thus $S \subseteq L_{\widehat{\kappa}(\lambda)}$ and so 
	$Y_{S}(\KK)\subseteq  Y_{L_{\widehat{\kappa}(\lambda)}}(\KK)$.
This shows that $\widehat{\kappa}(Y_{L_\lambda}(\KK))\subseteq Y_{L_{\widehat{\kappa}(\lambda)}}(\KK)$, by Lemma \ref{lem:recognise}.
 
 For the second statement, let $w\in V_{P_\lambda}(\KK)$: say, $w= \phi_{P_\lambda}(\tau)$ for some $\tau\in Y_{P_\lambda}(\KK)$.  Set $x= \phi_G(\lambda)$.  Choose a maximal split torus $T$ of $P_\lambda$ such that $\tau\in Y_{T}(\KK)$.  By Lemma~\ref{lem:basicpropertiesRpars}(i) and Remark~\ref{rem:obvious_gen}, there exists $u\in U_\lambda(k)$ such that $uTu^{-1}\subseteq L_\lambda$.  Then $\lambda\in Y_{uTu^{-1}}(\KK)$, again by Remark \ref{rem:xinVTK},
 so $u^{-1}\cdot \lambda\in Y_{T}(\KK)$, so $T\subseteq L_{u^{-1}\cdot \lambda}$ by Remark~\ref{rem:L_robust}(iii), so $\tau\in Y_{L_{u^{-1}\cdot \lambda}}(\KK)$.
 By the previous paragraph, we can conclude that $\widehat{\kappa}(\tau)\in Y_{L_{\widehat{\kappa}(u^{-1}\cdot \lambda)}}(\KK)\subseteq Y_{P_{\widehat{\kappa}(u^{-1}\cdot \lambda)}}(\KK)$.  
 But $P_{\widehat{\kappa}(u^{-1}\cdot \lambda)}= P_{\kappa(u^{-1}\cdot x)}= P_{\kappa(x)} = P_{\widehat{\kappa}(\lambda)}$ by construction, so $\widehat{\kappa}(\tau)\in Y_{P_{\widehat{\kappa}(\lambda)}}(\KK)$, as required. 

Now the corresponding statements on the level of vector edifices follow immediately.
\end{proof}

\begin{lem}\label{lem:autsarecombinatorial}
Let $G$ and $H$ be connected $k$-groups and let 
$\kappa\colon V_{G}(\KK)\to V_{H}(\KK)$ be an isomorphism of vector edifices. 
Then $\kappa$ induces a poset isomorphism $\sigma_x\mapsto \sigma_{\kappa(x)}$ between the corresponding combinatorial edifices.
\end{lem}

\begin{proof}
 Let $\tau\in Y_G(\KK)$.  Then $\widehat{\kappa}(Y_{P_\tau}(\KK))\subseteq Y_{P_{\widehat{\kappa}(\tau)}}(\KK)$ by Lemma~\ref{lem:lin_map_lift2}.  On the other hand, $\widehat{\kappa}^{-1}(Y_{P_{\widehat{\kappa}(\tau)}}(\KK))\subseteq Y_{P_{\widehat{\kappa}^{-1}(\widehat{\kappa}(\tau))}}(\KK)= Y_{P_\tau}(\KK)$ using Lemma~\ref{lem:lin_map_lift2} again, so we deduce that
 \begin{equation}
 \label{eqn:kappa}
  \widehat{\kappa}(Y_{P_\tau}(\KK))= Y_{P_{\widehat{\kappa}(\tau)}}(\KK)
 \end{equation}
 for all $\tau\in Y_G(\KK)$.
 
 Now let $\lambda, \mu\in Y_G(\KK)$ such that $P_\lambda\subseteq P_\mu$.  Then $Y_{P_\lambda}(\KK)\subseteq Y_{P_\mu}(\KK)$ by Lemma~\ref{lem:recognise}, so $Y_{P_{\widehat{\kappa}(\lambda)}}(\KK)= \widehat{\kappa}(Y_{P_\lambda}(\KK))\subseteq \widehat{\kappa}(Y_{P_\mu}(\KK))= Y_{P_{\widehat{\kappa}(\mu)}}(\KK)$ using \eqref{eqn:kappa}.  We deduce from Lemma~\ref{lem:recognise} that
 \begin{equation}
 \label{eqn:kappa2}
  P_{\widehat{\kappa}(\lambda)}\subseteq P_{\widehat{\kappa}(\mu)}.
 \end{equation}
 In particular, if $P_\lambda= P_\mu$ then $P_{\widehat{\kappa}(\lambda)}= P_{\widehat{\kappa}(\mu)}$, so the map $\sigma_x\mapsto \sigma_{\kappa(x)}$ is well-defined.  Clearly, this map is invertible (with inverse $\sigma_w\mapsto \sigma_{\kappa^{-1}(w)}$), and \eqref{eqn:kappa2} implies that it is a poset isomorphism.
\end{proof}

\begin{rem}
\label{rem:not_combnl}
 If $\kappa\colon V_G(\KK)\to V_H(\KK)$ is an arbitrary linear map of vector edifices then the prescription $\sigma_x\mapsto \sigma_{\kappa(x)}$ need not give a well-defined map, even when $G$ and $H$ are reductive.  For instance, if $f$ is the inclusion of a torus $G$ in a non-commutative reductive group $H$ then $P_x(G)= G$ for every $x\in V_G(\KK)$ but there is more than one possible value for $P_{\kappa_f(x)}(H)$.
 So there are linear maps of vector buildings/vector edifices that do not correspond to any maps at the combinatorial level.  We will exploit the extra flexibility yielded by this geometric approach in our forthcoming work \cite{BMR:typeA}.
 
 Even if $\kappa$ is bijective, the conclusion of Lemma~\ref{lem:autsarecombinatorial} can fail.
 Let $B$ be a Borel subgroup in a non-commutative split reductive group $G$ and let $f\colon B\to G$ be inclusion.  Then Example~\ref{ex:bijectivenotiso} shows that $\kappa_f\colon V_B(\KK)\to V_G(\KK)$ is a bijection.  We can choose $x=0\in V_B(\KK)$, $y\in V_B(\KK)$ such that $G= P_{\kappa_f(x)}(G)\neq P_{\kappa_f(y)}(G)= B$; but $P_x(B)= P_y(B)= B$.
\end{rem}

\subsection{Projection maps}\label{sec:projection}
We can now, as promised in the introduction, relate our constructions to the standard notion of \emph{projection} for a spherical building.
First, recall that given a simplicial spherical building $\Delta$ and any pair $\sigma,\tau$ of simplices in $\Delta$ there is a uniquely defined simplex 
$\sigma\tau$ which is the projection of $\tau$ onto $\sigma$, \cite[2.30, 3.19]{tits1}, \cite[Def.~3.109]{abro}.
If $\Delta = \Delta_G$ is the combinatorial edifice for a reductive $k$-group $G$ (which is a simplicial spherical building), then
this projection has a natural realisation in terms of the corresponding parabolic subgroups of $G$: 
if $\sigma = \sigma_P$ and $\tau = \sigma_Q$, then $\sigma\tau$ is the simplex corresponding to $(P\cap Q)R_u(P)$, which is also a parabolic subgroup of $G$ \cite[Prop.~4.4]{boreltits} and is contained in $P$.
Further, if we choose a maximal split torus $T$ of $G$ contained in $P\cap Q$ and write $P = P_\lambda$, $Q=P_\mu$ for $\lambda,\mu\in Y_T$, 
then $(P\cap Q)R_u(P) = P_{n\lambda+\mu}$ for sufficiently large $n\in \NN$ \cite[Cor.~6.9]{BMR}.

Now let $L= L_\lambda$ be the Levi subgroup of $P$ corresponding to the choice of $\lambda$ above. 
Then $L$ is a reductive group and the simplicial building $\Delta_L$ of $L$ can be realised inside $\Delta_G$.
Again, this can be seen in terms of parabolic subgroups: for each parabolic subgroup $Q$ of $G$ contained in $P$, the subgroup
$Q\cap L$ is a parabolic subgroup of $L$, all the parabolic subgroups of $L$ arise in this way, and there is an inverse map $X\mapsto XR_u(P)$ mapping parabolic subgroups of $L$ to parabolic subgroups of $G$ contained in $P$ \cite[Prop.~4.4]{boreltits} --- in this way we may identify $\Delta_L$ with the \emph{link} or \emph{star} of $\sigma_P$ \cite[1.1]{tits1}, \cite[Prop.~4.9]{abro}; 
in more modern terminology, $\Delta_L$ arises as a \emph{residue}, see \cite[Sec.~5.3, Cor.~5.30]{abro}.
Combining this with the observations from the previous paragraph, we obtain a (surjective) map $\proj_{P,L}:\Delta_G\to \Delta_L$; 
this is the same as the \emph{projection onto a residue} of \cite[Def.~5.35]{abro}.

The key observation of this section is that these constructions can be realised in the setting of a vector edifice $V_G(\KK)$ with the common apartment property (Definition \ref{defn:common_apt}).
Since a reductive group $G$ gives rise to such a vector edifice, what follows can be viewed as a generalisation of the previously described constructions.
So now suppose $G$ is any $k$-group such that $V_G(\KK)$ has the common apartment property.
We have seen in Example \ref{ex:bijectivenotiso} above that for such a $G$ and an R-parabolic subgroup $P=P_\lambda$ of $G$, the inclusion $i:P\to G$ gives a bijective linear map 
$\kappa_i:V_{P}(\KK)\to V_G(\KK)$.
It follows from Lemma \ref{lem:surj} that if we let $L = L_\lambda$, the projection $f:P\to L\cong P/U_\lambda$ gives rise to a surjective linear map $V_{P}(\KK) \to V_{L}(\KK)$.
Hence the composition $$F_{P,L}: =\kappa_f\circ\kappa_i^{-1}:V_G(\KK) \to V_{L}(\KK)$$ is a surjective map.\footnote{The notation $F_{P,L}$ makes sense because $U_\lambda$ does not depend on the choice of $\lambda$ for fixed $P$ and $L$: see Lemma~\ref{lem:basicpropertiesRpars}(iii).}
Note that $f$ is apte by Remark \ref{rem:isogenydiscussion}(iii), because the kernel $U_\lambda$ of
$f$ is smooth.

\begin{lem}\label{lem:projectionisprojection}
With the notation just set up, $F_{P,L}$ induces a well-defined map of combinatorial edifices $\Delta_G\to \Delta_L$, which we denote by $\proj_{P,L}$.
If $G$ is reductive, then $\proj_{P,L}$ is the projection map on the spherical building of $G$ described above.
\end{lem}

\begin{proof}
To check that the map on combinatorial edifices induced by $F_{P,L}$ is well-defined, we need to check that if $\sigma_\mu = \sigma_\nu$ for $\mu,\nu\in Y_G(\KK)$, 
then $\sigma_{F_{P,L}(\mu)} = \sigma_{F_{P,L}(\nu)}$.
Unwinding the definitions, this amounts to checking that if $\mu,\nu \in Y_L(\KK)$ give rise to the same R-parabolic subgroup of $G$ then they give rise to the same R-parabolic subgroup of $L$, which is true by Corollary~\ref{cor:par_intersect}.

Now to check that the map induced by $F_{P,L}$ corresponds to $\proj_{P,L}$ in the case that $G$ is reductive, we again translate into a condition on 
R-parabolic subgroups.
This time, it suffices to show that given any $\mu\in Y_{L_\lambda}(\KK)$, we have $P_\mu(L_\lambda) = P_{n\lambda+\mu}(L_\lambda)$ for all $n\in \NN$. 
To see this, first note that for any cocharacter approximation $\lambda'$ to $\lambda$ and any $\mu'\in Y_{L_\lambda}$ we have 
$P_{\mu'}(L_{\lambda'}) = P_{n\lambda'+\mu'}(L_{\lambda'})$ since the image of $\lambda'$ is central in $L_{\lambda'} = L_\lambda$. 
Hence, the result follows after choosing any cocharacter approximation $\mu'\in Y_{L_\lambda}$ to $\mu$.
\end{proof}

\begin{rems}
\label{rem:proj_not_linear}
(i). The map $F_{P,L}$ is not linear in general.  For example, let $G= \SL_2$, let $P$ be the standard upper triangular Borel subgroup $B$ and let $L$ be the standard diagonal maximal torus $T$.  Choose $\lambda\in Y_T$ such that $P_\lambda= B$ and set $x= \phi_G(-\lambda)$.  Let $y= u\cdot x$, where $1\neq u\in U_\lambda(k)$.    Now $x$ and $y$ are opposite by the example in Remark~\ref{rem:W_type}, so if $F_{P,L}$ is linear then $F_{P,L}(x)$ and $F_{P,L}(y)$ must be opposite.  But this is not the case as $F_{P,L}(x)= F_{P,L}(y)= x$ by Lemma~\ref{lem:unipt_quot}.

(ii). Still in the setting of the previous example, let $Q= P_{-\lambda}$ be the opposite Borel to $B$ with respect to $T$. We can form the map $F_{P,L}$ using the projection $P\to L$ and the map $F_{Q,L}$ using the projection $Q\to L$.  Then $F_{P,L}(y)= F_{P,L}(x)$; but Lemma~\ref{lem:unipt_quot} implies that $F_{Q,L}(y)\neq F_{Q,L}(x)$, as $y$ is not $U_{-\lambda}(k)$-conjugate to $x$.  This shows that $F_{P,L}$ depends on the choice of $P$ for given $L$.
\end{rems}

\subsection{Linear maps of spherical edifices}
\label{sec:linearmaps}
Let $\kappa\colon V_{G_1}(\KK)\to V_{G_2}(\KK)$ be a linear map of vector edifices.  If $\kappa$ is injective then we get an induced map $$\kappa^\flat\colon \Delta_{G_1}(\KK)\to \Delta_{G_2}(\KK)$$ between the corresponding spherical edifices given by $\kappa^\flat(\KK^+\cdot x)= \KK^+\cdot \kappa(x)$. 
Clearly this map of spherical edifices is also injective.  We call a map of the form $\kappa^\flat$ a \emph{linear map}.
(To make this definition we need $\kappa$ to be injective, for otherwise there are nonzero elements of $V_{G_1}(\KK)$ mapped to $0\in V_{G_2}(\KK)$,
and then $\kappa^\flat$ cannot be defined on all of $\Delta_{G_1}(\KK)$.)

\begin{rem}
\label{rem:spherical_deficient}
 Note that we get a linear map $\kappa_f\colon V_G(\KK)\to V_H(\KK)$ for any homomorphism from $G$ to a connected $k$-group $H$, but we only get a corresponding map of spherical buildings $\Delta_G(\KK)$ to $\Delta_H(\KK)$ if $\kappa_f$ is injective.  This is one of the main advantages of working with vector edifices rather than with spherical edifices.
\end{rem}

\begin{defn}
	\label{def:delta-auto}
We define an \emph{automorphism} of $\Delta_G(\KK)$ to be a bijective linear map whose inverse is also a linear map.  We write $\Aut (\Delta_G(\KK))$ for the group of automorphisms of $\Delta_G(\KK)$.	
\end{defn}

  Clearly, if $\kappa\in \Aut (V_G(\KK))$ then $\kappa^\flat\in \Aut (\Delta_G(\KK))$.  Conversely, suppose $\kappa, \rho\colon V_G(\KK)\to V_G(\KK)$ are injective linear maps and $\rho^\flat= (\kappa^\flat)^{-1}$.  Then $(\rho\circ \kappa)^\flat= \rho^\flat\circ \kappa^\flat= {\rm id_{\Delta_G(\KK)}}$, so $\rho\circ \kappa$ maps each ray in $V_G(\KK)$ to itself and each apartment in $V_G(\KK)$ to itself.  Hence $\rho$ is a bijection, and it follows that $\kappa$ is a bijection.  We see that if $T$, $T'$ are maximal split tori of $G$ and $\rho$ maps $V_T(\KK)$ to $V_{T'}(\KK)$ then $\kappa^{-1}$ also maps $V_T(\KK)$ to $V_{T'}(\KK)$, so $\kappa$ belongs to $\Aut (V_G(\KK))$.  Hence $\Aut (\Delta_G(\KK))= \{\kappa^\flat\mid \kappa\in \Aut (V_G(\KK))\}$.

\section{Linear maps arising from base change}
\label{sec:fieldexts}

Let $\alpha\colon k\to k'$ be a homomorphism of fields --- for example, $\alpha$ could be a field isomorphism, or the inclusion arising from a field extension $k'/k$ --- and let $G$ be a $k$-group.  
Base change along $\alpha$ gives rise to a $k'$-group ${^\alpha}G$:
from the functorial point of view, given a $k'$-algebra $A'$, we have ${^\alpha}G(A') = G(A'\otimes_\alpha k)$, which makes sense 
as $A'\otimes_\alpha k$ is a $k$-algebra via scalar multiplication in the second factor.
Note that when $\alpha\colon k\to k'$ is the inclusion arising from a field extension $k'/k$, then ${^\alpha}G = G_{k'}$,
but this construction works more generally, for example when $\alpha\colon k\to k$ is a field automorphism.

Identifying a point $x\in G(k)$ with a morphism $\mathrm{Spec}(k)\to G$,
we obtain by base change along $\alpha$ a morphism $\mathrm{Spec}(k')\to {^\alpha}G$,
which is to say a point ${^\alpha}x \in {^\alpha}G(k')$.
The corresponding map $$\psi_\alpha\colon  G(k)\to {^\alpha}G(k'), \quad x\mapsto {^\alpha}x$$ is a
homomorphism of abstract groups.

\subsection{Base change induces an injection of vector edifices}
\label{sec:basechange}
We first wish to show how base change gives rise to a linear map of the corresponding vector edifices.

Suppose $f\colon H\to G$ is a homomorphism of $k$-groups.
Then base change gives rise to a homomorphism ${^\alpha}f\colon {^\alpha}H\to {^\alpha}G$
of $k'$-groups.
Suppose $G = (\Gm)_{k}$ is the multiplicative group over $k$.
Then ${^\alpha G}\cong (\Gm)_{k'}$ is the multiplicative group over $k'$ 
(the ring $A'\otimes_\alpha k$ has the same units as $A'$).
Combining this observation with the previous one, we 
see that for each cocharacter $\lambda\in Y_G$, we obtain a cocharacter 
${^\alpha}\lambda\in Y_{{^\alpha}G}$.
We also see that a split torus $T$ of $G$ gives rise to a split torus ${^\alpha}T$ of ${^\alpha}G$,
and the $k$-rank of $T$ is the same as the $k'$-rank of ${^\alpha}T$.

Now suppose $T$ is a maximal split torus of $G$.
Then the map $\lambda\mapsto {^\alpha}\lambda$ is an isomomorphism of abelian groups 
$Y_T\to Y_{{^\alpha}T}$,
and hence induces an isomorphism $Y_T(\KK)\to Y_{{^\alpha}T}(\KK)$,
which we also denote by $\lambda\mapsto {^\alpha}\lambda$.
The map $\chi\mapsto {^\alpha}\chi$ is an isomomorphism of abelian groups $X_T\to X_{{^\alpha}T}$, and we have $\langle {^\alpha}\lambda, {^\alpha}\chi\rangle= \langle \lambda, \chi\rangle$ for all $\lambda\in Y_T$ and all $\chi\in X_T$.

Let $\lambda\in Y_T(\KK)$ and choose a $G$-equivariant embedding of $G$ in a rational $G$-module $V$.   Base change gives an ${^\alpha}G$-equivariant embedding of ${^\alpha}G$ in the rational ${^\alpha}G$-module ${^\alpha}V$.  The weights of ${^\alpha}V$ with respect to ${^\alpha}T$ are the characters of the form ${^\alpha}\chi$, where $\chi$ runs over the set of weights of $V$ with respect to $T$.  We see that if $\lambda'$ is a cocharacter approximation to $\lambda$ with respect to $V$ then ${^\alpha}\lambda'$ is a cocharacter approximation to ${^\alpha}\lambda$ with respect to $^\alpha V$.  We have $P_{{^\alpha}\lambda}= {^\alpha}P_\lambda$, $L_{{^\alpha}\lambda}= {^\alpha}L_\lambda$ and $U_{{^\alpha}\lambda}= {^\alpha}U_\lambda$ for any $\lambda\in Y_G(\KK)$;
this holds if $\lambda\in Y_G$ by \cite[Lem.\ 2.1.5]{CGP} and the discussion that precedes it, and for arbitrary $\lambda$ we get the result by taking cocharacter approximations.  It follows that $\psi_\alpha(P_\lambda(k))\subseteq P_{{^\alpha}\lambda}(k')$, $\psi_\alpha(L_\lambda(k))\subseteq L_{{^\alpha}\lambda}(k')$ and $\psi_\alpha(U_\lambda(k))\subseteq U_{{^\alpha}\lambda}(k')$ for all $\lambda\in Y_G(\KK)$.

\begin{lem}\label{lem:basechangelinear}
Suppose $\alpha\colon k\to k'$ is a homomorphism of fields.
Then the map $\lambda\mapsto {^\alpha}\lambda$ gives rise to well-defined injective maps 
$Y_G(\KK) \to Y_{{^\alpha}G}(\KK)$ and $V_G(\KK) \to V_{{^\alpha}G}(\KK)$, and the latter map is a linear map of vector edifices.
\end{lem}

\begin{proof}
Since we have already observed bijectivity and linearity at the level of apartments above, the essential point is well-definedness.
What we need follows from functoriality of base change: 
if $T_1,T_2$ are maximal split tori of $G$ and $g\in G(k)$ is such that $gT_1g^{-1} = T_2$, then we have 
$\psi_\alpha(g)({^\alpha}T_1)\psi_\alpha(g)^{-1} = {^\alpha}T_2$.
Similarly, if $g\cdot\lambda_1 = \lambda_2$ for $\lambda_1\in Y_{T_1}(\KK)$ and $\lambda_2\in Y_{T_2}(\KK)$,
then $\psi_\alpha(g)\cdot {^\alpha}\lambda_1 = {^\alpha}\lambda_2$.
We can now check that the map $\lambda\mapsto{^\alpha}\lambda$ behaves well with respect to the relations $\sim$ and $\approx$
before and after base change.
\end{proof}

\begin{rem}
Since it is not necessarily true that every maximal split torus of ${^\alpha}G$ is of the form ${^\alpha}T$ for some maximal split torus $T$ of $G$
(for example, take $\alpha$ a Frobenius endomorphism of a non-perfect field $k$), the maps in Lemma \ref{lem:basechangelinear} are not necessarily surjective.
\end{rem}

\begin{defn}\label{def:nu_alpha}
Suppose $\alpha\colon k\to k'$ is a homomorphism of fields, and $G$ is a $k$-group.
We denote the corresponding injective linear map of vector edifices by $$\nu_\alpha := \nu_{G,\alpha}\colon V_G(\KK)\to V_{{^\alpha}G}(\KK).$$ 
\end{defn}

We note in particular that this formalism allows us to identify the vector edifice $V_G(\KK)$ as a subset of the vector edifices $V_{G_{k_s}}(\KK)$ and $V_{G_{\ovl{k}}}(\KK)$,
which is helpful for later arguments on metrics.

\begin{lem}\label{lem:lin_map_fieldext}
	Suppose $G$ and $H$ are $k$-groups, and let $\alpha\colon k\to k'$ be a field homomorphism.
	For any homomorphism $f\colon G\to H$, we have $\kappa_{({^\alpha}f)}\circ \nu_{G,\alpha} = \nu_{H,\alpha}\circ \kappa_f$.
\end{lem}

\begin{proof}
	This is basically just functoriality of base change: it is clear that for $\lambda\in Y_G$ we have ${^\alpha}(f\circ\lambda) = {^\alpha}f\circ{^\alpha}\lambda$,
	and the result follows by chasing through the definitions.
\end{proof}

\subsection{Field automorphisms and vector edifice automorphisms}
\label{sec:field_auts}
First observe that if $\alpha\colon k\to k'$ is an isomorphism of fields and $G$ is a $k$-group, then the linear maps $\nu_{G,\alpha}$ and $\nu_{{^\alpha}G,\alpha^{-1}}$
are inverses of each other, and hence $V_G(\KK)$ and $V_{{^\alpha}G}(\KK)$ are isomorphic.

We also wish to consider the possibility that ${^\alpha}G$ and $G$ are isomorphic as $k$-groups.
For ease of exposition, suppose $\alpha\in \Aut(k)$ --- i.e., $\alpha$ is given to us as a field \emph{automorphism} --- so that
$G$ and ${^\alpha}G$ are both $k$-groups.
Then if there exists an isomorphism $f\colon {^\alpha}G\to G$,
we obtain a linear automorphism $$\kappa_{\alpha,f} :=\kappa_f\circ\nu_\alpha$$ of $V_G(\KK)$.

We discuss this further in \cite{BMR:typeA}.  Here we give two examples of this construction, showing how it captures naturally occurring phenomena.

\begin{exmp}\label{ex:Galois_action}
Suppose $\alpha\in \Aut(k)$ and let $k^\alpha$ be the fixed field of $\alpha$.
Suppose that $G$ has a $k^\alpha$-descent --- that is, that $G = H_k$ for some $k^\alpha$-group $H$.
Then we claim that ${^\alpha}G$ and $G$ are \emph{naturally} isomorphic as $k$-groups via a homomorphism $f_\alpha$ (say).
To see this, note that for a $k$-algebra $A$, when we base change and form the algebra $A\otimes_\alpha k$ (recall, the $k$-multiplication happens in the second factor), 
we have an isomorphism of these two as $k^\alpha$-algebras $A\otimes_\alpha k \mapsto A$ given by $a\otimes 1\to a$.
This gives rise to the corresponding natural isomorphisms 
$$
f_\alpha(A)\colon {^\alpha}G(A) = G(A\otimes_\alpha k) = H(A\otimes_\alpha k) \to H(A) = G(A)
$$ 
for each $k$-algebra $A$, defining the isomorphism $f_\alpha$.

A particular instance of this set-up is if $k/k_0$ is a Galois extension,
and $G$ is a $k$-group for which we can fix a $k_0$-descent $G_0$.
Then we obtain from the various elements $\kappa_{\alpha,f_\alpha}$
an action of the Galois group $\Gamma = \Gal(k/k_0)$ on $V_G(\KK)$.
Moreover, the set of fixed points $V_G(\KK)^\Gamma$ naturally identifies
with the vector edifice $V_{G_0}(\KK)$.
\end{exmp}

\begin{exmp}\label{ex:frobenius}
Suppose $k$ is a perfect field of characteristic $p$, so that the Frobenius map $\alpha_p\colon k\to k, x\mapsto x^p$ is a field automorphism.
Let $G = \GL_n$ be the general linear group over $k$.
Then this is a situation as in the previous example, where the group ${^\alpha}G$ is naturally isomorphic to $G$, say via a map $f_p$,
so we obtain an automorphism of the vector building $\kappa_p:=\kappa_{\alpha_p,f_p}$.
We also have the apte homomorphism $F_p\colon G\to G$ which is the standard Frobenius morphism raising matrix entries to the $p^{\rm th}$ power
(i.e., acting by $\alpha_p$ in each coordinate), which gives rise to a linear automorphism $\kappa_{F_p}$ of $V_G(\KK)$.
We can compare the effects of $\kappa_p$ and $\kappa_{F_p}$ by looking at the standard diagonal maximal torus $T$.

First, it is clear $\kappa_{F_p}(\lambda) = p\lambda$ for all $\lambda\in Y_T(\KK)$.
Second, because the natural isomorphism $f_p\colon {^\alpha}G\to G$ is induced by the maps on $A$-points $G(A\otimes_{\alpha_p} k) \to G(A)$ given by $a\otimes 1\mapsto a$ in each coordinate (as described in the previous example), the net effect on a diagonal cocharacter after precomposing with base change along $\alpha_p$ is to do nothing!  Note (and this is important) that we are \emph{not} claiming that $\kappa_p$ is the identity map on all of $V_G(\KK)$.
For a general split maximal torus $T'$, there exists $g\in G(k)$ such that $T' = gTg^{-1}$.
Then for $x\in V_T(\KK)$, $\kappa_{F_p}$ sends $g\cdot x \in V_{T'}(\KK)$ to $F_p(g)\cdot(px) \in V_{F_p(g)TF_p(g)^{-1}}(\KK)$, by Remark~\ref{rem:type-preserving}.
On the other hand, it can be shown that $\kappa_{p}$ sends $g\cdot x$ to $F_p(g)\cdot x$ --- the key point is that for $g\in \GL_n(k)$, applying 
$f_p\circ \psi_{\alpha_p}$ has the same effect as applying $F_p$. 
Thus we see that both $\kappa_{F_p}$ and $\kappa_p$ permute the apartments of $V_G(\KK)$,
in the same way, but $\kappa_{F_p}$ also scales by $p$.

Finally, we note that if $H$ is a subgroup of $\GL_n$ such that $F_p(H) = H$, then this conclusion naturally extends to the
corresponding linear maps on $V_H(\KK)$ given by restriction from $V_G(\KK)$.
\end{exmp}

We finish this section by recalling that in the case where $G$ is a simple $k$-group of rank at least $2$, Tits has shown \cite[Cor.~5.9, Cor.~5.10]{tits1} that 
all automorphisms of the spherical building $\Delta_G$ can be constructed by considering those induced by automorphisms of the group $G$ together
with suitable automorphisms of the field $k$ and certain exceptional Frobenius isogenies in characteristics $2$ and $3$. 
Our constructions in the previous two sections show how to realise all these automorphisms on the level of the edifice $V_G(\KK)$ as well.

\section{Admissible metrics}
\label{sec:admmetric}

In this section we show how to put a metric on $V_G(\KK)$.  Recall that two metrics $d_1$ and $d_2$ on a set $X$ are said to be \emph{bi-Lipschitz equivalent} if there exist $c,C> 0$ such that $d_1(x,y)\leq c\cdot d_2(x,y)\leq C\cdot d_1(x,y)$ for all $x,y\in X$; in this case, $d_1$ and $d_2$ induce the same topology on $X$.  When $G$ is reductive then (following \cite[Sec.\ 2]{kempf}, \cite{BMR:strong}) we can equip $V_G(\KK)$ with a metric, as follows.
Fix a maximal split torus $T$ of $G$.
The space $V_T(\RR)$ can be equipped with a $W_k$-invariant metric (this is standard: just take any positive-definite bilinear form on $V_T(\RR)$, average over the finite group $W_k$, and take the metric defined by the associated norm).
Now for any other maximal split torus $T'$ in $G$, there exists $g\in G(k)$ such that $gT'g^{-1} = T$ --- this choice of $g$ is not unique, but any two such differ by an element of $W_k$, and so we can translate the $W_k$-invariant metric on $V_T(\KK)$ to $V_{T'}(\KK)$.  This allows us to define $d(x,y)$ for any $x,y\in V_G(\KK)$, since the common apartment property holds.  We thus obtain a metric $d$ on $V_G(\KK)$ with the property that $d(x,y) = d(g\cdot x,g\cdot y)$ for all $x,y\in V_G(\KK)$, $g\in G(k)$;
we call such a metric arising from a $W_k$-invariant bilinear form on $V_T(\KK)$ \emph{admissible}.

As explained in Section~\ref{sec:common_apt}, for general $G$ the vector edifice $V_G(\KK)$ may fail to have the
common apartment property, so
we cannot immediately metrise $V_G(\KK)$ using this construction.
However, with Proposition~\ref{prop:inclusion} in hand, we can make the following definition.

\begin{defn}\label{defn:metric}
 Let $i\colon G\to G'$ be an embedding of $G$ in a reductive $k$-group $G'$.  
 Choose an admissible metric $d'$ on $V_{G'}(\KK)$.  
 Recall that the map $\kappa_i\colon V_{G}(\KK)\to V_{G'}(\KK)$ is injective (Proposition~\ref{prop:inclusion}).  
 We obtain a pullback metric $d$ on $V_{G}(\KK)$ defined by 
 $$
 d(x,y):=d'(\kappa_i(x),\kappa_i(y))
 $$  
 for all $x,y\in V_G(\KK)$.
 We also denote this metric
 by $i^*(d')$, and call a metric obtained in this way an \emph{admissible metric} on $V_{G}(\KK)$.  
 Note that the equivariance of $\kappa_i$ observed in Remark~\ref{rem:type-preserving}, together with the $G'(k)$-invariance of $d'$, implies that $d$ is $G(k)$-invariant and the restriction of $d$ to any apartment of $V_{G}(\KK)$ is given by a positive-definite bilinear form.  An admissible metric $d$ gives rise to a topology on $V_G(\KK)$.  Below we write ``open'' instead of ``$d$-open'' if $d$ is understood.
 
 Given an admissible metric $d$ on $V_G(\KK)$ we let $\|\cdot\| = \|\cdot\|_d$ denote the corresponding norm: $\|x\|:= d(x,0)$ for $x\in V_G(\KK)$.
 \end{defn}

\begin{rems}
\label{rem:admprops}
(i). Since any $G$ can be embedded into some $\GL_n$, admissible metrics always exist.

(ii). It is clear that if $G$ is reductive then the two notions of admissible metric coincide.

(iii). Let $H$ be a connected subgroup of $G$ and let $d$ be an admissible metric on $V_G(\KK)$, arising from an embedding of $G$ in a reductive group $G'$ and an admissible metric $d'$ on $V_{G'}(\KK)$.  Then the restriction of $d$ to $V_H(\KK)$ (regarded as a subset of $V_G(\KK)$) is admissible: for it is the metric obtained from $d'$ via the inclusion of $H$ in $G'$.

(iv). Let $d$ be an admissible metric on $V_G(\QQ)$.  It is easily checked that there is a unique extension of $d$ to an admissible metric on $V_G(\RR)$.

(v). If $d$ is an admissible metric on $V_G(\KK)$ arising from an embedding $i:G\to G'$ and an admissible metric $d'$ on $V_{G'}(\KK)$, 
then we have 
$$
d(ax,ay) = d'(\kappa_i(ax),\kappa_i(ay)) = d'(a\kappa_i(x),a\kappa_i(y)) = ad'(\kappa_i(x),\kappa_i(y)) = ad(x,y)
$$ 
for all $a\in \KK^+$ and $x,y\in V_G(\KK)$. 
This uses the linearity of $\kappa_i$ and the fact that the edifice $V_{G'}(\KK)$ has the common apartment property (because $G'$ is reductive). 

(vi). If $d$ is an admissible metric on $V_G(\KK)$, then the subspace topology on any apartment $V_T(\KK)$ is just the Euclidean topology induced by the 
positive-definite bilinear form obtained by restricting the metric to $V_T(\KK)$.
\end{rems}

\begin{rem}\label{rem:boundedconjugates}
Note that since an admissible metric is $G(k)$-invariant, so is the norm arising from it: in particular, if $(x_n)$ is a bounded sequence in $V_G(\KK)$ and $(g_n)$ 
 is a sequence of elements of $G(k)$, then the sequence $(g_n\cdot x_n)$ is also bounded. We use this fact often in what follows.
\end{rem}

\begin{rem}\label{rem:reductive_quasi_isom}
 If $G$ has the common apartment property --- e.g., if $G$ is pseudo-reductive --- then any two admissible metrics on $V_{G}(\KK)$ are bi-Lipschitz equivalent --- this follows from the common apartment property, the conjugation-invariance of the metrics and the standard fact that any two positive-definite bilinear forms on a finite-dimensional vector space over $\KK$ induce bi-Lipschitz equivalent metrics.  In particular, the induced topology on $V_{G}(\KK)$ does not depend on the choice of metric.  Below we extend these results to arbitrary connected $k$-groups (Proposition~\ref{prop:bdd_linear}).
\end{rem}

\begin{defn}
 Given an admissible metric $d$ on $V_G(\KK)$, we define $\Isom _d(V_G(\KK))$ to be the subgroup of $\Aut (V_G(\KK))$ consisting of the isometries with respect to the 
 metric $d$.
\end{defn}

\begin{exmp}
\label{ex:not_isom}
 Not every automorphism is an isometry: for a simple example, choose $a\in \KK^+$ with $a\neq 1$ and take $f_a\in \Aut (V_G(\KK))$ to be the map that sends $x$ to $ax$.  See also the paragraph following Conjecture~\ref{conj:sTCC_spherical}.
\end{exmp}

\begin{lem}
\label{lem:Cauchy_type}
 Fix an admissible metric $d$ on $V_G(\KK)$.  Let $(x_n)$ be a Cauchy sequence in $V_G(\KK)$ such that all the $x_n$ have the same type.  Then $(x_n)$ is eventually constant.
\end{lem}

\begin{proof}
 Clearly we can assume without loss that $G$ is reductive.  Let $x\in V_G(\KK)$,
 and define 
 $$\epsilon(x)= \inf\{d(x,y)\mid \mbox{$x\neq y$ and $x$ and $y$ have the same type}\}.$$  
  Remark~\ref{rem:W_type} implies that if we choose an apartment $V_T(\KK)$ with $x\in V_T(\KK)$,
  then 
  $$\inf\{d(x,y)\mid \mbox{$x\neq y\in V_T(\KK)$ and $x$ and $y$ have the same type}\}>0.$$
  But given any other apartment $V_{T'}(\KK)$ containing $x$, $T'$ is conjugate to $T$ by an element fixing $x$ by Remark \ref{rem:conjugateapartments},
  and so this lower bound does not vary from apartment to apartment containing $x$.
Since $G$ has the common apartment property we conclude that $\epsilon(x)>0$. 
Further, $\epsilon(x')= \epsilon(x)$ if $x$ and $x'$ have the same type, as $d$ is conjugation-invariant, so $\epsilon(x_n)$ does not depend on $n$.  
The result follows.
\end{proof}

\begin{prop}
\label{prop:complete}
 Let $G$ be a $k$-group and let $d$ be an admissible metric on $V_{G}(\RR)$.  
 Then $(V_{G}(\RR), d)$ is a complete metric space.
\end{prop}

\begin{proof}
 The argument of \cite[Prop.\ 12.10]{abro} establishes this result when $G$ is semisimple; our proof is a slight variation on that theme.  
By definition, there exist an embedding $i$ of $G$ in some reductive $k$-group $G'$ and an admissible metric $d'$ on $V_{G'}(\RR)$ such that $d= i^*(d')$.  
To ease notation, we regard $V_{G}(\RR)$ as a subset of $V_{G'}(\RR)$ via $\kappa_i$, so that $d$ is just the restriction of $d'$.  
Let $(x_n)$ be a Cauchy sequence in $V_{G}(\RR)$.  
Fix a maximal split torus $T$ of $G$.  
By conjugacy of maximal split tori in $G$, for each $n$ we can choose $h_n\in G(k)$ such that $z_n:= h_n\cdot x_n$ belongs to the apartment $V_{T}(\RR)$ of $V_{G}(\RR)$.  Note that the sequence $(z_n)$ is bounded, because $(x_n)$ is (see Remark \ref{rem:boundedconjugates}).  Fix a maximal split torus $S$ of $G'$ such that $T\subseteq S$.  Then each $z_n$ belongs to the apartment $V_{S}(\RR)$ of $V_{G'}(\RR)$.  
There is no harm in passing to a subsequence of $(x_n)$, so without loss we can assume that the parabolic subgroups $P_{z_n}(G')$ are all equal by Remark \ref{rem:finitepar}.
Since the restriction of $d'$ to $V_{S}(\RR)$ yields a complete metric and any closed ball in $V_S(\RR)$ is compact, we can assume after passing to a subsequence again that $(z_n)$ converges to some limit $z\in V_{S}(\RR)$.  Each $z_n$ belongs to $V_T(\RR)$ and $V_T(\RR)$ is a subspace of the vector space $V_S(\RR)$, so $z\in V_T(\RR)$.
 
 Define a sequence $(y_n)$ in $V_G(\KK)$ by $y_n= h_n^{-1}\cdot z$.  Then $d'(x_n,y_n)= d'(h_n\cdot x_n, h_n\cdot y_n)= d'(z_n, z)\to 0$; hence $(y_n)$ is Cauchy, since $(x_n)$ is.  But the $y_n$ all have the same type as each other as elements of $V_{G'}(\KK)$, so $(y_n)$ must eventually become constant by Lemma~\ref{lem:Cauchy_type}: say, $y_n = y$ for sufficiently large $n$.  We see that $x_n\to y$, so we are done.
\end{proof}

We now consider admissible metrics on the edifice for a product of groups.  Let $G_1$ and $G_2$ be $k$-groups.  The maximal split tori of $G_1\times G_2$ are precisely the subgroups of the form $T_1\times T_2$, where $T_i$ is a maximal split torus of $G_i$.  There is an obvious linear isomorphism from $V_{T_1}(\KK)\oplus V_{T_2}(\KK)$ to $V_{T_1\times T_2}(\KK)$.  For $\lambda= (\lambda_1, \lambda_2)\in V_{T_1\times T_2}(\KK)$, we claim that $P_\lambda(G_1\times G_2)= P_{\lambda_1}(G_1)\times P_{\lambda_2}(G_2)$, $L_\lambda(G_1\times G_2)= L_{\lambda_1}(G_1)\times L_{\lambda_2}(G_2)$ and $U_\lambda(G_1\times G_2)= U_{\lambda_1}(G_1)\times U_{\lambda_2}(G_2)$.  This is clear if $\lambda_1\in Y_{T_1}$ and $\lambda_2\in Y_{T_2}$.  To see this in general, choose a $G_i$-equivariant embedding of $G_i$ in a rational $G_i$-module $V_i$ for $i= 1,2$, then embed $G_1\times G_2$ in $V_1\oplus V_2$ in the obvious way.  Then $(\lambda_1', \lambda_2')$ is a cocharacter approximation to $(\lambda_1, \lambda_2)$ if and only if $\lambda_i'$ is a cocharacter approximation to $\lambda_i$ for $i= 1,2$, and the claim follows.  It is now easily checked that the maps $V_{T_1}(\KK)\oplus V_{T_2}(\KK)\to V_{T_1\times T_2}(\KK)$ paste together to give a bijection from $V_{G_1}(\KK)\times V_{G_2}(\KK)$ to $V_{G_1\times G_2}(\KK)$.

Now let $d_1= i_1^*(d_1')$ and $d_2= i_2^*(d_2')$, where $i_j$ is an embedding of $G_j$ in a reductive $k$-group $G_j'$ and $d_j'$ is an admissible metric on $V_{G_j'}(\KK)$.  We get an admissible metric $d'$ on $V_{G_1'\times G_2'}(\KK)= V_{G_1'}(\KK)\times V_{G_2'}(\KK)$ given by
$$
d'((x_1',x_2'), (y_1',y_2'))= \sqrt{d_1'(x_1',y_1')^2+ d_2'(x_2',y_2')^2}:
$$ 
on each apartment $V_{T_1'\times T_2'}(\KK)= V_{T_1'}(\KK)\oplus V_{T_2'}(\KK)$, the positive-definite bilinear form corresponding to $d'$ is the orthogonal direct sum of the positive-definite bilinear forms on $V_{T_j'}(\KK)$ corresponding to $d_j'$ for $j= 1,2$.  We have an embedding $i_1\times i_2$ of $G_1\times G_2$ in $G_1'\times G_2'$ and we define an admissible metric $d$ on $V_{G_1\times G_2}(\KK)$ by $d= (i_1\times i_2)^*(d')$.  We write $d= d_1\times d_2$ and we call a metric of this form a \emph{product metric}.  We have
$$
d((x_1,x_2), (y_1,y_2))= \sqrt{d_1(x_1,y_1)^2+ d_2(x_2,y_2)^2}
$$
for all $(x_1, x_2), (y_1, y_2)\in V_{G_1}(\KK)\times V_{G_2}(\KK)= V_{G_1\times G_2}(\KK)$.

\subsection{Central tori}
\label{subsec:cent_tor}
Suppose $G$ has a non-trivial split central torus $Z$.  We need an argument for reducing from $G$ to $G/Z$.  Let $d$ be an admissible metric on $V_G(\KK)$.  Then there exist an embedding $i$ of $G$ in a reductive group $G'$, and an admissible metric $d'$ on $V_{G'}(\KK)$ such that $d= i^*(d')$.  Set $Z'= i(Z)$.  We can replace $G'$ with the reductive group $C_{G'}(Z')$ if necessary, so without loss we can assume that $Z'$ is central in $G'$.

Now let $T'$ be any maximal split torus of $G'$.  Since $T'$ contains $Z'$, we may regard $V_{Z'}(\KK)$ as a subspace of $V_{T'}(\KK)$.  Set $W_{T'}= V_{Z'}(\KK)^\perp$, where $^\perp$ denotes the orthogonal complement corresponding to the bilinear form that defines $d'$.  Given $x'\in V_{T'}(\KK)$, we have a unique decomposition $x'= x'_{Z'}+ x'_\perp$, where $x'_{Z'}\in V_{Z'}(\KK)$ and $x'_\perp\in W_{T'}$.  
Note that since any other maximal split torus $S'$ of $G'$ with $x'\in V_{S'}(\KK)$ is conjugate to $T'$ by an element of $G'(k)$ fixing $x'$ (essentially by definition of $V_{G'}(\KK)$), and $Z'$ is central in $G'$, the elements $x'_{Z'}$ and $x'_\perp$ do not depend on the choice of $T'$ with $x'\in V_{T'}(\KK)$.

Given $x',y'\in V_{T'}(\KK)$, we define $d'_{Z'}(x', y')= d'(x'_{Z'}, y'_{Z'})$ and $d'_\perp(x', y')= d'(x'_\perp, y'_\perp)$.  For any $g'\in G'(k)$, $g'$ fixes $V_{Z'}(\KK)$ pointwise (since it fixes $Y_{Z'}$ pointwise), so $W_{g'T'(g')^{-1}}= g'\cdot W_{T'}$ and for any $x'\in V_{T'}(\KK)$,
\begin{equation}
\label{eqn:perp_eqvce}
 (g'\cdot x')_{Z'}= g'\cdot x'_{Z'}= x'_{Z'} \ \mbox{\ and\ }\ (g'\cdot x')_\perp= g'\cdot x'_\perp.
\end{equation}
It follows easily that $d'_{Z'}$ and $d'_\perp$ give well-defined $G'(k)$-invariant functions from $V_{G'}(\KK)\times V_{G'}(\KK)$ to $\RR$.  We define $d_{Z}(x, y)= d'_{Z'}(\kappa_i(x), \kappa_i(y))$ and $d_\perp(x, y)= d'_\perp(\kappa_i(x), \kappa_i(y))$ for $x,y\in V_G(\KK)$.  We also have $d(x,y) = d'(\kappa_i(x),\kappa_i(y))$, by definition, so we can conclude that
\begin{equation}
\label{eqn:pythagoras}
 d(x, y)= \sqrt{d_{Z}(x, y)^2+ d_\perp(x, y)^2}\ \mbox{for all $x,y\in V_G(\KK)$},
\end{equation}
because the corresponding equation holds for $d'$, $d'_{Z'}$, $d'_\perp$ and for all $\kappa_i(x),\kappa_i(y)\in V_{G'}(\KK)$.
Note that $d$ and $d_Z$ agree on $V_Z(\KK)\times V_Z(\KK)$.

Now let $G_1= G/Z$ and let $G_1'= G'/Z'$.  Let $\pi\colon G\to G_1$ and $\pi'\colon G'\to G_1'$ denote the canonical projections.  Then $i$ gives rise to an embedding $i_1$ of $G_1$ in $G_1'$.  Choose an admissible metric $d_1'$ on $V_{G_1'}(\KK)$ and let $d_1= i_1^*(d_1')$.

\begin{lem}
 Let the notation be as above.  Then there exist $c, C> 0$ such that
 \begin{equation}
 \label{eqn:perpqisom}
 d_1(\kappa_\pi(x), \kappa_\pi(y))\leq c \cdot d_\perp(x, y)\leq C \cdot d_1(\kappa_\pi(x),\kappa_\pi(y))
 \end{equation}
 for all $x,y\in V_G(\KK)$.
\end{lem}

\begin{proof}
 It is enough to show that there exist $c, C> 0$ such that
 \begin{equation}
 \label{eqn:perpqisom2}
 d'_1(\kappa_{\pi'}(x'), \kappa_{\pi'}(y'))\leq c \cdot d'_\perp(x', y')\leq C \cdot d'_1(\kappa_{\pi'}(x'),\kappa_{\pi'}(y'))
 \end{equation}
 for all $x',y'\in V_{G'}(\KK)$.  So fix a maximal split torus $T_1'$ of $G_1'$.  Since $\pi'$ is smooth and surjective, $\pi'$ is apte by Remark \ref{rem:isogenydiscussion}(iii), so there is a maximal split torus $T'$ of $G'$ such that $\pi'(T')= T_1'$.  It is straightforward to check that $\kappa_{\pi'}$ gives an isomorphism of $\KK$-vector spaces from $V_{T'}(\KK)/V_{Z'}(\KK)$ to $V_{T_1'}(\KK)$, so $\kappa_{\pi'}$ gives an isomorphism of $\KK$-vector spaces from $W_{T_1'}$ to $V_{T_1'}(\KK)$.  Since any two positive-definite bilinear forms on a vector space yield equivalent norms, there exist $c, C> 0$ such that \eqref{eqn:perpqisom2} holds for all $x', y'\in V_{T'}(\KK)$.  It follows from the $G'(k)$-invariance of $d'_1$ and $d'_\perp$ that the same $c$ and $C$ work for any maximal split torus of $G'$.  This gives the result.
\end{proof}

Consider the product group $G_1\times Z$.  The maximal split tori of $G_1\times Z$ are precisely the subgroups of the form $T_1\times Z$, where $T_1$ is a maximal split torus of $G_1$.  Keeping the notation above, we endow $V_{G_1\times Z}(\KK)$ with the product metric $\widetilde{d}:= d_1\times d_Z$, where $d_Z$ is the restriction of $d$ to $V_Z(\KK)$.  Given a maximal split torus $T_1$ of $G_1$, pick a maximal split torus $T$ of $G$ such that $\pi(T)= T_1$.  Since $T$ contains $\ker(\pi)= Z$, $T$ is in fact unique.  Define $$\beta_{Z,T}\colon V_T(\KK)\to V_{T_1}(\KK)\times V_Z(\KK)= V_{T_1\times Z}(\KK)$$ by $\beta_{Z,T}(x)= (\kappa_\pi(x), x_Z)$.

\begin{lem}
\label{lem:Z_qisom}
 The maps $\beta_{Z,T}$ paste together to give a well-defined isomorphism $\beta_Z\colon V_G(\KK)\to V_{G_1\times Z}(\KK)$.  Moreover, $\beta_Z$ is a bi-Lipschitz map from $(V_G(\KK), d)$ to $(V_{G_1\times Z}(\KK), \widetilde{d})$.
\end{lem}

\begin{proof}
 That $\beta_Z$ is well-defined follows from \eqref{eqn:perp_eqvce} and Remark~\ref{rem:type-preserving}. Since each $\beta_{Z,T}$ is bijective, $\beta_Z$ is bijective and the linearity of $\beta_Z$ and $\beta_Z^{-1}$ is a consequence of the definitions.  The bi-Lipschitz property follows from \eqref{eqn:perpqisom}.
\end{proof}

\subsection{Bi-Lipschitz equivalence of norms}
We have seen that the common apartment property can fail for $V_G(\KK)$.  We can, however, prove a weaker local version.

\begin{lem}
\label{lem:1/2_cts}
 Let $d$ be an admissible metric on $V_{G}(\KK)$.  Let $x\in V_{G}(\KK)$.
  Then there exists a $d$-open neighbourhood $O$ of $x$ in $V_G(\KK)$ such that 
  \begin{itemize}
  \item[(i)] $P_y\subseteq P_x$ for all $y\in O$;
  \item[(ii)] $O\subseteq V_{L_\lambda}(\KK)$ for any $\lambda\in Y_G(\KK)$ such that $\phi_G(\lambda)= x$.
  \end{itemize}
\end{lem}
\begin{proof}
 First suppose $G$ is reductive.  Recall (see Lemma \ref{lem:basechangelinear}, Definition \ref{def:nu_alpha}) that we have an embedding of $Y_G(\KK)$ in $Y_{G_{k_s}}(\KK)$ and of $V_G(\KK)$ in $V_{G_{k_s}}(\KK)$; we denote the image of $\mu\in Y_G(\KK)$ in $Y_{G_{k_s}}(\KK)$ by $\widetilde{\mu}$, and of $y\in V_G(\KK)$ in $V_{G_{k_s}}(\KK)$ by $\widetilde{y}$.  Choose a maximal split torus $T$ of $G$ such that $x\in V_T(\KK)$, and choose a maximal torus $\widetilde{T}$ of $G_{k_s}$ such that $T_{k_s}\subseteq \widetilde{T}$: then $\widetilde{x}\in V_{\widetilde{T}}(\KK)$.  Choose $\lambda\in Y_T(\KK)$ such that $x= \phi_G(\lambda)$, and recall that in this case we can determine the R-parabolic subgroup $P_{\widetilde{\lambda}}$ by considering the pairing of $\widetilde{\lambda}$ with roots of $G_{k_s}$ (see Example \ref{ex:red_virtual_cochar}).  
 Let $\alpha_1,\ldots, \alpha_s$ be the roots of $U_{\widetilde{\lambda}}$ with respect to $\widetilde{T}$.  The set $\{\mu\in Y_T(\KK)\mid \langle \widetilde{\mu}, \alpha_i\rangle> 0\ \mbox{for $1\leq i\leq s$}\}$ is open and contains $\lambda$, so its image $\widetilde{O}$ in $V_T(\KK)$ is an open neighbourhood of $x$.  Hence there exists $\epsilon> 0$ such that $y\in \widetilde{O}$ for all $y\in V_T(\KK)$ such that $d(x, y)< \epsilon$.
 
 Let $O$ be the open ball $B(x, \epsilon)$ in $V_G(\KK)$.  Let $y\in O$.  Since $G$ is reductive, the common apartment property holds, so there is a maximal split torus $T_0$ of $G$ such that $x, y\in V_{T_0}(\KK)$.  Now $T_0$ is a maximal split torus of $P_x$, so there exists $g\in P_x(k)$ such that $gT_0g^{-1}= T$.  We have $g\cdot x= x$, so $d(x, g\cdot y)= d(g\cdot x, g\cdot y)= d(x, y)< \epsilon$.  It follows from Example~\ref{ex:red_virtual_cochar} and by construction that $U_{g\cdot y}\supseteq U_x$, so $P_{g\cdot y}\subseteq P_x$
 (this conclusion holds because $G$ is reductive).  Hence $P_y= g^{-1}P_{g\cdot y}g\subseteq g^{-1}P_xg= P_{g^{-1}\cdot x}= P_x$ by Example~\ref{ex:parconj}.  This proves part (i) when $G$ is reductive.
 
 Now let $G$ be arbitrary and let $x\in V_G(\KK)$.  There exist an embedding $i$ of $G$ in a reductive $k$-group $G'$ and an admissible metric $d'$ on $V_{G'}(\KK)$ such that $d= i^*(d')$.  By the reductive case, there is a $d'$-open neighbourhood $O'$ of $\kappa_i(x)$ such that $P_w(G')\subseteq P_{\kappa_i(x)}(G')$ for all $w\in O'$.  By construction, $O:= \kappa_i^{-1}(O')$ is a $d$-open neighbourhood of $x$.  Let $y\in O$.  Then $P_{\kappa_i(y)}(G')\subseteq P_{\kappa_i(x)}(G')$, so $P_y\subseteq P_x$  by Corollary~\ref{cor:par_intersect}.  This finishes the proof of (i).
 
 To prove (ii), let $y\in O$.  Choose $\lambda, \mu\in Y_G(\KK)$ such that $\phi_G(\lambda)= x$ and $\phi_G(\mu)= y$.  Since $P_y\subseteq P_x$, $\mu$ belongs to $Y_{P_x}(\KK)$, so by Remark~\ref{rem:common_tor} there exist a maximal split torus $T$ of $P_x$ and $u\in U_\lambda(k)$ such that $\lambda, u\cdot \mu\in Y_T(\KK)$.
 Since $P_{\kappa_i(u\cdot y)}(G')= P_{\widehat{\kappa}_i(u\cdot \mu)}(G')$ is contained in $P_{\kappa_i(x)}(G')= P_{\widehat{\kappa}_i(\lambda)}(G')$, it follows from the description of R-parabolic subgroups and R-Levi subgroups of reductive groups given in Example~\ref{ex:red_virtual_cochar} that $L_{\widehat{\kappa}_i(u\cdot \mu)}(G')$ is contained in $L_{\widehat{\kappa}_i(\lambda)}(G')$.  We deduce from Corollary~\ref{cor:par_intersect} that $L_{u\cdot \mu}$ is contained in $L_\lambda$, so $u\cdot y\in V_{L_{u\cdot \mu}}(\KK)\subseteq V_{L_\lambda}(\KK)$.  Now $U_{\widehat{\kappa}_i(\lambda)}(G')\subseteq U_{\widehat{\kappa}_i(\mu)}(G')$ because $P_{\widehat{\kappa}_i(\mu)}(G')\subseteq P_{\widehat{\kappa}_i(\lambda)}(G')$ and $G'$ is reductive.  It follows that $U_\lambda= U_{\widehat{\kappa}_i(\lambda)}(G')\cap G\subseteq U_{\widehat{\kappa}_i(\mu)}(G')\cap G= U_\mu$
 by Corollary~\ref{cor:par_intersect}, so $u\in U_\mu(k)$, so $u\cdot y= y$.  Hence $y\in V_{L_\lambda}(\KK)$, and we are done.
\end{proof}

\begin{lem}
\label{lem:local}
 Let $d$ be an admissible metric on $V_G(\KK)$.  Let $x\in V_G(\KK)$.  Then there is a $d$-open neighbourhood $O$ of $x$ such that for all $y\in O$, $x$ and $y$ lie in a common apartment.
\end{lem}

\begin{proof}
 We use induction on $\dim(G)$.  The statement is vacuous if $\dim(G)= 0$.  If $x\in V_{Z(G)^0}(\KK)$ then every $y\in V_G(\KK)$ lies in a common apartment with $x$, so $O= V_G(\KK)$ will do.  Suppose $x\not\in V_{Z(G)^0}(\KK)$.  Pick $\lambda\in Y_G(\KK)$ such that $\phi_G(\lambda)= x$.  We identify $V_{L_\lambda}(\KK)$ with its image in $V_G(\KK)$;
the restriction $d_1$ of $d$ to $V_{L_\lambda}(\KK)$ is an admissible metric.  By Lemma~\ref{lem:1/2_cts}, there is an open neighbourhood $O_1$ of $x$ such that $O_1\subseteq V_{L_\lambda}(\KK)$.  Since $\dim(L_\lambda)< \dim(G)$ by Lemma~\ref{lem:lives_in_centre}(ii), our induction hypothesis implies that there is an open neighbourhood $O_2$ of $x$ in $V_{L_\lambda}(\KK)$ such that for all $y\in O_2$, $x$ and $y$ lie in a common apartment of $V_{L_\lambda}(\KK)$.  But an apartment of $V_{L_\lambda}(\KK)$ is also an apartment of $V_G(\KK)$, so we can take $O= O_1\cap O_2$.
\end{proof}

\begin{lem}
\label{lem:cty}
 Let $T, T'$ be maximal split tori of $G$.  Let $(\lambda_n)$ and $(\lambda'_n)$ be sequences in $Y_T(\KK)$ and $Y_{T'}(\KK)$, respectively, and let $\lambda\in Y_T(\KK)$, $\lambda'\in Y_{T'}(\KK)$ such that $\lambda_n\to \lambda$ and $\lambda'_n\to \lambda'$.  Suppose $(T, \lambda_n)\approx (T', \lambda'_n)$ for all $n\in \NN$.  Then $(T, \lambda)\approx (T', \lambda')$.
\end{lem}

\begin{proof}
 For each $n$, there exists $g_n\in P_{\lambda_n}(k)$ such that $g_nTg_n^{-1}= T'$ and $g_n\cdot \lambda_n= \lambda'_n$.  The conjugation map $\Inn_{g_n}$ gives rise to a linear map $h_n\colon V_T(\KK)\to V_{T'}(\KK)$.  Since $W_k$ is finite, the set $\{h_n\mid n \in \NN\}$ is finite.  By passing to subsequences, we can assume that the $h_n$ are all equal --- say, to $h$.  So $h(\lambda)= \lambda'$ by the continuity of $h$.  We have $h(\lambda)= g_n\cdot \lambda$ for each $n\in \NN$.  Lemma~\ref{lem:1/2_cts} implies that $g_n\in P_\lambda(k)$ for large $n$.  The result follows.
\end{proof}

\begin{lem}
\label{lem:linmapextn}
 Let $\kappa$ be a linear map from $V_G(\QQ)$ to $V_H(\QQ)$.  Then there is a unique extension of $\kappa$ to a linear map $\ovl{\kappa}:V_G(\RR)\to V_H(\RR)$.
\end{lem}

\begin{proof}
 We can extend $\kappa$ uniquely on each apartment in the obvious way.  We just need to check that these maps paste together to give a well-defined map $\ovl{\kappa}$ from $V_G(\KK)$ to $V_H(\KK)$ (it is clear that the map is linear if it is well-defined, and uniqueness is also clear).  Recall from Lemma~\ref{lem:lin_map_lift} that $\kappa$ lifts to a map $\widehat{\kappa}\colon Y_G(\QQ)\to Y_H(\QQ)$.  Let $T,T'$ be maximal split tori of $G$ and choose maximal split tori $S$ and $S'$ of $H$ such that $\widehat{\kappa}$ maps $Y_T(\QQ)$ into $Y_S(\QQ)$ and $Y_{T'}(\QQ)$ into $Y_{S'}(\QQ)$.  Let $\widehat{\kappa}_{T,S}\colon V_T(\RR)\to V_S(\RR)$ (resp., $\widehat{\kappa}_{T',S'}\colon V_{T'}(\RR)\to V_{S'}(\RR)$) be the extension of $\widehat{\kappa}$.  Pick $\lambda\in Y_T(\RR)$ and $\lambda'\in Y_{T'}(\RR)$ such that $(\lambda, T)\approx (\lambda', T')$.  It is enough to show that $(S, \widehat{\kappa}_{T,S}(\lambda))\approx (S', \widehat{\kappa}_{T',S'}(\lambda'))$.
 
 Pick a rational approximating sequence $(\lambda_n)$ to $\lambda$ in $Y_T(\RR)$.  By hypothesis, there exists $g\in P_\lambda(k)$ such that $g\cdot \lambda= \lambda'$ and $gTg^{-1}= T'$.  Set $\lambda_n'= g\cdot \lambda_n$ for each $n\in \NN$.  Then $\lambda_n'\to \lambda'$, since conjugation by $g$ induces an isomorphism from $V_T(\RR)$ to $V_{T'}(\RR)$.  Further, $g\in P_\lambda(k)= P_{\lambda_n}(k)$, so $(T', \lambda_n')\approx (T, \lambda_n)$ and $\phi_G(\lambda_n')= \phi_G(\lambda_n)$ for each $n\in \NN$.  It follows that $\kappa(\phi_G(\lambda_n'))= \kappa(\phi_G(\lambda_n))$ and hence $(S', \widehat{\kappa}(\lambda_n'))\approx (S, \widehat{\kappa}(\lambda_n))$ for each $n\in \NN$.  But $\widehat{\kappa}(\lambda_n')\to \widehat{\kappa}_{T',S'}(\lambda')$ and $\widehat{\kappa}(\lambda_n)\to \widehat{\kappa}_{T,S}(\lambda)$ as $\widehat{\kappa}_{T',S'}$ and $\widehat{\kappa}_{T,S}$ are isomorphisms.  It follows from Lemma~\ref{lem:cty} that $(S', \widehat{\kappa}_{T',S'}(\lambda'))\approx (S, \widehat{\kappa}_{T,S}(\lambda))$, as required.
\end{proof}

\begin{defn}
\label{defn:strongly_type-preserving}
 Let $G$ and $G'$ be connected $k$-groups and let $\kappa\colon V_G(\KK)\to V_{G'}(\KK)$ be a linear map.  We say that $\kappa$ is \emph{strongly type-preserving} if for all $g\in G(k)$, there exists $g'\in G'(k)$ such that $\kappa(g\cdot x)= g'\cdot \kappa(x)$ for all $x\in V_G(\KK)$.
\end{defn}

\begin{rems}
\label{rems:strongly_type-preserving}
 (i). Let $G$ and $G'$ be connected $k$-groups and let $d'$ be an admissible metric on $V_{G'}(\KK)$.  Let $\kappa\colon V_G(\KK)\to V_{G'}(\KK)$ be a strongly type-preserving linear map.  Because $d'$ is $G'(k)$-invariant, we have
 $$ d'(\kappa(g\cdot x), \kappa(g\cdot y))= d'(\kappa(x), \kappa(y)) $$
for all $x,y\in V_G(\KK)$.

 (ii). It follows from Remark~\ref{rem:type-preserving} that if $f\colon G\to G'$ is a homomorphism of connected $k$-groups then $\kappa_f$ is strongly type-preserving.  Now let $G$ be a connected $k$-group and let $\alpha\colon k\to k'$ be a field homomorphism.  It follows from Lemma~\ref{lem:lin_map_fieldext} (applied to the homomorphism $f= \Inn_g$ for $g\in G(k)$) that $\nu_\alpha$ is a strongly type-preserving linear map; note that $^\alpha(\Inn_g)= \Inn_{^\alpha g}$, where we regard $g\in G(k)$ as a map ${\rm Spec}(k)\to G$ and obtain $^\alpha g\in\,\!^\alpha G(k')$ by base extension.
\end{rems}

We can now prove:

\begin{prop}
\label{prop:bdd_linear}
Let $G$ be a connected $k$-group and let $G'$ be a connected $k'$-group, where $k'$ is a field.  Let $\kappa\colon V_{G}(\KK)\to V_{G'}(\KK)$ be a strongly type-preserving linear map.  
Let $d$, $d'$ be admissible metrics on $V_{G}(\KK)$, $V_{G'}(\KK)$, respectively.  
Then there exists $C> 0$ such that for all $x_1,x_2\in V_{G}(\KK)$, we have $d'(\kappa(x_1), \kappa(x_2))\leq C \cdot d(x_1,x_2)$. 
\end{prop}

\begin{proof}
We may assume that $\KK=\RR$, for if the result holds in that case then it quickly follows for $\KK=\QQ$ by Lemma~\ref{lem:linmapextn} and Remark~\ref{rem:admprops}(iv).
Let $T$ be any maximal split torus of $G$ and let $x_1, x_2\in V_{T}(\KK)$.  
Then $\kappa(V_{T}(\KK))$ is contained in $V_{T'}(\KK)$ for some maximal split torus $T'$ of $G'$.  
The restriction of $d$ (resp., $d'$) to $V_{T}(\KK)$ (resp., $V_{T'}(\KK)$) is a metric arising from a symmetric positive-definite bilinear form, so there exists $D>0$ depending on $d,d'$ but independent of $x_1,x_2$ such that $d'(\kappa(x_1), \kappa(x_2))\leq D\cdot d(x_1,x_2)$.
Note that $D$ is
independent of the choice of $T$: for if $\widetilde{T}= gTg^{-1}$ is another maximal split torus of $G$ for some $g\in G(k)$ then, taking $g'$ as in Definition~\ref{defn:strongly_type-preserving}, we have
$$ d'(\kappa(g\cdot x_1), \kappa(g\cdot x_2))= d'(\kappa(x_1), \kappa(x_2))\leq D \cdot d(x_1, x_2)= D \cdot d(g\cdot x_1, g\cdot x_2) $$
for all $x_1, x_2\in V_T(\KK)$ by Remark~\ref{rems:strongly_type-preserving}(i).  Hence we are done when $x_1$ and $x_2$ belong to a common apartment of $V_{G}(\KK)$; in particular, this is the case when one of $x_1, x_2$ belongs to $V_{Z(G)^0}(\KK)$.  
Taking $x_1= x$ and $x_2= 0$ yields 
\begin{equation}
\label{eqn:bdd_norm}
 \|\kappa(x)\|\leq D\|x\|
\end{equation}
for all $x\in V_{G}(\KK)$.

Now we use induction on $\dim(G)$.  The result is trivial if $\dim(G)= 0$ or $G$ is itself a split torus.  
Suppose that $G$ is not a split torus and that $G$ contains a non-trivial central split torus $Z$; let $G_1 = G/Z$ as above. 
Then $\dim(G_1), \dim(Z)< \dim(G)$.  Let $d_1$ be as in Section~\ref{subsec:cent_tor} and let $d_Z$ be the restriction of $d$ to $V_Z(\KK)$.
By Lemma~\ref{lem:Z_qisom} we have a linear map $\kappa_1:= \kappa\circ \beta_Z^{-1}$ from $V_{G_1\times Z}(\KK)$ to $V_{G'}(\KK)$.  By Lemma~\ref{lem:Z_qisom},
we can replace $V_G(\KK)$ and $\kappa$ with $V_{G_1\times Z}(\KK)$ and $\kappa_1$.  But the result follows in this case by applying the induction hypothesis to $(V_{G_1}(\KK), d_1)$ and $(V_Z(\KK),d_Z)$ and the linear maps $V_{G_1}(\KK)\stackrel{\kappa_{j_1}}{\to} V_{G_1\times Z}(\KK)\stackrel{\kappa_1}{\to} V_{G'}(\KK)$ and $V_Z(\KK)\stackrel{\kappa_{j_Z}}{\to} V_{G_1\times Z}(\KK)\stackrel{\kappa_1}{\to} V_{G'}(\KK)$, where $j_1$ and $j_Z$ are the inclusions of $G_1$ and $Z$ in $G_1\times Z$, respectively.

So we may assume that $G$ has no non-central split torus and hence, by Lemma~\ref{lem:lives_in_centre}(ii), that $L_\lambda$ is a 
proper subgroup of $G$ for every nonzero $\lambda\in Y_G(\KK)$.
Suppose for a contradiction that no constant $C$ as in the statement of the lemma exists.  
Then there are sequences $(x_n)$ and $(y_n)$ in $V_{G}(\KK)$ such that $d'(\kappa(x_n),\kappa(y_n))> n \cdot d(x_n,y_n)$ for all $n\in \NN$.
We claim that by scaling and then successively replacing $(x_n)$ and $(y_n)$ with suitable subsequences and conjugates, we may 
assume that $(x_n)$ and $(y_n)$ are bounded sequences converging to a common limit $x$. 
First, we may assume (interchanging $x_n$ and $y_n$ at any given step if necessary), that $\|x_n\|\geq \|y_n\|$ for all $n\in\NN$.
By scaling, using Remark \ref{rem:admprops}(v),
we can further assume that $\|x_n\|=1$ and $\|y_n\|\leq 1$ for all $n\in \NN$.  
Now fix a maximal split torus $T$ of $G$.  
For each $n$, after conjugating both $x_n$ and $y_n$ by an appropriate element of $G(k)$, we can assume each $x_n$ belongs to $V_{T}(\KK)$; by Remark~\ref{rems:strongly_type-preserving}(i) we can do this without changing the value of $d'(\kappa(x_n), \kappa(y_n))$.  Now $(x_n)$ is a bounded sequence in the $d$-complete space $V_T(\KK)$, and hence some subsequence of $(x_n)$ $d$-converges to an element $x\in V_{T}(\KK)$ (recall that $\KK= \RR$);
replace $(x_n)$ by this subsequence and $(y_n)$ by the corresponding subsequence,
and note that $x\neq 0$ since $\|x_n\| = 1$ for all $n\in \NN$. 
Now $(\|\kappa(x_n)\|)$ and $(\|\kappa(y_n)\|)$ are bounded by \eqref{eqn:bdd_norm}, so $(d'(\kappa(x_n), \kappa(y_n)))$ is bounded.  
Since we have $d'(\kappa(x_n),\kappa(y_n))> n \cdot d(x_n,y_n)$, it follows that $d(x_n, y_n)\to 0$, so $y_n\to x$ also.
 
Choose $\lambda\in Y_{G}(\KK)$ such that $x= \phi_G(\lambda)$.  By hypothesis, we have $\dim(L_\lambda)< \dim(G)$. 
Let $i_\lambda\colon L_\lambda\to G$ be the inclusion map, let $\kappa_\lambda := \kappa_{i_\lambda}\colon V_{L_\lambda}(\KK)\to V_G(\KK)$ be the 
corresponding inclusion of vector edifices, and identify $V_{L_\lambda}(\KK)$ as a subset of $V_G(\KK)$ via $\kappa_\lambda$, as usual.
Let $O$ be an open neighbourhood of $x$ as in Lemma~\ref{lem:1/2_cts}, so that $O\subseteq V_{L_\lambda}(\KK)$.
Then for all $n$ sufficiently large we have $x_n, y_n\in O\subseteq V_{L_\lambda}(\KK)$.  
By our induction hypothesis applied to $L_\lambda$ and the map $\kappa\circ \kappa_\lambda\colon V_{L_\lambda}(\KK)\to V_{G'}(\KK)$, there exists some $C> 0$, independent of $n$, such that $d'(\kappa(x_n), \kappa(y_n))\leq C \cdot d(x_n, y_n)$, a contradiction.  
The result now follows by induction.
\end{proof}

\begin{cor}
\label{cor:quasi-isom}
 Let $G$ be a $k$-group and let $d_1, d_2$ be admissible metrics on $V_{G}(\KK)$.  Then $d_1$ and $d_2$ are bi-Lipschitz equivalent.  In particular, the topology on $V_{G}(\KK)$ does not depend on the choice of admissible metric.
\end{cor}

\begin{proof}
 This follows by applying Proposition~\ref{prop:bdd_linear} to the identity map from $V_{G}(\KK)$ to $V_{G}(\KK)$.
\end{proof}

We call the topology arising from any admissible metric the \emph{metric topology}.  Below when we refer to open and closed sets, we mean relative to the metric topology.

The next result is an immediate consequence of Proposition~\ref{prop:bdd_linear}.

\begin{cor}
\label{cor:lincts}
 Let $G$ and $G'$ be $k$-groups and let $\kappa\colon V_G(\KK)\to V_{G'}(\KK)$ be a strongly type-preserving linear map.  Let $d$ and $d'$ be admissible metrics on $V_G(\KK)$ and $V_{G'}(\KK)$, respectively.  Then $\kappa$ is continuous.
\end{cor}

\begin{cor}
\label{cor:closed}
 Let $H$ be a subgroup of $G$ and let $i\colon H\to G$ be inclusion.  Let $e$ and $d$ be admissible metrics on $V_H(\KK)$ and $V_G(\KK)$, respectively.  Then $\kappa_i$ is a continuous closed map.  In particular, $\kappa_i(V_H(\KK))$ is a closed subset of $V_G(\KK)$.
\end{cor}

\begin{proof}
 By Corollary~\ref{cor:quasi-isom} we can assume that $e= i^*(d)$.  If $(x_n)$ is any sequence in $V_H(\KK)$ then $(x_n)$ is Cauchy if and only if $(\kappa_i(x_n))$ is Cauchy.  The result follows.
\end{proof}

\begin{exmp}
 Let $G$ be reductive and let $P$ be a proper parabolic subgroup of $G$.  Fix admissible metrics $e$ on $V_P(\KK)$ and $d$ on $V_G(\KK)$.  Let $i\colon P\to G$ be inclusion.  Then $\kappa_i$ is a bijective linear map which is not an isomorphism of vector edifices (Example~\ref{ex:bijectivenotiso}), but it is a homeomorphism.  If $e= i^*(d)$ then $\kappa_i$ is even an isometry.
\end{exmp}

\begin{exmp}
 Let $k/k_0$ be a Galois field extension and suppose $G$ has a descent to a $k_0$-group $G_0$.  Recall from Example~\ref{ex:Galois_action} that the Galois group $\Gamma = \Gamma(k/k_0)$ acts on $V_G(\KK)$ by automorphisms of the form $\kappa_{\alpha, f_\alpha}$, and we may identify the fixed point set $V_G(\KK)^\Gamma$ with $V_{G_0}(\KK)$.  Now each $\kappa_{\alpha, f_\alpha}$ is strongly type-preserving as it is a composition of strongly type-preserving linear maps (Remark~\ref{rems:strongly_type-preserving}(ii)), so it is continuous by Corollary~\ref{cor:lincts}.  We deduce that $V_{G_0}(\KK)$ is a closed subset of $V_G(\KK)$.
\end{exmp}

We finish this subsection by extending Corollary~\ref{cor:closed} to the case of homomorphisms with finite kernel.

\begin{lem}
\label{lem:common_apt_isog}
 Let $f\colon G\to H$ be an isogeny of connected $k$-groups.  Let $x,y\in V_G(\KK)$.  Then $x$ and $y$ lie in a common apartment of $V_G(\KK)$ if and only if $\kappa_f(x)$ and $\kappa_f(y)$ lie in a common apartment of $V_H(\KK)$.
\end{lem}

\begin{proof}
 It is clear from the construction that if $x$ and $y$ lie in a common apartment of $V_G(\KK)$ then $\kappa_f(x)$ and $\kappa_f(y)$ lie in a common apartment of $V_H(\KK)$.  Conversely, suppose $\kappa_f(x)$ and $\kappa_f(y)$ lie in a common apartment of $V_H(\KK)$: say, $\kappa_f(x), \kappa_f(y)\in V_S(\KK)$, where $S$ is a maximal split torus of $H$.  Then $S\subseteq P_{\kappa_f(x)}(H)\cap P_{\kappa_f(y)}(H)$, so $S'\subseteq P_{\kappa_f(x)}(H)\cap P_{\kappa_f(y)}(H)$ for some maximal torus $S'$ of $H$ (Lemma~\ref{lem:common_tor_crit}).  Extending scalars to $\ovl{k}$ and applying Lemma~\ref{lem:par_levi_f_prop}(i) gives $S'_{\ovl{k}}\subseteq f_{\ovl{k}}((P_x)_{\ovl{k}})\cap f_{\ovl{k}}((P_y)_{\ovl{k}})$.  We have $f_{\ovl{k}}^{-1}((P_x)_{\ovl{k}})^0= (P_x)_{\ovl{k}}$ and $f_{\ovl{k}}^{-1}((P_y)_{\ovl{k}})^0= (P_y)_{\ovl{k}}$ by the proof of Lemma~\ref{lem:par_levi_f_prop}(ii), so $f_{\ovl{k}}^{-1}(S')^0\subseteq (P_x)_{\ovl{k}}\cap (P_y)_{\ovl{k}}$.  But $f_{\ovl{k}}^{-1}(S')^0$ contains a maximal torus of $G_{\ovl{k}}$ since $f_{\ovl{k}}$ is an isogeny.  The result follows from Lemma~\ref{lem:common_tor_crit}.
\end{proof}
 
\begin{prop}
\label{prop:finite_kernel}
 Let $f\colon G\to H$ be a homomorphism of connected $k$-groups with finite kernel.  Then $\kappa_f$ is a closed map.  In particular, $\kappa_f(V_G(\KK))$ is closed in $V_H(\KK)$.
\end{prop}

\begin{proof}
 We can factor $f$ as a homomorphism onto its image followed by a closed embedding, so we can assume without loss by Corollary~\ref{cor:closed} that $f$ is an isogeny.  Choose embeddings $i_i\colon G\to M_1$ and $i_2\colon H\to M_2$, where $M_1$ and $M_2$ are connected reductive $k$-groups.  Choose admissible metrics $d_1$ and $d_2$ on $V_{M_1}(\KK)$ and $V_{M_2}(\KK)$, respectively, and let $d$ be the product metric on $M_1\times M_2$.  Define $i\colon G\to M_1\times M_2$ by $g\mapsto (i_1(g), i_2(f(g))$.
 
 Consider the metrics $e:= i^*(d)$ and $e_2:= (i_2\circ f)^*(d_2)$ on $V_G(\KK)$.  Then $e$ (resp., $i_2^*(d_2)$) is an admissible metric on $V_G(\KK)$ (resp., $V_H(\KK)$) by construction, but we do not claim that $e_2$ is an admissible metric on $V_G(\KK)$.  Nonetheless $e$ and $e_2$ are $G(k)$-invariant and are given by a bilinear form on any fixed apartment.  It follows from an argument similar to the one in the proof of \eqref{eqn:perpqisom} that there exist $c,C> 0$ such that
 \begin{equation}
 \label{eqn:comparison}
  c \cdot e(x,y)\leq e_2(x,y)\leq C \cdot e(x,y)
 \end{equation}
 for all $x,y\in V_G(\KK)$ belonging to a common apartment.
 
 Let $(x_n)$ be an $e_2$-convergent sequence in $V_G(\KK)$ with limit $x$.  To complete the proof, it is enough to show that $(x_n)$ $e$-converges to $x$.  Let $O$ be an $i_2^*(d_2)$-open neighbourhood of $\kappa_f(x)$ as in Lemma~\ref{lem:local}.  Then $\kappa_f^{-1}(O)$ is an $e$-open neighbourhood of $x$ by Corollary~\ref{cor:lincts}.  By omitting finitely many terms at the start of the sequence, we can assume that $x_n\in \kappa_f^{-1}(O)$ for all $n\in \NN$.  For any $n\in \NN$, $\kappa_f(x_n)$ and $\kappa_f(x)$ lie in a common apartment of $V_H(\KK)$, so $x_n$ and $x$ lie in a common apartment of $V_G(\KK)$ by Lemma~\ref{lem:common_apt_isog}.  It follows from \eqref{eqn:comparison} that the sequence $(x_n)$ $e$-converges to $x$, since it $e_2$-converges to $x$.  This completes the proof.
\end{proof}

\begin{rem}
 Let $f\colon G\to H$ be a homomorphism of connected $k$-groups and suppose $k$ is perfect.  Then $\kappa_f(V_G(\KK))$ is closed in $V_H(\KK)$.  To see this, observe that $f$ factors as $G\stackrel{f_1}{\to} G/N_{\rm red}\stackrel{f_2}{\to} G/N\stackrel{i_3}{\to} H$, where $N:= \ker(f)$; note that $N_{\rm red}$ is $k$-defined as $k$ is perfect.  Then $\kappa_{f_1}$ is surjective by Lemma~\ref{lem:surj} and Remark~\ref{rem:isogenydiscussion}(iii), and $\kappa_{f_2}$ and $\kappa_{f_3}$ are closed maps by Proposition~\ref{prop:finite_kernel} and Corollary~\ref{cor:closed}, respectively, so $\kappa_f(V_G(\KK))$ is closed.
 
 We conjecture that the same result holds without the hypothesis that $k$ is perfect. 
\end{rem}

\subsection{The metric topology and the quotient topology}
Recall that $\varpi_G\colon \bigsqcup_T Y_T(\KK)\to Y_G(\KK)$ denotes the canonical projection and 
$\omega_G\colon \bigsqcup_T Y_T(\KK)\to V_G(\KK)$
is the composition $\phi_G\circ \varpi_G$.  
We may endow $V_G(\RR)$ with the quotient topology that it inherits via $\omega_G$, where we regard $\bigsqcup_T Y_T(\RR)$ as the disjoint topological union of the $Y_T(\RR)$ and each $Y_T(\RR)$ is topologised as a Euclidean space.
We observed in Remark \ref{rem:admprops}(vi) above that if $d$ is an admissible metric on $V_G(\RR)$, then $d$ induces on the subspaces $V_T(\RR)$ this Euclidean topology, because $d$ coincides with the metric on $V_T(\RR)$ coming from a positive-definite bilinear form.
Together with the fact that the restriction of $\omega_G$ gives a linear isomorphism $\omega_T$ from $Y_T(\RR)$ to $V_T(\RR)$,
this implies that the topology on $V_T(\RR)$ arising from $d$ coincides with that arising from the quotient topology. 
Below we show that the metric topology and the quotient topology on $V_G(\RR)$ actually coincide, for which we need a result of Whitehead (see \cite[Thm.~3.3.17]{engelking}).\footnote{The proof given in \emph{loc.\ cit.} is for $Z'= Z$ and $g= {\rm id}_Z$, but the generalisation in Lemma \ref{lem:prod_quot} is immediate.}

\begin{lem}
\label{lem:prod_quot}
 Let $X,Y,Z,Z'$ be topological spaces with $Z$ locally compact, let $f\colon X\to Y$ be a quotient map and let $g\colon Z\to Z'$ be a homeomorphism.  Then $f\times g\colon X\times Z\to Y\times Z'$ is a quotient map.
\end{lem}

\begin{prop}
\label{prop:quot_top_crit}
 The quotient topology on $V_G(\RR)$ coincides with the topology induced by any admissible metric $d$.  If $C\subseteq V_G(\RR)$ then $C$ is $d$-closed (resp., $d$-open) if and only if $C\cap V_T(\RR)$ is closed (resp., open) with respect to the subspace topology on $V_T(\RR)$ for every maximal split torus $T$ of $G$.
\end{prop}

\begin{proof}
 Let $C\subseteq V_G(\RR)$.  By definition of the quotient topology, $C$ is closed if and only if $\omega_G^{-1}(C)$ is closed in $\bigsqcup_T Y_T(\RR)$, and this is the case if and only if $\omega_G^{-1}(C)\cap Y_T(\RR)$ is closed for every maximal split torus $T$ of $G$.  But it follows from Lemma~\ref{lem:basicpropertiesRpars}(i) and Remark~\ref{rem:obvious_gen} that $\omega_T(\omega_G^{-1}(C)\cap Y_T(\RR))= C\cap V_T(\RR)$.  This shows that $C$ is closed in the quotient topology if and only if $C\cap V_T(\RR)$ is closed for every maximal split torus $T$ of $G$.  Likewise, $C$ is open in the quotient topology if and only if $C\cap V_T(\RR)$ is open for every maximal split torus $T$ of $G$.  Hence the second assertion of the proposition follows from the first.  It also follows that if a subset of $V_G(\RR)$ is metric-open (resp., metric-closed) then it is open (resp., closed) in the quotient topology.
 
 To prove the first assertion of the proposition, we use induction on $\dim(G)$.  The result holds trivially if $\dim(G)= 0$.  Suppose $G$ contains a non-trivial split central torus $Z$.  Let $G_1= G/Z$, and fix an admissible metric $d_1$ on $V_{G_1}(\RR)$.  We may identify $V_T(\RR)$ with $Y_T(\RR)$, and likewise for $Y_{T_1}(\RR)$ and $Y_Z(\RR)$, so we may regard each map $\beta_{Z,T}$ from Section~\ref{subsec:cent_tor} as a map from $Y_T(\RR)$ to $Y_{T_1}(\RR)\times Y_Z(\RR)$; these paste together to give a homeomorphism $\gamma_Z$ from $\bigsqcup_T Y_T(\RR)$ to $\left(\bigsqcup_{T_1} Y_{T_1}(\RR)\right)\times Y_Z(\RR)$, where $T$ runs over the maximal split tori of $G$ and $T_1$ runs over the maximal split tori of $G_1$.  We have a commutative diagram
   $$
 \xymatrixcolsep{5pc}
 \xymatrix{
 	\bigsqcup_T Y_T(\RR)\ar[d]_{\omega_{G}}\ar[r]^{\gamma_Z}& \left(\bigsqcup_{T_1} Y_{T_1}(\RR)\right)\times Y_Z(\RR)\ar[d]^{\omega_{G_1}\times \omega_Z}\\
 	V_G(\RR)\ar[r]^{\beta_Z}& V_{G_1}(\RR)\times V_Z(\RR)\\}
 $$
 where we give $V_G(\RR)$ and $V_{G_1}(\RR)\times V_Z(\RR)$ the quotient topology from $\omega_G$ and $\omega_{G_1}\times \omega_Z$, respectively.  Since $\gamma_Z$ is a homeomorphism and $\beta_Z$ is a bijection, we see that $\beta_Z$ must also be a homeomorphism.  Now $\dim(G_1), \dim(Z)< \dim(G)$, so the quotient topology on $V_{G_1}(\RR)$ from $\omega_{G_1}$ (resp., on $V_Z(\RR)$ from $\omega_Z$) coincides with the metric topology by our induction hypothesis.  Since $V_Z(\RR)$ is locally compact, Lemma~\ref{lem:prod_quot} applied to $\omega_{G_1}\times \omega_Z$ implies that the quotient topology on $V_{G_1}(\RR)\times V_Z(\RR)$ is the product of the quotient topologies on the factors.  Hence the quotient topology on $V_{G_1}(\RR)\times V_Z(\RR)$ is the product of the metric topologies on the factors, which coincides with the topology coming from the product metric $\widetilde{d}$.  But $\beta_Z$ is a bi-Lipschitz map from $(V_G(\RR), d)$ to $(V_{G_1}(\RR)\times V_Z(\RR), \widetilde{d})$ by Lemma~\ref{lem:Z_qisom}.  It follows that the quotient topology and the metric topology on $V_G(\RR)$ coincide.
 
 So assume $G$ does not contain a non-trivial split central torus.  If $C\subseteq V_G(\RR)$ is $d$-open then $C\cap V_T(\RR)$ is open for every maximal split torus $T$ of $G$, so $C$ is open in the quotient topology.  Conversely, suppose $C\subseteq V_G(\RR)$ is open in the quotient topology.  Let $x\in C$.  It is enough to show that $B_d(x,\epsilon)\subseteq C$ for some $\epsilon> 0$.  By Lemma~\ref{lem:lives_in_centre}(ii), Lemma~\ref{lem:1/2_cts} and our hypothesis on $G$, there exist $\delta> 0$ and a proper R-Levi subgroup $L$ of $P_x$ such that $B_d(x,\delta)\subseteq V_L(\RR)$.  Set $C_0= C\cap B_d(x,\delta)$.  The first paragraph of the proof implies that $C_0$ is an open neighbourhood of $x$ with respect to the quotient topology on $V_G(\RR)$.  Likewise, $C_0$ is an open neighbourhood of $x$ with respect to the quotient topology on $V_L(\RR)$.  But $\dim(L)< \dim(G)$ and the restriction $d_0$ of $d$ is an admissible metric on $V_L(\RR)$, so by our induction hypothesis there exists $\epsilon> 0$ such that $\epsilon< \delta$ and $B_{d_0}(x,\epsilon)\subseteq C_0$.  Finally, observe that $B_d(x,\epsilon)= B_{d_0}(x,\epsilon)$ since $B_d(x,\delta)\subseteq V_L(\RR)$.  The result now follows by induction.
\end{proof}

\begin{rem}
 Proposition~\ref{prop:quot_top_crit} generalises \cite[Lem.\ 2.7]{BMR:strong}.
\end{rem}

\subsection{Admissible metrics on the spherical edifice}
Let $d$ be an admissible metric on $V_G(\KK)$.  Recall that if $\KK= \QQ$ then we have a unique extension of $d$ to an admissible metric on $V_G(\RR)$.  We may identify $\Delta_G(\RR)$ with the unit sphere in $V_G(\RR)$ (hence the terminology ``spherical edifice''), and $\Delta_G(\QQ)$ with a subset of the unit sphere in $V_G(\RR)$: explicitly, we can take $\Delta_G(\QQ)$ to be the set of all $x\in V_G(\RR)$ such that $\|x\|_d= 1$ and the ray $\RR^+\cdot x$ contains a $\QQ$-point.  We give $\Delta_G(\KK)$ the metric $d^\flat$ it inherits as a subspace of $V_G(\RR)$, where the induced map $d^\flat$
is defined in Section \ref{sec:linearmaps}.\footnote{For spherical buildings it is traditional to use the spherical metric on each apartment rather than the metric inherited from the ambient Euclidean space, but we cannot do this for $\Delta_G(\KK)$ for arbitrary $G$ because the common apartment property can fail.}  We call any such metric on $\Delta_G(\KK)$ an \emph{admissible metric}.  It follows from Corollary~\ref{cor:quasi-isom} that any two admissible metrics on $\Delta_G(\KK)$ are bi-Lipschitz equivalent, so they define the same topology.  Given an admissible metric $d$ on $\Delta_G(\KK)$, we define $\Isom _d(\Delta_G(\KK))$ to be the subgroup of $\Aut (\Delta_G(\KK))$ consisting of isometries.

Let $d$ be an admissible metric on $V_G(\KK)$.  Clearly, if $\kappa\in \Isom _d(V_G(\KK))$ then we have $\kappa^\flat\in \Isom _{d^\flat}(\Delta_G(\KK))$.  The converse, however, is false.  For example, if $f_a$ is as in Example~\ref{ex:not_isom} then $(f_a)^\flat= {\rm id}_{\Delta_G(\KK)}$ is an isometry but $f_a$ is not.

\section{The Tits Centre Conjecture and geometric invariant theory}
\label{sec:TCCGIT}

We finish by discussing some motivation for our constructions, elaborating on the summary in Section~\ref{sec:intro}.  In this section we assume $G$ is reductive, although some of the ideas below make sense for arbitrary $G$.  For more details, see \cite{BMR:strong} or \cite{BMR:typeA}.

Let $\Delta$ be a spherical building.  We call a subcomplex $\Sigma$ of $\Delta$ \emph{completely reducible} (\emph{cr} for short) if every simplex of $\Sigma$ has an opposite in $\Sigma$.  Let $\Gamma$ be a subgroup of $\Aut (\Delta)$ that stabilises $\Sigma$.  We say that a simplex $\sigma$ is a \emph{simplicial $\Gamma$-centre\footnote{Usually in the literature this is referred to as a $\Gamma$-centre.  We have added the adjective ``simplicial'' to distinguish these from the $\Gamma$-centres introduced in Definition~\ref{defn:geom_centre}.} of $\Sigma$} if $\sigma$ is is nonempty and $\sigma$ is fixed by $\Gamma$.

The following is the Tits Centre Conjecture (TCC), which was proved in a series of papers \cite{muhlherrtits}, \cite{lrc}, \cite{rc}.

\begin{conj}[Tits Centre Conjecture]
 Assume $\Delta$ is thick, let $\Sigma$ be a convex non-cr subcomplex of $\Delta$ and let $\Gamma$ be a subgroup of $\Aut (\Delta)$ that stabilises $\Sigma$.  Then $\Sigma$ has a simplicial $\Gamma$-centre.
\end{conj}

We are concerned with the special case when $\Delta= \Delta_G$, where $G$ is reductive.  We mention in passing a connection with Serre's theory of $G$-complete reducibility.  We say that a subgroup $H$ of $G$ is \emph{$G$-completely reducible over $k$} ($G$-cr over $k$) if whenever $H$ is contained in a parabolic subgroup $P$ of $G$, $H$ is contained in some Levi subgroup $L$ of $P$.  Let $\Sigma$ be the fixed point set $(\Delta_G)^H$.  It is not hard to show that $\Sigma$ is a convex subcomplex of $\Delta_G$ and that $H$ is $G$-cr if and only if $\Sigma$ is cr.  So if $H$ is not $G$-cr then $\Sigma$ is a convex non-cr subcomplex of $\Delta_G$, and one can apply the Tits Centre Conjecture. This yields results on $G$-complete reducibility.
See \cite{sphericalcochars}, \cite{BMR}, \cite{BMR:tits},  \cite{BMR:semisimplification}, \cite{GIT}, \cite{BMR:sepext}, \cite{serre2} for more details.

We say that a simplicial $\Gamma$-centre of $\Sigma$ is \emph{unopposed} if it has no opposite in $\Sigma$.  The usual formulations of the TCC don't touch on this.  Note, however, that the proof of the TCC in type $A$ in \cite{muhlherrtits} automatically produces a $\Gamma$-centre that is unopposed.  We return to this idea shortly.

For reasons related to geometric invariant theory, we want an analogue to the Tits Centre Conjecture but working with points in $\Delta_G(\KK)$ rather than with simplices.  We have notions of convexity and opposition of points (see Section~\ref{sec:common_apt}).  We say that a closed convex subset $\Sigma$ of $\Delta_G(\KK)$ is \emph{completely reducible} (\emph{cr}) if every $x\in \Sigma$ has an opposite in $\Sigma$.

\begin{defn}
\label{defn:geom_centre}
  Let $\Gamma$ be a subgroup of $\Aut (\Delta_G(\KK))$ that stabilises $\Sigma$.  We say that $x\in \Sigma$ is a \emph{$\Gamma$-centre of $\Sigma$} if $x\neq 0$ and $\sigma_x$ is fixed by $\Gamma$.  We say that a $\Gamma$-centre $x$ of $\Sigma$ is \emph{unopposed} if $x$ has no opposite in~$\Sigma$.
\end{defn}

\begin{rems}
\label{rem:not_fixed}
 (i). We do not insist in Definition~\ref{defn:geom_centre} that $x$ is fixed by $\Gamma$.  It can be shown, however, that if $\Gamma\subseteq \Isom_d(V_G(\KK))$ for some admissible metric $d$ and if $\Sigma$ admits a $\Gamma$-centre then $\Sigma$ admits a $\Gamma$-centre that is fixed by $\Gamma$.  See \cite{BMR:typeA}.
 
 (ii). A simplicial $\Gamma$-centre always corresponds to a proper R-parabolic subgroup of $G$.  If $Y_{Z(G)^0}\neq 0$, however, then there can exist a $\Gamma$-centre $x$ such that $P_x= G$.
\end{rems}

\begin{conj}
\label{conj:sTCC_spherical}
 Let $\Sigma$ be a closed convex non-cr subset of $\Delta_G(\KK)$ and let $\Gamma$ be a subgroup of $\Aut(\Delta_G(\KK))$ that stabilises $\Sigma$.  Then $\Sigma$ has an unopposed $\Gamma$-centre.
\end{conj}

We call Conjecture~\ref{conj:sTCC_spherical} the \emph{strong Tits Centre Conjecture} (sTCC).  It is weaker than the version given in \cite[Conj.\ 2.10]{BMR:strong}, which asserts the existence of a $\Gamma$-fixed point in $\Sigma$; this turns out to be false if we don't insist that $\Gamma$ acts by isometries (cf.\ Remark~\ref{rem:not_fixed}).  For a discussion of these and related matters, we refer the reader yet again to \cite{BMR:typeA} and to \cite{BMR:strong}.

Any convex subcomplex of $\Delta_G(\KK)$ is closed \cite[Lem.\ 2.7]{BMR:strong}, so there is clearly --- as the terminology suggests --- a close link between the TCC and the sTCC.  To make this concrete, however, one needs to understand the relationship between automorphisms of $\Delta_G$ and automorphisms of $\Delta_G(\KK)$.  This is not completely straightforward.  Note, for example, that there can exist automorphisms of $\Delta_G(\KK)$ which arise from very natural operations on $G$, but which are not isometries.  For more details, see \cite{BMR:typeA}.

It is convenient to work with the vector edifice rather than the spherical edifice.  We define a \emph{cone} in $V_G(\KK)$ to be a subset that is stable under multiplication by $\KK^+$.  We say that a closed convex cone $C$ in $V_G(\KK)$ is \emph{completely reducible} (\emph{cr}) if every $x\in C$ has an opposite in $C$.  Here is our formulation of the sTCC for vector edifices.

\begin{conj}
\label{conj:sTCC_vector}
 Let $C$ be a closed convex non-cr cone in $V_G(\KK)$ and let $\Gamma$ be a subgroup of $\Aut(V_G(\KK))$ that stabilises $C$.  Then $C$ has an unopposed $\Gamma$-centre.
\end{conj}

Conjectures~\ref{conj:sTCC_vector} and \ref{conj:sTCC_spherical} are closely related: one can show that if $\Sigma$ is a closed convex non-cr subset of $\Delta_G(\KK)$ then $\zeta_G^{-1}(\Sigma)\cup \{0\}$ is a closed convex non-cr cone in $V_G(\KK)$, and conversely if $C$ is a closed convex non-cr cone in $V_G(\KK)$ then $\zeta_G(C\backslash \{0\})$ is a closed convex non-cr subset of $\Delta_G(\KK)$ (see \cite[Lem.\ 2.6]{BMR:strong}).

Now we can describe the relationship to GIT.  We keep our assumption that $G$ is reductive.  Suppose $G$ acts on an affine variety $X$.  Let $x\in X(k)$ and let $\lambda\in Y_G$.  We say that $\lambda$ \emph{destabilises $x$} if $\lim_{a\to 0} \lambda(a)\cdot x$ exists.  We say that $\lambda$ \emph{properly destabilises $x$ over $k$} if $x':= \lim_{a\to 0} \lambda(a)\cdot x$ lies outside the orbit $G(k)\cdot x$.  We say that the orbit $G(k)\cdot x$ is \emph{cocharacter-closed over $k$} if there does not exist any $\lambda$ in $Y_G$ such that $\lambda$ properly destabilises $x$ over $k$, see \cite[Sec.\ 3]{GIT} and \cite[Sec.\ 3]{cochars}.

We define a subset $\Lambda_x$ of $V_G$ and a subset $D_x$ of $V_G(\QQ)$ by
$$ \Lambda_x := \{\phi_G(\lambda)\mid \lambda\in Y_G, \mbox{$\lambda$ destabilises $x$}\} $$
and
$$ D_x := \{c\lambda\mid \lambda\in \Lambda_x, c\in \QQ^+\}. $$
We call $D_x$ the \emph{destabilising locus} for $x$.  In \cite{GIT} it is shown that $D_x$ is a closed convex cone in $V_G(\QQ)$ and $D_x$ is stabilised by the building automorphisms arising from $G_x(k)$ (by $G_x$ we mean the scheme-theoretic stabiliser of $x$).  Moreover, $D_x$ is non-cr if and only if there exists $\lambda\in Y_G$ such that $\lambda$ properly destabilises $x$ over $k$.  Note, however, that $D_x$ is not in general a subcomplex of $V_G(\QQ)$.

Fix an admissible metric $d$ on $V_G(\KK)$; this amounts to fixing a ``length function'' in the sense of \cite[Sec.\ 2]{kempf}.  Kempf proved the following theorem \cite[Thm.\ 3.4, Cor.\ 3.5]{kempf} (Hesselink \cite{He} and Rousseau \cite{rousseau} found closely related results).

\begin{thm}
\label{thm:kempf}
 Suppose $k$ is algebraically closed.  
 Let $G$, $X$ and $x$ be as above, and suppose $G\cdot x$ is not closed. Then there exists $\lambda_\opt\in Y_G$ such that $\lambda_\opt$ properly destabilises $x$ over $k$ and $G_x(k)\subseteq P_\lambda(k)$.
\end{thm}

The cocharacter $\lambda_\opt$ is often referred to as the \emph{optimal destabilising cocharacter}; it arises by optimising a certain real-valued function on $D_x$.  Note that $\lambda_\opt$ can depend on the choice of $d$.

Now we come to our key insight: in the language above, $\phi_G(\lambda_\opt)$ is an unopposed $\Gamma$-centre of $D_x$ for a certain subgroup $\Gamma$ of $\Isom_d(V_G(\QQ))$.\footnote{For more discussion of $\Gamma$, see \cite[Sec.\ 6.2]{BMR:strong}; note that $\Gamma$ contains all the building automorphisms arising from inner automorphisms of $G(k)$.}  This suggests two complementary paths.

(a). Use methods from GIT, including optimality, to prove cases of the sTCC, at least when $k= \ovl{k}$.  For some steps in this direction, see \cite[Sec.\ 5]{BMR:strong}.

(b). Use known cases of the sTCC to prove the existence of optimal destabilising cocharacters.  There is no known generalisation of Theorem~\ref{thm:kempf} to arbitrary fields, although there are some partial results \cite[Thm.\ 4.2]{kempf}, \cite[Thm.\ 5.2]{He}, \cite[Sec.\ 6]{BMR:strong}, \cite[Thm.\ 4.7]{GIT}, \cite[Thm.\ 4.3]{cochars}.  For further discussion, see \cite[Sec.\ 1]{GIT}.  Recall that $D_x$ need not be a subcomplex of $V_G(\QQ)$, so we need the full force of the sTCC to deduce the existence of a $\Gamma$-centre.  On the other hand, if $D_x$ does happen to be a subcomplex then we can apply the TCC; for one such result in the context of $G$-complete reducibility, see \cite[Thm.\ 1.1]{BMR:sepext}.

As the above discussion makes clear, we are concerned with finding unopposed $\Gamma$-centres.  We finish by sketching an approach for finding an unopposed $\Gamma$-centre, given a $\Gamma$-centre.  Let $C$ be a closed convex non-cr cone in $V_G(\KK)$ and let $\Gamma\leq \Aut (V_G(\KK))$ such that $\Gamma$ stabilises $C$.  Let $x\in C$ be a $\Gamma$-centre.  If $x$ has no opposite in $C$ then we are done.  Otherwise there exists $\lambda\in Y_G(\KK)$ such that $x= \zeta_G(\lambda)$ and $\zeta_G(-\lambda)\in C$.  Set $L= L_\lambda$.
We can regard $V_L(\KK)$ as a subset of $V_G(\KK)$ via the inclusion $i\colon L\to G$.  Recall from Section~\ref{sec:projection} that we have a map $F_{P,L}:= \kappa_\pi\circ \kappa_i^{-1}$ from $V_G(\KK)$ to $V_L(\KK)$.  The strategy is to show that $\Gamma$ gives rise to a group $\Gamma'$ of automorphisms of $V_L(\KK)$ and that $F_{P,L}(C)$ is a closed convex non-cr cone in $V_L(\KK)$ stabilised by $\Gamma'$.  Then one can look for a $\Gamma'$-centre $x'$ of $F_{P,L}(C)$, and show that $\kappa_i(x')$ is a $\Gamma$-centre of $C$.  Repeating this process, one hopes eventually to find an unopposed $\Gamma$-centre of $C$.  We will explore this idea in more detail in \cite{BMR:typeA}.


\bigskip
{\bf Acknowledgements}:
We are grateful to Bernhard M\"uhlherr for his encouragement and for helpful conversations. We thank the editors of this special volume in honour of Jacques Tits for inviting us to contribute, and for their forbearance during the manuscript's slow gestation.  The second author was supported by a VIP grant from the Ruhr-Universit\"at Bochum.  Some of this work was completed during visits to the Mathematisches Forschungsinstitut Oberwolfach: we thank them for their support.

We are also indebted to the referees for their careful reading of the paper and for many suggestions making various arguments more transparent.

\end{document}